\documentclass[12pt, reqno]{amsart}
\usepackage{amsmath,amssymb,amsfonts}
\usepackage[all]{xypic}
\usepackage[pdftex]{graphicx}
\usepackage{stmaryrd}
\usepackage{color}
\usepackage{tikz-cd}
\usepackage{hyperref}
\usepackage{pifont}
\usepackage{booktabs}
\usepackage{multirow}
\usepackage{float}

\theoremstyle{plain}
\newtheorem{theorem}{Theorem}[section]

\newtheorem{lemma}[theorem]{Lemma}
\newtheorem{corollary}[theorem]{Corollary}
\newtheorem{prop}[theorem]{Proposition}

\newtheorem{algorithm}[theorem]{Algorithm}

\theoremstyle{remark}

\newtheorem{remark}[theorem]{Remark}

\newtheorem{example}[theorem]{Example}

\newtheorem*{note*}{Note}
\newtheorem*{remark*}{Remark}
\newtheorem*{example*}{Example}

\theoremstyle{definition}
\newtheorem*{definition*}{Definition}
\newtheorem*{hypothesis*}{Hypothesis}
\newtheorem*{assumptions*}{Assumptions}
\newtheorem{definition}[theorem]{Definition}

\newcommand{\Z}{\mathbb{Z}}
\newcommand{\R}{\mathbb{R}}
\newcommand{\Q}{\mathbb{Q}}

\newcommand{\F}{\mathbb{F}}
\newcommand{\Aut}{\mathrm{Aut}}
\newcommand{\Cl}{\mathrm{Cl}}
\newcommand{\GL}{\mathrm{GL}}
\newcommand{\rad}{\mathrm{rad}}
\newcommand{\im}{\mathrm{im}}
\newcommand{\coker}{\mathrm{coker}}
\newcommand{\nr}{\mathrm{nr}}

\newcommand{\st}{\mathrm{st}}
\newcommand{\iso}{\mathrm{iso}}
\newcommand{\LF}{\mathrm{LF}}

\renewcommand{\labelenumi}{(\roman{enumi})}
\numberwithin{equation}{section}

\newcommand{\calO}{\mathcal{O}}
\newcommand{\calM}{\mathcal{M}}

\newcommand{\fra}{\mathfrak{a}}

\newcommand{\frf}{\mathfrak{f}}
\newcommand{\frg}{\mathfrak{g}}
\newcommand{\frh}{\mathfrak{h}}

\newcommand{\frp}{\mathfrak{p}}
\newcommand{\frP}{\mathfrak{P}}
\newcommand{\frq}{\mathfrak{q}}
\newcommand{\frQ}{\mathfrak{Q}}

\newcommand{\bnr}{\overline{\mathrm{nr}}}
\newcommand{\lra}{\longrightarrow}
\newcommand{\Mat}{\mathrm{Mat}}
\newcommand{\mpar}[1]{}
\newcommand{\OC}{\calO_C}

\newcommand{\OL}{\calO_L}
\newcommand{\sseq}{\subseteq}

\title[Determination of the stably free cancellation property for orders]{Determination of the stably free \\
cancellation property for orders}

\author{Werner Bley}
\address{
Ludwig-Maximilians-Universität München\\
Theresienstr.\ 39\\
D-80333 M\"unchen\\
Germany}
\email{bley@math.lmu.de}
\urladdr{https://www.mathematik.uni-muenchen.de/$\sim$bley/}

\author{Tommy Hofmann}
\address{
Naturwissenschaftlich-Technische Fakult\"at\\
Universit\"at Siegen\\
Walter-Flex-Straße 3\\
57068 Siegen\\
Germany}
\email{tommy.hofmann@uni-siegen.de}

\author{Henri Johnston}
\address{
Department of Mathematics\\
University of Exeter\\
Exeter\\
EX4 4QE\\
United Kingdom
}
\email{H.Johnston@exeter.ac.uk}
\urladdr{https://mathematics.exeter.ac.uk/people/profile/index.php?username=hj241}

\date{Version of 8th September 2025}
\begin{document}

\begin{abstract}
Let $K$ be a number field, let $A$ be a finite-dimensional semisimple $K$-algebra,
and let $\Lambda$ be an $\mathcal{O}_{K}$-order in $A$.
We give practical algorithms that determine whether $\Lambda$ 
has stably free cancellation (SFC). 
As an application, we determine all finite groups $G$ of order at most $383$
such that the integral group ring $\Z[G]$ has SFC.
\end{abstract}

\maketitle

\section{Introduction}

A ring $R$ is said to have \emph{stably free cancellation} (SFC)
if every finitely generated stably free (left) $R$-module is in fact free, that is, 
if $M$ is a finitely generated (left) $R$-module such that $M \oplus R^{\oplus n} \cong 
R^{\oplus (n+m)}$ for some $m,n \geq 0$, then $M \cong R^{\oplus m}$. 
In this article, we develop practical algorithms 
to determine whether a given order in a finite-dimensional semisimple algebra
over a number field has SFC.

Let $K$ be a number field with ring of integers $\mathcal{O}=\mathcal{O}_{K}$.
Let $A$ be a finite-dimensional semisimple $K$-algebra and let
$\Lambda$ be an $\mathcal{O}$-order in $A$.
For example, if $G$ is a finite group then the group ring $\mathcal{O}[G]$
is an $\mathcal{O}$-order in the group algebra $K[G]$. 
In addition to stably free cancellation for $\Lambda$, 
one can also consider \textit{locally free cancellation} (LFC),
whose definition we defer to \S \ref{subsec:SFCandLFC}.
If $\Lambda$ has LFC then it also has SFC. 
Note, however, that there exist orders with SFC but without LFC 
(see \cite{MR3403458} or \cite{MR4105795}).

We say that $A$ satisfies the \textit{Eichler condition} if no simple component of $A$
is a totally definite quaternion algebra (see \S \ref{subsec:totdefEichler}).
Jacobinski \cite{MR251063} showed that if $A$ satisfies the Eichler condition 
then $\Lambda$ has LFC and thus SFC. 
Smertnig--Voight \cite{MR4105795} showed that when 
$A$ is a totally definite quaternion algebra, the questions of whether
$\Lambda$ has SFC or LFC can be decided algorithmically.
They then used this to classify all orders in totally definite quaternion algebras that
have SFC, and of these, all those that have LFC. 
This classification builds on work of Vign\'eras \cite{MR429841}, Hallouin--Maire \cite{MR2244802} 
and Smertnig \cite{MR3403458}.

We can always decompose $A$ into a direct product of $K$-algebras 
$A= \prod_{i=1}^{k} A_{i}$, where $A_{1}$ satisfies the Eichler 
condition and $A_{2}, \ldots, A_{k}$ are all totally definite quaternion algebras. 
If $\Lambda$ decomposes as a product of $\mathcal{O}$-orders
$\Lambda= \prod_{i=1}^{k} \Lambda_{i}$
such that each $\Lambda_{i}$ is contained in $A_{i}$, then the questions of whether 
$\Lambda$ has SFC or LFC can be decided using the results discussed in the previous
paragraph, since both properties respect direct products. 
However, in many cases of interest, $\Lambda$ does not exhibit such a decomposition.
The desire to consider such cases motivates the following result, 
which is the main result of this article.

\begin{theorem}
There exists an algorithm that, given an arbitrary 
$\mathcal{O}$-order $\Lambda$ in a finite-dimensional semisimple $K$-algebra, 
decides whether $\Lambda$ has SFC.
\end{theorem}

In fact, we present three algorithms, all of which are practical in certain situations.
The first of these, Algorithm~\ref{alg:SFC-naive}, is only practical when $A$ is of `small' dimension. 
This relies on Theorem~\ref{thm:SFC-criterion} and
works by computing a maximal order $\mathcal{M}$ containing $\Lambda$, a suitable two-sided ideal $\mathfrak{f}$ of $\mathcal{M}$ contained in $\Lambda$, and the image of
$(\Lambda/\mathfrak{f})^{\times}$ in a certain ray class group, 
and then using this data to determine whether a finite number of explicit 
stably free `test lattices' are in fact free over $\Lambda$.
The main ideas employed utilize and build upon the results of \cite{bley-boltje} and
\cite{MR4493243}.
The second, Algorithm~\ref{alg:SFC-fail}, cannot prove that $\Lambda$ has SFC, but
can be used to show that it fails SFC, and in this situation is often much faster than Algorithm~\ref{alg:SFC-naive}. This works by sampling `random' test lattices
and checking if they fail to be free over $\Lambda$.
The third, Algorithm~\ref{alg:sfc-fiberproduct}, uses fiber products to reduce the problem 
for the original order to that for an order in an algebra of smaller dimension as well as the
computation of a certain subgroup of the unit group of a finite ring.
It uses Corollary~\ref{cor:Eichler-W-consequences}, which itself
crucially relies on results of Reiner--Ullom \cite{MR349828} (see Proposition \ref{prop:properties-of-Mu}) and Swan \cite{MR584612} (see Theorem~\ref{thm:Swan-Eichler}). 
These algorithms require several other algorithms of independent interest, 
including those that can be used to compute certain primary decompositions and
unit groups of finite rings (see \S \ref{sec:primary-decomposition}), and 
new freeness tests for certain modules over orders (see \S \ref{sec:hybrid-method}). 

Orders of particular interest are integral group rings $\Z[G]$ for $G$ a finite group.
For instance, the property of whether $\Z[G]$ has SFC or LFC has applications
in topology (see \cite{MR2012779}, \cite{MR4302168}, \cite{MR4788715}, for example) 
and in the determination of whether a finite Galois extension of $\Q$ has a normal 
integral basis (see \cite{MR1313877}, \cite{MR1827291}).

Let $G$ be a finite group and let 
$\mathbb{H}$ denote the quaternion division algebra over $\R$.
If no simple component of $\R[G]$ is isomorphic to $\mathbb{H}$,
then $\Q[G]$ satisfies the Eichler condition and so $\Z[G]$ has LFC (and thus SFC)
by the aforementioned result of Jacobinski. 
This condition on $G$ is equivalent to the requirement that $G$ has no quotient that 
is isomorphic to a \textit{binary polyhedral group}, that is, a generalized quaternion
group $Q_{4n}$ for $n \geq 2$,
or one of $\tilde{T}$, $\tilde{O}$ or $\tilde{I}$, the binary 
tetrahedral, octahedral and icosahedral groups.
In \S \ref{sec:review-canc-int-grp-rings}, 
we review important work of Swan \cite{MR703486}, Chen \cite{ChenPhD}, and Nicholson \cite{MR4299591} on SFC and LFC for $\Z[G]$ in the case that $G$ does have
a quotient isomorphic to a binary polyhedral group.

Before stating our own results on SFC for integral group rings, we first establish some further notation.
Let $C_{n}$ denote the cyclic group of order $n$.
Let $G_{(n,m)}$ denote the $m$th group of order $n$ in the
GAP \cite{GAP4} small groups library \cite{MR1935567}.

A special case of a result due to Swan and Fr\"ohlich says that 
if $\Z[G]$ has SFC then $\Z[G/N]$ has SFC for every quotient $G/N$,
or equivalently, if $\Z[G/N]$ fails SFC for some choice of $N$, then $\Z[G]$ also fails SFC
(see Theorem~\ref{thm:cancellation-under-map} and Lemma~\ref{lem:quot-group-ring}).
The following result is Theorem~\ref{thm:new-group-rings-fail-SFC},
in which the claim that each $\Z[G]$ fails SFC was proven using Algorithm~\ref{alg:SFC-fail};
the claim that $\Z[G/N]$ has SFC for each proper quotient $G/N$ 
can be interpreted as saying that each $G$ is minimal with respect to quotients.

\begin{theorem}\label{intro-thm:new-group-rings-fail-SFC}
Let $G$ be one of the following finite groups:
\begin{enumerate}
\item $Q_{4n} \times C_{2}$ for $2 \leq n \leq 5$,
\item $\tilde{T} \times C_{2}^{2}$, $\tilde{O} \times C_{2}$, $\tilde{I} \times C_{2}^{2}$, or
\item $G_{(32,14)}$, $G_{(36,7)}$, $G_{(64,14)}$, $G_{(100,7)}$. 
\end{enumerate}
Then $\Z[G]$ fails SFC but $\Z[G/N]$ has SFC for every proper quotient $G/N$.  
\end{theorem}

Note that Swan \cite{MR703486} already proved that $\Z[Q_{8} \times C_{2}]$ fails SFC and Chen \cite{ChenPhD} already proved that $\Z[G]$ fails SFC for
$G=Q_{12} \times C_{2}$, $Q_{16} \times C_{2}$, $Q_{20} \times C_{2}$, $G_{(36,7)}$, $G_{(100,7)}$. Since Swan \cite{MR703486} proved that $\Z[Q_{4n}]$ fails SFC
for $n \geq 6$, we obtain the following, which is the 
most general result to date on integral group rings that fail SFC.

\begin{corollary}\label{intro-cor:group-ring-fail-SFC}
Let $G$ be a finite group with a quotient isomorphic to a group listed in 
Theorem~\ref{intro-thm:new-group-rings-fail-SFC}
or to $Q_{4n}$ for some $n \geq 6$. 
Then $\Z[G]$ fails SFC. 
\end{corollary}

Using Algorithm~\ref{alg:sfc-fiberproduct} with judicious choices of fiber products, 
we prove the following result on new cases of integral group rings with SFC (this
is Theorem~\ref{thm:new-groups-with-SFC}). 

\begin{theorem}\label{intro-thm:new-groups-with-SFC}
Let $G$ be one of the following finite groups:
\[
\tilde{T} \times C_{2}, \, 
\tilde{T} \times Q_{12}, \, 
\tilde{T} \times Q_{20}, \,
Q_{8} \rtimes \tilde{T}, \,
Q_{8} \rtimes Q_{12}, \,
\tilde{I} \times C_{2}, \, 
G_{(192, 183)}, \,
G_{(384, 580)}, \,
G_{(480, 962)}.
\]
Then $\Z[G]$ has SFC.
\end{theorem}

By combining Corollary~\ref{intro-cor:group-ring-fail-SFC} and 
Theorem~\ref{intro-thm:new-groups-with-SFC}, and using the 
GAP \cite{GAP4} small groups library \cite{MR1935567}, we obtain the following result
(see Corollary~\ref{cor:classification-383}).

\begin{theorem}\label{intro-thm:classification-383}
Let $G$ be a finite group with $|G| \leq 383$.  
Then $\Z[G]$ has SFC if and only if $G$ has no quotient of the form
\begin{enumerate}
\item $Q_{4n} \times C_{2}$ for $2 \leq n \leq 5$,
\item $\tilde{T} \times C_{2}^{2}$, $\tilde{O} \times C_{2}$, $\tilde{I} \times C_{2}^{2}$,
\item $G_{(32,14)}$, $G_{(36,7)}$, $G_{(64,14)}$, $G_{(100,7)}$, or
\item $Q_{4n}$ for $6 \leq n \leq 95$.
\end{enumerate} 
\end{theorem}

This represents a significant advance in the state-of-the-art, since before the results of this article, 
such a classification was only known for groups $G$ with $|G| \leq 31$
(for example, it was not known whether $\Z[G_{(32,14)}]$ had SFC).
In fact, Theorem~\ref{intro-thm:classification-383} is a special case of 
Theorem~\ref{thm:part-class-1023}, which gives a classification for $|G| \leq 1023$, but
with $49$ explicitly listed gaps. In other words, of the $11759892$ groups $G$ with 
$|G| \leq 1023$, there are now only $49$ for which it is not known whether $\Z[G]$ has SFC.

Nicholson \cite[Theorem 1.7]{nicholson2024cancellation}
has now shown that $\Z[\tilde{T} \times C_{2}]$ in fact has LFC, 
using our result that it has SFC 
(i.e.\ part of Theorem~\ref{intro-thm:new-groups-with-SFC}) as an input.
By combining this with further new results on integral group rings with LFC, and
using Theorem~\ref{intro-thm:new-group-rings-fail-SFC} as a crucial input, he has given 
classifications of all integral group rings $\Z[G]$ with LFC and SFC, provided that 
either (i) $G$ has no quotient isomorphic to $\tilde{T}$, $\tilde{O}$ or $\tilde{I}$, or 
(ii) $G$ does have a quotient isomorphic to $C_{2} \times C_{2}$
(see \cite[Theorems A and B]{nicholson2024cancellation}).
We remark that in the case of integral group rings,
the LFC property discussed here is equivalent to the projective cancellation (PC) property
discussed in loc.\ cit.

\subsection*{Organization}
This article is organized as follows. 
In \S \ref{sec:lattices-and-orders}, we give preliminaries on lattices and orders.
We review cancellation properties for orders in \S \ref{sec:review-canc-orders}, and in \S \ref{sec:review-canc-int-grp-rings}, we specialize to the case 
of integral group rings.
In \S \ref{sec:fiber-products}, we give reduction criteria for SFC for orders using
fiber products, and in \S \ref{sec:app-to-int-group-rings} we consider 
applications to integral group rings.
Next, in \S \ref{sec:critsfc}, we give criteria for freeness and stable freeness
of lattices over orders.
We develop the three algorithms discussed above for determining whether an order has SFC in \S \ref{sec:alg-SFC}, \S \ref{sec:withoutSFC} and \S \ref{sec:det-SFC-fiber-prods}. In \S \ref{sec:primary-decomposition}, we give algorithms to compute certain primary decompositions and
unit groups of finite rings, and in \S \ref{sec:hybrid-method} we give
new freeness tests for certain modules over orders.
Our main results on SFC 
for integral group rings, which include Theorem~\ref{intro-thm:classification-383}
as a special case, are proven in \S \ref{sec:further-SFC-for-int-group-rings}.  
Finally, in Appendix \ref{appendix}, we discuss the implementation of the algorithms
of this article, and give relevant data and timings.

\subsection*{Acknowledgments}
The authors are grateful to John Nicholson for helpful comments and discussions, 
including suggestions for which integral group rings to investigate initially. 
They would also like to thank Nigel Byott for several corrections and suggestions, 
as well as an anonymous referee for suggesting a number of improvements.
For the purpose of open access, the authors have applied a Creative Commons Attribution (CC BY) license to any author accepted manuscript version arising.

\subsection*{Funding}
Hofmann gratefully acknowledges support by the Deutsche Forschungsgemeinschaft —
Project-ID 286237555 – TRR 195; and Project-ID 539387714.

\subsection*{Notation and conventions}
All rings are assumed to be associative and unital. 
All modules are assumed to be left modules unless otherwise stated. 
For notational convenience, we shall henceforth 
abuse notation by using the symbol $\oplus$
to denote direct products of orders and algebras. 
Denote the center of a ring $R$ by $Z(R)$. 
For a division algebra $D$ and a positive integer $k$, let $\Mat_{k}(D)$
denote the ring of all $k \times k$ matrices with entries in $D$.
We call $\Mat_{k}(D)$ a \emph{full matrix algebra} over $D$. 

\section{Preliminaries on lattices and orders}\label{sec:lattices-and-orders}

For further background on lattices and orders, we refer the reader to \cite{curtisandreiner_vol1,Reiner2003}.

\subsection{Lattices and orders}\label{subsec:lattices-and-orders}
Let $R$ be a Dedekind domain with field of fractions $K$. 
An \textit{$R$-lattice} is a finitely generated torsion-free $R$-module, or equivalently,
a finitely generated projective $R$-module.
For any finite-dimensional $K$-vector space $V$, an \textit{$R$-lattice in $V$} is a finitely generated $R$-submodule $M$ of  $V$.
We say that $M$ is a \textit{full} $R$-lattice in $V$ if $K M=V$. 
We may identify $K \otimes_{R} M$ with $KM$ via the map $k \otimes m \mapsto km$.

Now let $A$ be a finite-dimensional semisimple $K$-algebra and let
$\Lambda$ be an \textit{$R$-order} in $A$, that is, a subring of $A$ that is an $R$-lattice such that $K\Lambda = A$. A \textit{$\Lambda$-lattice} is a left $\Lambda$-module that is also an $R$-lattice.

\subsection{Localizations and completions}\label{subsec:localizations-and-completions}
Henceforth assume that $K$ is a number field. 

Let $\mathfrak{p}$ range over all maximal ideals of $R$. 
Let $R_{\mathfrak{p}}$ denote the localization of $R$ at $\mathfrak{p}$, 
that is, the ring of $\mathfrak{p}$-integral elements in $K$.  
For an $R$-lattice $M$, define the localization $M_{\mathfrak{p}}$ to be the $R_{\mathfrak{p}}$-lattice 
$R_{\mathfrak{p}} \otimes_{R} M$. 
Localising a $\Lambda$-lattice $X$ at $\mathfrak{p}$ yields a 
$\Lambda_{\mathfrak{p}}$-lattice $X_{\mathfrak{p}}$.
Two $\Lambda$-lattices $X$ and $Y$ are said to be \textit{locally isomorphic}, or
\textit{in the same genus}, 
if there is an isomorphism 
of $\Lambda_{\mathfrak{p}}$-lattices $X_{\mathfrak{p}} \cong Y_{\mathfrak{p}}$ for every $\mathfrak{p}$,
in which case we write $X \vee Y$.

Let $\widehat{R}_{\mathfrak{p}}$ 
denote the $\mathfrak{p}$-adic completion of $R$.
For an $R$-lattice $M$, define the completion $\widehat{M}_{\mathfrak{p}}$ to be the 
$\widehat{R}_{\mathfrak{p}}$-lattice $\widehat{R}_{\mathfrak{p}} \otimes_{R} M$. 
Completing a $\Lambda$-lattice $X$ at $\mathfrak{p}$ yields a 
$\widehat{\Lambda}_{\mathfrak{p}}$-lattice $\widehat{X}_{\mathfrak{p}}$.
Given  $\Lambda$-lattices $X$ and $Y$,
we have that $X_{\mathfrak{p}} \cong Y_{\mathfrak{p}}$ as $\Lambda_{\mathfrak{p}}$-lattices
if and only if $\widehat{X}_{\mathfrak{p}} \cong \widehat{Y}_{\mathfrak{p}}$ as 
$\widehat{\Lambda}_{\mathfrak{p}}$-lattices
(see \cite[(30.17)]{curtisandreiner_vol1} or \cite[(18.2)]{Reiner2003}).
Therefore $X \vee Y$ if and only if 
there is an isomorphism 
of $\widehat{\Lambda}_{\mathfrak{p}}$-lattices $\widehat{X}_{\mathfrak{p}} \cong \widehat{Y}_{\mathfrak{p}}$ for every $\mathfrak{p}$.

\subsection{Locally free lattices and the principal genus}\label{subsec:lf-pg}
A $\Lambda$-lattice $X$ is said to be \textit{locally free of rank $n$} if
$X \vee \Lambda^{\oplus n}$.
Every locally free $\Lambda$-lattice is projective over $\Lambda$ by 
\cite[(8.19)]{curtisandreiner_vol1}. 

Define the \textit{principal genus} $g(\Lambda)$ to be the class 
of locally free $\Lambda$-lattices of rank $1$. 
By \cite[(31.7)]{curtisandreiner_vol1}, given $Y_{1}, Y_{2}, Y_{3} \in g(\Lambda)$, there
exists $Y_{4} \in g(\Lambda)$ such that 
\begin{equation}\label{eq:principal-genus-eq}
Y_{1} \oplus Y_{2} \cong Y_{3} \oplus Y_{4}. 
\end{equation}
Moreover, by \cite[(31.14)]{curtisandreiner_vol1}, given $Y \vee \Lambda^{\oplus n}$,
there exists $X \in g(\Lambda)$ such that 
\begin{equation}\label{eq:genus-reduce-to-rank-1}
Y \cong \Lambda^{\oplus (n-1)} \oplus X. 
\end{equation}
For $X \in g(\Lambda)$ we write $[X]_{\iso}$ for the isomorphism class of $X$
and define 
\[
\LF_{1}(\Lambda) = \{ [X]_{\iso} : X \in g(\Lambda) \}.
\]
Note that $\LF_{1}(\Lambda)$ is finite by virtue of the 
Jordan--Zassenhaus theorem \cite[(26.4)]{Reiner2003}.

\subsection{Stable isomorphism}\label{subsec:stable-iso}
Let $P(\Lambda)$ denote the class of finitely generated projective $\Lambda$-modules.
For $X, Y \in P(\Lambda)$, we say that $X$ and $Y$ are \textit{stably isomorphic} if
\[
X \oplus \Lambda^{\oplus m} \cong Y \oplus \Lambda^{\oplus m} \textrm{ for some } m \geq 0.
\]
A module $X \in P(\Lambda)$ is said to be \textit{stably free of rank $n$} if
it is stably isomorphic to $\Lambda^{\oplus n}$.
For $X \in P(\Lambda)$, let $[X]_{\st}$ denote the stable isomorphism class of $X$.
Note that by \cite[(6.15)]{curtisandreiner_vol1} and the discussion in \S \ref{subsec:localizations-and-completions}, stably isomorphic $\Lambda$-lattices are locally isomorphic.
In particular, stably free lattices are locally free.

\subsection{Locally free class groups}\label{subsec:lfcg}
We give an overview of the \textit{locally free class group of $\Lambda$},
for which there are several equivalent definitions (see \cite[\S 49A]{curtisandreiner_vol2}).
Here we follow \cite[(49.10)]{curtisandreiner_vol2}, to which we refer the reader for further details.
Let
\[
\Cl(\Lambda) = \{ [X]_{\st} : X \in g(\Lambda) \}.
\]
By \eqref{eq:principal-genus-eq}, given $X_{1}, X_{2} \in g(\Lambda)$ 
there exists $X_{3} \in g(\Lambda)$ such that $X_{1} \oplus X_{2} \cong \Lambda \oplus X_{3}$.
It is straightforward to check that the stable isomorphism class 
$[X_{3}]_{\st}$ is uniquely determined by the classes $[X_{1}]_{\st}$ and $[X_{2}]_{\st}$. 
We define a binary operation on $\Cl(\Lambda)$ by
\[
[X_{1}]_{\st} + [X_{2}]_{\st} = [X_{3}]_{\st} 
\textrm{ whenever } 
X_{1} \oplus X_{2} \cong \Lambda \oplus X_{3},
\]
where $X_{1}, X_{2}, X_{3} \in g(\Lambda)$. 
This is a well-defined commutative and associative operation, with identity element 
$[\Lambda]_{\st}$. 
Moreover, inverses exist since by \eqref{eq:principal-genus-eq}, 
given $X \in g(\Lambda)$ there exists $X' \in g(\Lambda)$ such that $X \oplus X' \cong \Lambda^{\oplus 2}$.
Therefore $\Cl(\Lambda)$ is an abelian group. 
Moreover, $\Cl(\Lambda)$ is finite since
$\LF_{1}(\Lambda)$ is finite and the canonical map $\LF_{1}(\Lambda) \rightarrow \Cl(\Lambda)$
defined by $[X]_{\iso} \mapsto [X]_{\st}$ is surjective.

\section{Review of cancellation properties for orders}\label{sec:review-canc-orders}

Let $R$ be a Dedekind domain whose quotient field $K$ is a number field.
Henceforth every $R$-order is assumed to be in a finite-dimensional semisimple $K$-algebra.

\subsection{Stably free and locally free cancellation}\label{subsec:SFCandLFC}
Let $\Lambda$ be an $R$-order.
Then $\Lambda$ is said to have \textit{stably free cancellation} (SFC) 
if every finitely generated stably free $\Lambda$-module is in fact free, that is, 
if for every finitely generated $\Lambda$-module $X$ we have
\begin{equation}\label{eq:SFC}
X \oplus \Lambda^{\oplus m} \cong \Lambda^{\oplus (m+n)} \textrm{ for some } m,n \geq 0 
\implies X \cong \Lambda^{\oplus n}.
\end{equation}
Similarly, $\Lambda$ is said to have \textit{locally free cancellation} (LFC) if for every pair of 
locally free $\Lambda$-lattices $X$ and $Y$ we have
\begin{equation}\label{eq:LFC}
X \oplus \Lambda^{\oplus m} \cong Y \oplus \Lambda^{\oplus m} \textrm{ for some } m \geq 0 \implies X \cong Y. 
\end{equation}

Since locally free $\Lambda$-lattices are projective (see \S \ref{subsec:lf-pg}), 
one can easily deduce that $\Lambda$ has LFC if and only if it has the apparently stronger property 
that justifies the name: for every triple of locally free $\Lambda$-lattices $X,Y,Z$ we have 
\[
X \oplus Z \cong Y \oplus Z \implies X \cong Y.
\]

If $X$ is a finitely generated $\Lambda$-module such that  
$X \oplus \Lambda^{\oplus m}
\cong \Lambda^{\oplus (m+n)} \textrm{ for some } m,n \geq 0$
then it is straightforward to see that $X$ is a $\Lambda$-lattice; 
since $X$ is stably free it is also locally free (see \S \ref{subsec:stable-iso}). 
Therefore if $\Lambda$ has LFC then it also has SFC.
Note, however, that there exist orders with SFC but without LFC 
(see \cite{MR3403458} or \cite{MR4105795}).

By \eqref{eq:genus-reduce-to-rank-1}, $\Lambda$ has SFC if and only if \eqref{eq:SFC} holds 
for all $X \in g(\Lambda)$. 
Similarly, $\Lambda$ has LFC if and only if \eqref{eq:LFC} holds for all $X,Y \in g(\Lambda)$.
Hence the canonical surjective map $\LF_{1}(\Lambda) \rightarrow \Cl(\Lambda)$
defined by $[X]_{\iso} \mapsto [X]_{\st}$ is bijective precisely when $\Lambda$ has LFC.
Moreover, $\Lambda$ has SFC precisely when the fiber above $0$ has exactly one element,
that is, when every stably free $\Lambda$-lattice of rank $1$
is in fact free (of rank $1$).

\begin{theorem}[Fr\"ohlich, Swan]\label{thm:cancellation-under-map}
Let $\Lambda \rightarrow \Gamma$ be a map of $R$-orders 
inducing a surjection $K\Lambda \rightarrow K\Gamma$ of $K$-algebras.
If $\Lambda$ has SFC (resp.\ LFC) then $\Gamma$ has SFC (resp.\ LFC).
\end{theorem}

\begin{proof}
The claim for LFC was proven by Fr\"ohlich \cite[VIII]{MR376619}, whose approach was subsequently refined by Swan. Here we adapt the proof of \cite[Theorem 1.1]{MR4299591}.

By \cite[Theorem A10]{MR703486}, the diagram
\[
\begin{tikzcd}
\LF_{1}(\Lambda) \arrow[r] \arrow[d, twoheadrightarrow] & \LF_{1}(\Gamma)  \arrow[d,twoheadrightarrow] \\
\Cl(\Lambda) \arrow[r,twoheadrightarrow] & \Cl(\Gamma)
\end{tikzcd}
\]
is a weak pull back in the sense that a compatible pair of elements  in $\Cl(\Lambda)$ and 
$\LF_{1}(\Gamma)$ has a not-necessarily-unique lift to $\LF_{1}(\Lambda)$.
It follows that the fiber in $\LF_{1}(\Lambda)$ over an element $x \in \Cl(\Lambda)$ maps onto 
the fiber in $\LF_{1}(\Gamma)$ over the image of $x$ in $\Cl(\Gamma)$.
Thus if $\Lambda$ has LFC then $\Gamma$ has LFC.
Taking $x=0$ shows that if $\Lambda$ has SFC then $\Gamma$ has SFC.
\end{proof}

\begin{corollary}\label{cor:cancellation-over-order}
Let $\Lambda \subseteq \Gamma$ be $R$-orders such that $K\Lambda = K\Gamma$.
If $\Lambda$ has SFC (resp.\ LFC) then $\Gamma$ has SFC (resp.\ LFC). 
\end{corollary}

\begin{corollary}\label{cor:cancellation-direct-product}
Let $\Lambda_{1}$ and $\Lambda_{2}$ be $R$-orders
and let $\Lambda = \Lambda_{1} \oplus \Lambda_{2}$.
Then $\Lambda$ has SFC (resp.\ LFC) if and only if both $\Lambda_{1}$ and $\Lambda_{2}$ have 
SFC (resp.\ LFC).
\end{corollary}

\begin{proof}
One direction follows from Theorem~\ref{thm:cancellation-under-map}; 
the other follows from the definitions.
\end{proof}

\subsection{Totally definite quaternion algebras and the Eichler condition}\label{subsec:totdefEichler}
Let $\mathbb{H}$ denote the quaternion division algebra over $\R$.
A \emph{totally definite quaternion algebra} is a simple finite-dimensional $K$-algebra $B$ 
whose completion at every archimedean prime of the center of $B$ is isomorphic to 
$\mathbb{H}$ (in particular, this forces $K$ to be totally real).

We say that a finite-dimensional semisimple $K$-algebra 
$A$ satisfies the \emph{Eichler condition} if no simple component of $A$
is a totally definite quaternion algebra.
A proof of the following result can be found in \cite[(51.24)]{curtisandreiner_vol2} or \cite[\S 4]{MR251063}.

\begin{theorem}[Jacobinski]\label{thm:Jacobinski-cancellation}
If $A$ is a finite-dimensional semisimple $K$-algebra satisfying the Eichler condition then
every $R$-order in $A$ has LFC. 
\end{theorem}

For the rest of this section, we specialize to the case of orders over the ring 
of integers $\mathcal{O}=\mathcal{O}_{K}$ of a number field $K$.
In \cite{MR4105795}, it was shown that the questions of whether a given order in a
totally definite quaternion algebra has SFC or LFC can be tackled algorithmically.

\begin{theorem}[Smertnig--Voight]\label{thm:Smertnig-Voight-alg} 
There exists an algorithm that given an $\mathcal{O}$-order $\Lambda$ in a totally definite quaternion algebra over $K$, determines whether $\Lambda$ has SFC (resp.\ LFC).
\end{theorem}

This result was used to classify all orders inside totally definite quaternion algebras that have SFC, 
and of these, all that have LFC (see \cite[Theorem 1.3]{MR4105795}).
This classification restricted to SFC and the case of maximal or hereditary orders had
previously been given by Vign\'eras \cite{MR429841}, Hallouin--Maire \cite{MR2244802} 
and Smertnig \cite{MR3403458}.

\begin{corollary}\label{cor:sfc-alg-maxord}
There exists an algorithm that given a maximal $\mathcal{O}$-order 
$\mathcal{M}$ 
in a finite-dimensional 
semisimple $K$-algebra, determines whether $\mathcal{M}$ has SFC (resp.\ LFC).
\end{corollary}

\begin{proof}
We can write $\mathcal{M} = \mathcal{M}_{1} \oplus \cdots \oplus \mathcal{M}_{r}$,
where each $\mathcal{M}_{i}$ is a maximal $\mathcal{O}$-order contained in a simple $K$-algebra $A_{i}$,
and each $A_{i}$ either satisfies the Eichler condition or is a totally definite quaternion algebra.
Therefore the desired result now follows from Corollary~\ref{cor:cancellation-direct-product},
Theorem~\ref{thm:Jacobinski-cancellation} and Theorem~\ref{thm:Smertnig-Voight-alg}.
\end{proof}

\section{Review of cancellation properties for integral group rings}\label{sec:review-canc-int-grp-rings}

We recall results concerning SFC and LFC for integral group rings $\Z[G]$, where
$G$ is a finite group.

\begin{lemma}\label{lem:quot-group-ring}
Let $G$ be a finite group and let $H$ be a quotient of $G$.
If $\Z[G]$ has SFC (resp.\ LFC) then $\Z[H]$ has SFC (resp.\ LFC). 
\end{lemma}

\begin{proof}
The canonical map $\Z[G] \rightarrow \Z[H]$ of $\Z$-orders
induces a surjection $\Q[G] \rightarrow \Q[H]$ of $\Q$-algebras, 
and so the desired result is a consequence of Theorem~\ref{thm:cancellation-under-map}.
\end{proof}

The finite non-cyclic subgroups of $\mathbb{H}^{\times}$ are called the \emph{binary polyhedral groups}.
These are the generalized quaternion groups $Q_{4n}$ for $n \geq 2$ or one of $\tilde{T}$, $\tilde{O}$ or $\tilde{I}$, the binary tetrahedral, octahedral and icosahedral groups (see \cite[\S 51A]{curtisandreiner_vol2} and the references therein).
It is well known that a finite group $G$ has a binary polyhedral quotient precisely when 
the group algebra $\Q[G]$ does not satisfy the Eichler condition
(see \cite[Proposition~1.5]{MR4299591} or \cite[(51.3)]{curtisandreiner_vol2}, for example).

We now recall two important theorems of Swan \cite[Theorems I and II]{MR703486}.

\begin{theorem}[Swan]\label{thm:bp-canc-group-rings}
Let $G$ be a binary polyhedral group.
Then $\Z[G]$ has SFC if and only if $\Z[G]$ has LFC if and only if $G$ is one of the following 
$7$ groups:
\[
Q_{8}, \, Q_{12}, \,  Q_{16}, \, Q_{20}, \, \tilde{T}, \, \tilde{O}, \, \tilde{I}.
\]
\end{theorem}

\begin{corollary}[Swan]\label{cor:ZG-fails-SFC-when-large-Q4n-quotient}
Let $G$ be a finite group with a quotient isomorphic to $Q_{4n}$ for some $n \geq 6$. 
Then $\Z[G]$ fails SFC. 
\end{corollary}
\begin{proof}
This follows from one direction of Theorem~\ref{thm:bp-canc-group-rings} combined with
Lemma~\ref{lem:quot-group-ring}.
\end{proof}

\begin{theorem}[Swan]\label{thm:bp-LFC-max-order-group-rings}
Let $G$ be a binary polyhedral group and let $\mathcal{M}$ be a maximal $\Z$-order 
in $\Q[G]$.
Then $\mathcal{M}$ has LFC if and only if $G$ is one of the following $11$ groups:
\begin{equation}\label{eq:list-bp-max}
Q_{8}, \, Q_{12}, \, Q_{16}, \, Q_{20}, \, Q_{24}, \, Q_{28}, \, Q_{36}, \, Q_{60}, \, 
\tilde{T}, \, \tilde{O}, \, \tilde{I}. 
\end{equation}
\end{theorem}

\begin{remark}
In order, the groups listed in \eqref{eq:list-bp-max} are isomorphic to 
\begin{align*}
& G_{(8,4)}, \, G_{(12,1)}, \, G_{(16,9)}, \, G_{(20,1)}, \, G_{(24,4)}, \, G_{(28,1)}, \,
G_{(36,1)}, \, G_{(60,3)}, \\ 
& G_{(24,3)} \cong \mathrm{SL}_{2}(\F_{3}), \quad
G_{(48,28)} \cong \mathrm{CSU}_{2}(\F_{3}), \quad 
G_{(120,5)} \cong \mathrm{SL}_{2}(\F_{5}). 
\end{align*}
\end{remark}

\begin{corollary}\label{cor:LFC-max-order-group-rings}
Let $G$ be a finite group and let $\mathcal{M}$ be a maximal $\Z$-order in $\Q[G]$. 
Then $\mathcal{M}$ has LFC if and only if $G$ has no binary polyhedral quotient other than those
in \eqref{eq:list-bp-max}.
\end{corollary}

\begin{proof}
Suppose that $G$ has a binary polyhedral quotient $H$ different to those listed in \eqref{eq:list-bp-max}.
Then the image of $\mathcal{M}$ under the canonical projection map $\Q[G] \rightarrow \Q[H]$ is
a maximal $\Z$-order, and thus does not have LFC by Theorem~\ref{thm:bp-LFC-max-order-group-rings}.
Hence $\mathcal{M}$ does not have LFC by Theorem~\ref{thm:cancellation-under-map}.

Suppose that $G$ has no binary polyhedral quotient other than those listed in \eqref{eq:list-bp-max}.
We can write $\mathcal{M} = \mathcal{M}_{1} \oplus \cdots \oplus \mathcal{M}_{r}$,
where each $\mathcal{M}_{i}$ is a maximal $R$-order contained in a simple $K$-algebra $A_{i}$,
and each $A_{i}$ either satisfies the Eichler condition or is a totally definite quaternion algebra.
By Corollary~\ref{cor:cancellation-direct-product}, it suffices to show that each $\mathcal{M}_{i}$
has LFC.
Fix $i$ with $1 \leq i \leq r$.
If $A_{i}$ satisfies the Eichler condition then $\mathcal{M}_{i}$ has LFC by Theorem 
~\ref{thm:Jacobinski-cancellation}.
So suppose that $A_{i}$ is a totally definite quaternion algebra.
Let $\psi_{i} |_{G} : G \rightarrow \mathbb{H}^{\times}$ denote the group homomorphism obtained by 
restricting the composition $\psi_{i} : \Q[G] \twoheadrightarrow A_{i} \hookrightarrow \mathbb{H}$.
Let $H_{i}$ and $N_{i}$ denote the image and kernel of $\psi_{i} |_{G}$, respectively. 
Then the elements of $H_{i}$ span $\mathbb{H}$ over $\R$ and so $H_{i}$ must be non-abelian
since $\mathbb{H}$ is a noncommutative ring. 
Thus $H_{i}$ is a binary polyhedral quotient of $G$ and so by hypothesis $H_{i}$ is 
isomorphic to one of the groups listed in \eqref{eq:list-bp-max}.
Let $e_{N_{i}} = |N_{i}|^{-1} \sum_{n \in N_{i}} n$ 
denote the central idempotent of $\Q[G]$ attached to $N_{i}$.
Then
\[
\Q[G] = \Q[G] e_{N_{i}} \oplus \Q[G] (1 - e_{N_{i}})
\]
and since $\psi_{i} |_{ \Q[G] (1 - e_{N_{i}})} = 0$ it follows that $A_i$ is a direct summand of
$\Q[G] e_{N_i} \cong \Q[H_{i}]$. Therefore $\mathcal{M}_{i}$
is a  direct summand of a maximal $\Z$-order $\mathcal{M}_{i}'$ containing $\Z[H_{i}]$.
Then $\mathcal{M}_{i}'$ has LFC by Theorem~\ref{thm:bp-LFC-max-order-group-rings}
and thus $\mathcal{M}_{i}$ also has LFC by Corollary~\ref{cor:cancellation-direct-product}.
\end{proof}

The following two results are \cite[Theorems 13.7 and 15.5]{MR703486}. 
Though the latter result is stated in loc.\ cit.\ for LFC only, it is
immediate from the proof that it also holds for SFC; the authors are grateful
to John Nicholson for bringing this to their attention.

\begin{theorem}[Swan]\label{thm:TnxIm}
The group ring $\Z[\tilde{T}^{m} \times \tilde{I}^{n}]$ has LFC for all $m,n \geq 0$. 
\end{theorem}

\begin{theorem}[Swan]\label{thm:direct-products}
Let $G$ and $H$ be finite groups such that $\Z[G]$ has SFC (resp.\ LFC), 
$\Q[H]$ satisfies the Eichler condition, and $H$ has no subgroup of index $2$.
Then $\Z[G \times H]$ has SFC (resp. LFC).
\end{theorem}

The following results on the failure of SFC are \cite[Theorem 15.1]{MR703486}
and a summary of the main results of Chen \cite{ChenPhD}. 
Though the latter result is stated for LFC in loc.\ cit., a close examination of the proofs shows that 
it also holds for SFC. 

\begin{theorem}[Swan]\label{thm:Q8xC2}
The group ring $\Z[Q_{8} \times C_{2}]$ fails SFC. 
\end{theorem}

\begin{theorem}[Chen]\label{thm:Chen}
Let $G$ be one of the following $5$ groups:
\[
Q_{12} \times C_{2}, \quad
Q_{16} \times C_{2}, \quad
Q_{20} \times C_{2}, \quad
G_{(36,7)}, \quad
G_{(100,7)}.
\]
Then $\Z[G]$ fails SFC.
\end{theorem}

We now recall several results of Nicholson \cite[Theorems A, B, C]{MR4299591}.
Define the \emph{$\mathbb{H}$-multiplicity} 
$m_{\mathbb{H}}(G)$ of a finite group $G$ to be the number of 
copies of $\mathbb{H}$ in the Wedderburn decomposition of $\R[G]$, that is, the number of irreducible
one-dimensional quaternionic representations. 
Thus $\Q[G]$ satisfies the Eichler condition if and only if $m_{\mathbb{H}}(G)=0$.
Recall that for a ring $R$, the \emph{Whitehead group of $R$} is defined to be
$K_{1}(R)=\GL(R)^{\mathrm{ab}}$, where $\GL(R) = \varinjlim \GL_{n}(R)$ and the 
direct limit is taken with respect to inclusions $\GL_{n}(R) \hookrightarrow \GL_{n+1}(R)$
(see \cite[\S 40A]{curtisandreiner_vol2}). 
We say that $K_{1}(R)$ is \emph{represented by units} if the canonical map
$R^{\times} = \GL_{1}(R) \hookrightarrow \GL(R) \rightarrow K_{1}(R)$ is surjective.

\begin{theorem}[Nicholson]\label{thm:Nicholson-A}
Let $G$ be a finite group with quotient $H$ such that $m_{\mathbb{H}}(G)=m_{\mathbb{H}}(H)$
and suppose that $K_{1}(\Z[H])$ is represented by units. 
Then $\Z[G]$ has SFC if and only if $\Z[H]$ has SFC.
\end{theorem}

\begin{theorem}[Nicholson]\label{thm:Nicholson-B}
Let $G$ be a finite group with binary polyhedral quotient $H$ such that
$m_{\mathbb{H}}(G)=m_{\mathbb{H}}(H)$.
Then $\Z[G]$ has SFC if and only if $\Z[H]$ has SFC. 
\end{theorem}

\begin{theorem}[Nicholson]\label{thm:Nicholson-C}
If $G$ is a finite group with $m_{\mathbb{H}}(G) \leq 1$ then $\Z[G]$ has SFC. 
\end{theorem}

For a finite group $G$ and a field $\F$, let $r_{\F}(G)$ denote the number of 
irreducible $\F$-representations of $G$, and let 
$SK_{1}(\Z[G]) = \ker(K_{1}(\Z[G]) \rightarrow K_{1}(\Q[G]))$; this is a finite abelian group by a theorem 
of Bass \cite[(45.20)]{curtisandreiner_vol2}.
We say that $K_{1}(\Z[G])$ is \emph{represented by trivial units} if the canonical map
$\pm G \rightarrow K_{1}(\Z[G])$ is surjective.

As explained in \cite[p.\ 324]{MR4299591}, 
the following result can be deduced from results of Bass \cite[Corollary~6.3]{MR193120} 
(see also \cite[(45.21)(ii)]{curtisandreiner_vol2})
and Wall
\cite[Proposition~6.5]{MR376746} (see also \cite[(46.4)]{curtisandreiner_vol2}),
and in certain situations, can be used to apply Theorem~\ref{thm:Nicholson-A}.

\begin{prop}\label{prop:Nicholson-D}
Let $G$ be a finite group such that $SK_{1}(\Z[G])=0$ and $r_{\R}(G)=r_{\Q}(G)$.
Then $K_{1}(\Z[G])$ is represented by trivial units. 
\end{prop}

\section{Fiber products and reduction criteria for SFC}\label{sec:fiber-products}

Let $R$ be a Dedekind domain whose quotient field $K$ is a number field.
Let $A$ be a finite-dimensional semisimple $K$-algebra and let $\Lambda$ be an 
$R$-order in $A$. 

\subsection{Fiber products}\label{subsec:fiber-products}
Suppose that there is a fiber product 
\begin{equation}\label{eq:fiber-product}
\begin{tikzcd}
\Lambda \arrow{d}[swap]{\pi_{2}} \arrow{r}{\pi_{1}} & \Lambda_{1} \arrow{d}{g_{1}}  \\
\Lambda_{2} \arrow{r}[swap]{g_{2}} & \overline{\Lambda}.
\end{tikzcd}
\end{equation}
in which $\Lambda_{1}$ and $\Lambda_{2}$ are $R$-orders in finite-dimensional semisimple 
$K$-algebras $A_{1}$, and $A_{2}$, respectively, and where $\overline{\Lambda}$ is an $R$-torsion $R$-algebra. In particular, this means that there are $R$-algebra isomorphisms
\begin{equation}\label{eq:fiber-in-direct-sum}
\Lambda \cong \{ (x_{1},x_{2}) \in \Lambda_{1} \oplus \Lambda_{2} \mid g_{1}(x_{1})=g_{2}(x_{2}) \}
\quad \textrm{ and } \quad
A \cong A_{1} \oplus A_{2}, 
\end{equation}
which we take to be identifications.

\begin{example}\label{example:fiber-ideals}
Let $I$ and $J$ be two-sided ideals of $\Lambda$ such that $I \cap J = 0$
and $K(I+J)=A$. Let $\Lambda_{1} = \Lambda / I$, let $\Lambda_{2} = \Lambda / J$
and let $\overline{\Lambda} = \Lambda / (I+J)$.
Then $\overline{\Lambda}$ is $R$-torsion and there is a fiber product as in \eqref{eq:fiber-product} such that each of $\pi_{1},\pi_{2},g_{1}, g_{2}$ is the relevant canonical surjection.
\end{example}

\subsection{Review of results of Reiner--Ullom}
We now recall \cite[Lemma 4.20]{MR349828}, which is crucial to all of the results in this section.

\begin{prop}[Reiner--Ullom]\label{prop:properties-of-Mu}
Given a fiber product as in \eqref{eq:fiber-product}, assume that at least one of $g_{1}$ and $g_{2}$
is surjective. For $u \in \overline{\Lambda}^{\times}$, define
\[
M(u) := \{ (x_{1},x_{2}) \in \Lambda_{1} \oplus \Lambda_{2} \mid g_{1}(x_{1})u=g_{2}(x_{2}) \}. 
\]
Then the following statements hold:
\begin{enumerate}
\item Each $M(u)$ is a locally free $\Lambda$-lattice of rank $1$ such that 
\[
\Lambda_{i} \otimes_{\Lambda} M(u) \cong \Lambda_{i} \cdot M(u) = \Lambda_{i}
\textrm{ for } i=1,2.
\]
\item For $u, u' \in \overline{\Lambda}^{\times}$, we have
\[
M(u) \oplus M(u') \cong \Lambda \oplus M(u u') \cong \Lambda \oplus M(u'u).
\]
\item Let $\mathcal{M}$ be a maximal $R$-order in $A$ containing $\Lambda$.
Then 
\[
\mathcal{M} \otimes_{\Lambda} M(u) \cong \mathcal{M} \cdot M(u) = \mathcal{M}.
\]
\item Let $X$ be any locally free $\Lambda$-lattice such that $\Lambda_{i} \otimes_{\Lambda} X \cong \Lambda_{i}$ for $i=1,2$. Then $X \cong M(u)$ for some $u \in \overline{\Lambda}^{\times}$.
\item For $u, u' \in \overline{\Lambda}^{\times}$, we have
$M(u) \cong M(u')$ if and only if 
$u' \in g_{1}(\Lambda_{1}^{\times}) \cdot u \cdot g_{2}(\Lambda_{2}^{\times})$.
\end{enumerate} 
\end{prop}

\begin{corollary}\label{cor:properties-of-Mu}
In the setting of Proposition~\ref{prop:properties-of-Mu}, the following statements hold:
\begin{enumerate}
\item The map $\partial : \overline{\Lambda}^{\times} \rightarrow \Cl(\Lambda)$
defined by $\partial(u) = [M(u)]_{\st}$ is a homomorphism.
\item  For $u \in \overline{\Lambda}^{\times}$, 
we have that $M(u) \cong \Lambda$ if and only if $u \in W$, where
\[
W := \{ w_{1}w_{2} \mid w_{i} \in g_{i}(\Lambda_{i}^{\times}), i=1,2 \}. 
\]
\item For $u \in \overline{\Lambda}^{\times}$, 
we have that $M(u)$ is stably free over $\Lambda$ if and only if $u \in \ker(\partial)$.
\item We have a containment of sets
$W \subseteq \ker(\partial)$.
\end{enumerate}
\end{corollary}

\begin{proof}
Parts (i) and (ii) follow from Proposition~\ref{prop:properties-of-Mu} parts (ii) and (v), respectively. 
Part (iii) follows from the definition of $\Cl(\Lambda)$ given in  \S \ref{subsec:lfcg}.
Part (iv) follows from parts (ii) and (iii).
\end{proof}

\subsection{A reduction criterion for SFC}
We henceforth assume the setting and notation of Proposition~\ref{prop:properties-of-Mu} and Corollary~\ref{cor:properties-of-Mu}.
The following underpins many of our results.

\begin{prop}\label{prop:SFC-W-criterion}
$\Lambda$ has SFC if and only if both 
$\Lambda_{1}$ and $\Lambda_{2}$ have SFC and $W=\ker(\partial)$.
\end{prop}

\begin{proof}
Suppose that $\Lambda$ has SFC. 
The maps $K \otimes_{R} \pi_{i} : A \rightarrow A_{i}$ are surjective for $i=1,2$
by \eqref{eq:fiber-in-direct-sum}.
Thus both $\Lambda_{1}$ and $\Lambda_{2}$ have SFC by Theorem~\ref{thm:cancellation-under-map}. 
Let $u \in \ker(\partial)$. 
Then $\partial(u) = [M(u)]_{\st} = 0$ in $\Cl(\Lambda)$ and so $M(u) \cong \Lambda$ since $\Lambda$ has SFC.
Thus $u \in W$ by Corollary~\ref{cor:properties-of-Mu} (ii).
Hence $\ker(\partial) \subseteq W$.
But the reverse containment also holds by Corollary~\ref{cor:properties-of-Mu} (iv) and 
so $W=\ker(\partial)$.

Suppose conversely that both $\Lambda_{1}$ and $\Lambda_{2}$ have SFC and that 
$W=\ker(\partial)$. 
Recall from \S \ref{subsec:SFCandLFC} that in order to show that $\Lambda$ has SFC, 
it suffices to show that every stably free $\Lambda$-lattice of rank $1$
is in fact free of rank $1$.
So suppose that  $X$ is a stably free $\Lambda$-lattice of rank $1$, that is, $X \in g(\Lambda)$
and $[X]_{\st} = 0$ in $\Cl(\Lambda)$. Then for $i=1,2$ we have
$[\Lambda_{i} \otimes_{\Lambda} X]_{\st} = 0$ in $\Cl(\Lambda_{i})$ 
and so $\Lambda_{i} \otimes_{\Lambda} X \cong \Lambda_{i}$ since $\Lambda_{i}$ has SFC.
Hence $X \cong M(u)$ for some $u \in \overline{\Lambda}^{\times}$
by Proposition~\ref{prop:properties-of-Mu} (iv). 
Since $X$ is stably free, Corollary~\ref{cor:properties-of-Mu} (iii)
implies that $u \in \ker(\partial)$.
Thus since $W=\ker(\partial)$, Corollary~\ref{cor:properties-of-Mu} (ii) implies that $M(u) \cong \Lambda$.
Therefore $X \cong M(u) \cong \Lambda$.
\end{proof}

\begin{corollary}\label{cor:subgroup-of-W}
Suppose that both $\Lambda_{1}$ and $\Lambda_{2}$ have SFC and that $W$ is a group. 
Let $N$ be a normal subgroup of $\ker(\partial)$ contained in $W$ and let $S$ be a generating subset 
of $\ker(\partial)/N$.
Then the following assertions are equivalent:
\begin{enumerate}
\item $\Lambda$ has SFC.
\item For every $\overline{u} \in \ker(\partial)/N$ and every lift $u \in \ker(\partial)$ we have
$M(u) \cong \Lambda$.
\item For every $\overline{u} \in S$ there exists a lift $u \in \ker(\partial)$ such that
$M(u) \cong \Lambda$. 
\end{enumerate}
\end{corollary}

\begin{proof}
We shall show that (i) $\implies$ (ii) $\implies$ (iii) $\implies$ (i).
Suppose (i) holds.
Then $W = \ker(\partial)$ by Proposition~\ref{prop:SFC-W-criterion}.
Thus $M(u) \cong \Lambda$ for every $u \in \ker(\partial)$ by Corollary~\ref{cor:properties-of-Mu}~(ii)
and so, in particular, assertion (ii) holds.
Clearly (ii) implies (iii).
Suppose (iii) holds. 
By Proposition~\ref{prop:SFC-W-criterion}, to show that (i) holds, it suffices to show that $\ker(\partial) = W$.
Let $v \in \ker(\partial)$ and let $\overline{v}$ denote its image in $\ker(\partial)/N$.
Then there exist (not necessarily distinct) elements $\overline{s}_{1}, \ldots, \overline{s}_{r} \in S$
such that $\overline{v}=\overline{s}_{1} \cdots \overline{s}_{r}$. 
By hypothesis, there exist lifts $s_{1}, \ldots, s_{r} \in \ker(\partial)$ such that 
$M(s_{1}) \cong \cdots \cong M(s_{r}) \cong \Lambda$, and so $s_{1}, \ldots, s_{r} \in W$
by Corollary~\ref{cor:properties-of-Mu}~(ii).
Moreover, there exists $n \in N$ such that $v=s_{1}\cdots s_{r}n$. 
Since $W$ is a group and $N \leq W$, we conclude that $v \in W$. 
Therefore $\ker(\partial) \subseteq W$.
Since the reverse inclusion holds by Corollary~\ref{cor:properties-of-Mu}~(iv), 
we conclude that $\ker(\partial) = W$.
\end{proof}

\subsection{A criterion in terms of orders of locally free class groups}\label{subsec:criterion-sizes-lfcg}

The following is a variant of \cite[(49.30)]{curtisandreiner_vol2}, which itself is based on 
\cite[\S 4, \S 5]{MR349828}.

\begin{prop}\label{prop:SFC-exact sequence}
There is an exact sequence
\begin{equation}\label{eq:SFC-exact-sequence}
1 \longrightarrow \ker(\partial) \longrightarrow
\overline{\Lambda}^{\times} \stackrel{\partial}{\longrightarrow} \Cl(\Lambda)
\longrightarrow \Cl(\Lambda_{1}) \oplus \Cl(\Lambda_{2}) \longrightarrow 0. 
\end{equation}
\end{prop}

\begin{proof}
Since $\overline{\Lambda}$ is semilocal, the canonical map
$\pi : \overline{\Lambda}^{\times} \rightarrow K_{1}(\overline{\Lambda})$ is surjective
by \cite[(40.31)]{curtisandreiner_vol2}. Moreover, 
by \cite[(49.27)]{curtisandreiner_vol2} there is an exact sequence
\begin{equation}\label{eq:exact-seq-with-K1s}
K_{1}(\Lambda_{1}) \times K_{1}(\Lambda_{2}) 
\longrightarrow K_{1}(\overline{\Lambda}) 
\stackrel{\partial'}{\longrightarrow} 
\Cl(\Lambda)
\longrightarrow \Cl(\Lambda_{1}) \oplus \Cl(\Lambda_{2}) \longrightarrow 0.
\end{equation}
As shown in the proof of loc.\ cit.\ $\partial'$ can be defined by $\partial = \partial' \circ \pi$ and taking lifts via $\pi$.
(Note that the map $\partial$ here is different to that of loc.\ cit.)
Thus $\im(\partial)=\im(\partial')$.
Therefore the desired result follows from the existence and exactness of \eqref{eq:exact-seq-with-K1s}.
\end{proof}

\begin{corollary}\label{cor:SFC-conditions}
We have
\begin{equation}\label{eq:W-inequality}
|W| \leq |\ker(\partial)|= 
|\Cl(\Lambda_{1})||\Cl(\Lambda_{2})||\overline{\Lambda}^{\times}||\Cl(\Lambda)|^{-1}. 
\end{equation}
Moreover, $\Lambda$ has SFC if and only if both $\Lambda_{1}$ and $\Lambda_{2}$ have SFC and
\eqref{eq:W-inequality} is an equality.
\end{corollary}

\begin{proof}
Proposition~\ref{prop:SFC-exact sequence} shows that 
$|{\ker(\partial)}| = |\Cl(\Lambda_{1})||\Cl(\Lambda_{2})||\overline{\Lambda}^{\times}||\Cl(\Lambda)|^{-1}$.
The result now follows from Corollary~\ref{cor:properties-of-Mu}~(iv) and 
Proposition~\ref{prop:SFC-W-criterion}.
\end{proof}

\begin{remark}\label{rmk:replace-lfcgs-by-kernel-grps}
The locally free class groups in Proposition~\ref{prop:SFC-exact sequence} 
and Corollary~\ref{cor:SFC-conditions} can be replaced by 
the corresponding kernel groups as defined in \cite[(49.33)]{curtisandreiner_vol2};
this follows from Proposition~\ref{prop:properties-of-Mu}~(iii) and \cite[(49.39)]{curtisandreiner_vol2}. 
Thus \eqref{eq:SFC-exact-sequence} becomes
\[
1 \longrightarrow \ker(\partial) \longrightarrow
\overline{\Lambda}^{\times} \stackrel{\partial}{\longrightarrow} D(\Lambda)
\longrightarrow D(\Lambda_{1}) \oplus D(\Lambda_{2}) \longrightarrow 0,
\]
and so \eqref{eq:W-inequality} becomes
$|W| 
\leq |{\ker(\partial)}| =  
|D(\Lambda_{1})||D(\Lambda_{2})||\overline{\Lambda}^{\times}||D(\Lambda)|^{-1}$.  
Both locally free class groups and kernel groups can be calculated
by using the algorithm of \cite[\S 3]{bley-boltje}. 
\end{remark}

\begin{remark}\label{rmk:units-jac-radical}
For any associative unital ring $\mathfrak{R}$, we have a short exact sequence
\[
1 \longrightarrow 
1+J(\mathfrak{R}) 
\longrightarrow 
\mathfrak{R}^{\times} 
\longrightarrow
\left( \mathfrak{R}/J(\mathfrak{R}) \right)^{\times}
\longrightarrow 1,
\]
where $J(\mathfrak{R})$ denotes the Jacobson radical of $\mathfrak{R}$ (see \cite[Exercise 5.2]{curtisandreiner_vol1}, for example). 
In particular, if $\mathfrak{R}$ is finite then $|\mathfrak{R}^{\times}| 
= |J(\mathfrak{R})| |(\mathfrak{R}/J(\mathfrak{R}))^{\times}|$.
If $\mathfrak{R}$ is a finite-dimensional algebra over a finite field, then $J(\mathfrak{R})$
can be computed using the algorithm of \cite{Friedl1985}.
We will use these observations for the computation of the cardinality of
$\overline{\Lambda}^{\times}$ in certain situations.
\end{remark}

\subsection{A criterion assuming the Eichler condition for one component}
The following is a specialisation of \cite[Theorem 10.2]{MR584612}.

\begin{theorem}[Swan]\label{thm:Swan-Eichler}
Let $\mathfrak{I}$ be a two-sided ideal of $\Lambda$ such that 
$\Lambda/\mathfrak{I}$ is $R$-torsion
and let $f: \Lambda \rightarrow \Lambda / \mathfrak{I}$ be the canonical quotient map.
Assume that $A$ satisfies the Eichler condition.
Then $f(\Lambda^{\times})$ contains the commutator subgroup
$[(\Lambda/\mathfrak{I})^{\times}, (\Lambda/\mathfrak{I})^{\times}]$.
\end{theorem}

This leads to the following result whose corollary is useful for computational purposes.

\begin{prop}\label{prop:Eichler-W-consequences}
Let $I$ and $J$ be two-sided ideals of $\Lambda$ such that $K(I+J)=A$ and $I \cap J = 0$. 
Let $\Lambda_{1} = \Lambda / I$, let $\Lambda_{2} = \Lambda / J$ 
and let $\overline{\Lambda} = \Lambda / (I+J)$.
Assume that $A_{1} := K\Lambda_{1} \cong KJ$ satisfies the Eichler condition.
For $i=1,2$, let $g_{i} : \Lambda_{i} \rightarrow \overline{\Lambda}$ be the canonical surjection.
Let $\partial$ and $W$ be as in Corollary~\ref{cor:properties-of-Mu}.
Then $\overline{\Lambda}$ is finite and we have
\[
[\overline{\Lambda}^{\times}, \overline{\Lambda}^{\times}] 
\trianglelefteq g_{1}(\Lambda_{1}^{\times})
\trianglelefteq W
\trianglelefteq \ker(\partial)
\trianglelefteq \overline{\Lambda}^{\times},
\]
where each term is a normal subgroup of $\overline{\Lambda}^{\times}$ with finite abelian quotient.
\end{prop}

\begin{proof}
The hypotheses imply that we are in the setting of Example~\ref{example:fiber-ideals}.
In particular, $\overline{\Lambda}$ is $R$-torsion and hence finite. 
Moreover, every subgroup of $\overline{\Lambda}^{\times}$ containing 
$[\overline{\Lambda}^{\times}, \overline{\Lambda}^{\times}]$ is necessarily normal in 
$\overline{\Lambda}^{\times}$ and has finite abelian quotient.
Theorem~\ref{thm:Swan-Eichler} applied with $f=g_{1}$ implies that
$[\overline{\Lambda}^{\times}, \overline{\Lambda}^{\times}] \trianglelefteq g_{1}(\Lambda_{1}^{\times})$.
Since $W=g_{1}(\Lambda_{1}^{\times})g_{2}(\Lambda_{2}^{\times})$ by definition and 
$g_{1}(\Lambda_{1}^{\times})$ is normal in $\overline{\Lambda}^{\times}$, we have that 
$W$ is in fact a subgroup of $\overline{\Lambda}^{\times}$ (a priori, it is only a subset),
and that $g_{1}(\Lambda_{1}^{\times}) \trianglelefteq W$. 
The remaining containments follow from Corollary~\ref{cor:properties-of-Mu}.
\end{proof}

\begin{corollary}\label{cor:Eichler-W-consequences}
Assume the setting of Proposition~\ref{prop:Eichler-W-consequences}.
Then $W$ is a group.
Suppose that $N$ is a subgroup of $W$ containing 
$[\overline{\Lambda}^{\times}, \overline{\Lambda}^{\times}]$.
Then $N \trianglelefteq \ker(\partial)$ and $\ker(\partial)/N$ is a finite abelian group.
Let $S$ be a generating subset of $\ker(\partial)/N$.
Suppose that $\Lambda_{2}$ has SFC.
Then the following assertions are equivalent:
\begin{enumerate}
\item $\Lambda$ has SFC.
\item For every $\overline{u} \in \ker(\partial)/N$ and every lift $u \in \ker(\partial)$ we have
$M(u) \cong \Lambda$.
\item For every $\overline{u} \in S$ there exists a lift $u \in \ker(\partial)$ such that
$M(u) \cong \Lambda$. 
\end{enumerate}
\end{corollary}

\begin{proof}
The assertions that $W$ is a group, $N \trianglelefteq \ker(\partial)$ and $\ker(\partial)/N$ is a finite abelian group all follow from Proposition~\ref{prop:Eichler-W-consequences}.
Since $A_{1}$ satisfies the Eichler condition, $\Lambda_{1}$ has SFC by
Theorem~\ref{thm:Jacobinski-cancellation}. 
Moreover, $\Lambda_{2}$ satisfies SFC by hypothesis.
Thus assertions (i), (ii) and (iii) are equivalent by Corollary~\ref{cor:subgroup-of-W}.
\end{proof}

\section{Applications to integral group rings}\label{sec:app-to-int-group-rings}

In this section, we consider applications of Corollary~\ref{cor:SFC-conditions} to integral group rings.

\subsection{A fiber product for integral group rings}\label{subsec:fiber-prod-int-grp-rings}
We first record the following well-known instance of a fiber product.
Let $G$ be a finite group, let $N$ be a normal subgroup and let $H=G/N$.
Let $n=|N|$ and let $\mathrm{Tr}_{N} = \sum_{\tau \in N} \tau \in \Z[G]$ 
be the trace element associated to $N$.
Let $\pi_{1} : \Z[G] \rightarrow \Z[H]$ be the canonical projection and let $I = \ker(\pi_{1})$.
Let $J= \mathrm{Tr}_{N} \Z[G]$, let $\Gamma = \Z[G]/J$ and let $\pi_{2} : \Z[G] \rightarrow \Z[G]/J$
be the canonical projection. Then $I \cap J = 0$ and $I+J = n\Z[G]$.
Thus, by Example~\ref{example:fiber-ideals} we have a fiber product
\[
\begin{tikzcd}
\Z[G] \arrow{d}[swap]{\pi_{2}} \arrow{r}{\pi_{1}} & \Z[H] \arrow{d}{g_{1}}   \\
\Gamma \arrow{r}[swap]{g_{2}} & (\Z/n\Z)[H],
\end{tikzcd}
\]
where each of $\pi_{1},\pi_{2},g_{1}, g_{2}$ is the relevant canonical surjection.
Let 
\[
W_{G,N} := \{ w_{1}w_{2} \mid w_{1} \in g_{1}(\Z[H]^{\times}), w_{2} \in g_{2}(\Gamma^{\times}) \}.
\]
Then Corollary~\ref{cor:SFC-conditions} immediately implies the following result.

\begin{prop}\label{prop:group-ring-SFC-reduction}
Assume the above setting and notation. Then
\begin{equation}\label{eq:W-inequality-igr}
|W_{G,N}| \le \frac{|\Cl(\Z[H])| |\Cl(\Z[\Gamma])| |(\Z/n\Z)[H]^\times|}{|\Cl(\Z[G])|}. 
\end{equation}
Moreover, $\Z[G]$ has SFC if and only if both $\Z[H]$ and $\Gamma$ have SFC and
\eqref{eq:W-inequality-igr} is an equality.
\end{prop}

\begin{remark}
The locally free class groups in Proposition~\ref{prop:group-ring-SFC-reduction} 
(and in Propositions~\ref{prop:ZHtimesC2} and \ref{prop:ZG-SFC-iff-Gamma-SFC} below)
can be replaced by the corresponding kernel groups (see Remark~\ref{rmk:replace-lfcgs-by-kernel-grps}).
\end{remark}

\begin{remark}\label{rmk:embed-units-into-GLFp}
Let $p$ be a prime number and suppose that $n=p$ in Proposition~\ref{prop:group-ring-SFC-reduction}.
Then $(\Z/n\Z)[H] = \mathbb{F}_{p}[H]$.
Thus left multiplication by an element of $\mathbb{F}_{p}[H]^{\times}$ induces 
an $\F_{p}$-linear automorphism of $\F_{p}[H]$. 
Together with the choice of $H$ as an $\F_{p}$-basis of $\F_{p}[H]$, 
this determines an explicit embedding of $\F_{p}[H]^{\times}$ into $\GL_{h}(\F_{p})$,
where $h=|H|$. This allows computations in $\F_{p}[H]^{\times}$ to be performed 
efficiently on a computer algebra system. 
By contrast, in the more general setting of Corollary~\ref{cor:SFC-conditions}, it
can be difficult to perform computations in $\overline{\Lambda}^{\times}$
(see \S \ref{subsec:unit-groups-finite-rings}).
\end{remark}

\subsection{Integral group rings of the form $\Z[H \times C_{2}]$}
We will need the following result.

\begin{prop}\label{prop:ZHtimesC2}
Let $H$ be a finite group such that $\Z[H]$ has SFC.
Let $V_{H}$ be the image of the canonical projection $\Z[H]^{\times} \rightarrow \F_{2}[H]^{\times}$.
Then
\[
|V_{H}| \leq
\frac{|\Cl(\Z[H])|^{2}|\F_{2}[H]^{\times}|}{|\Cl(\Z[H \times C_{2}])|},
\] 
and this inequality is an equality if and only if $\Z[H \times C_{2}]$ has SFC.
\end{prop}

\begin{proof}
Let $G=H \times C_{2}$ and let $N=\{ 1 \} \times C_{2} \leq G$. 
Then in the notation of \S \ref{subsec:fiber-prod-int-grp-rings}, we have $n=2$, 
$\Gamma \simeq \Z[H]$ and $V_{H} = W_{G,N}$.
Thus the desired result now follows immediately from Proposition~\ref{prop:group-ring-SFC-reduction}.
\end{proof}

\begin{example}\label{ex:Q8xC2}
Let $Q_{8}$ denote the quaternion group of order $8$.
Swan \cite[Theorem~15.1]{MR703486} showed that $\Z[Q_{8} \times C_{2}]$ does not have SFC.
We now give an alternative proof of this result using Proposition~\ref{prop:ZHtimesC2}
and the following facts: 
\begin{enumerate}
\item $\Z[Q_{8}]$ has SFC (\cite[Theorem I]{MR703486}),
\item $\Z[Q_{8}]^{\times} = \pm Q_{8}$ (\cite[Lemma 15.3]{MR703486}),
\item $|\Cl(\Z[Q_{8}])|=2$ (\cite[Theorem III]{MR703486}), 
\item $|\Cl(\Z[Q_{8} \times C_{2}])|=2^{4}$ (\cite[Theorem 16.5]{MR703486}), and
\item $|\F_{2}[Q_{8}]^{\times}|=2^{7}$ (since $\F_{2}[Q_{8}]$ is a local ring with residue field $\F_{2}$).
\end{enumerate}
Fact (ii) implies that 
$V_{Q_{8}} = Q_{8} \leq \F_{2}[Q_{8}]^{\times}$.
Hence
\[
2^{3} = |Q_{8}| = |V_{Q_{8}}| <  
|\Cl(\Z[Q_{8}])|^{2}|\Cl(\Z[Q_{8} \times C_{2}])|^{-1}|\F_{2}[Q_{8}]^{\times}| 
= 2^{2} \cdot 2^{-4} \cdot 2^{7} = 2^{5},
\]
and so $\Z[Q_{8} \times C_{2}]$ does not have SFC.
Note that the quantities in (iii) and (iv) can instead be calculated on a computer 
by using the algorithm of \cite[\S 3]{bley-boltje}.
\end{example}

Let $\tilde{T} \cong \mathrm{SL}_{2}(\F_{3}) \cong G_{(24,3)}$ denote the binary tetrahedral group. 
Then $\tilde{T} \times C_{2} \cong G_{(48,32)}$.

\begin{theorem}\label{thm:tildeT-times-C2-has-SFC}
The group ring $\Z[\tilde{T} \times C_{2}]$ has SFC. 
\end{theorem}

\begin{proof}
Further details on the computations used in this proof can be found in 
Appendix~\ref{appendix}.
The algorithm of \cite[\S 3]{bley-boltje} shows that 
$|\Cl(\Z[\tilde{T}])|=2$ and $|\Cl(\Z[\tilde{T} \times C_{2}])|=2^{6}$.
By Remark~\ref{rmk:units-jac-radical}, we have
$|\F_{2}[\tilde{T}]^{\times}| = |J(\mathbb{F}_{2}[\tilde{T}])| |(\F_{2}[\tilde{T}]/J(\F_{2}[\tilde{T}]))^{\times}|$, 
where $J(\mathbb{F}_{2}[\tilde{T}])$ denotes the Jacobson radical of $\mathbb{F}_{2}[\tilde{T}]$.
By \cite[Theorem 2.5]{F2SL2F3} we have 
$\F_{2}[\tilde{T}]/J(\F_{2}[\tilde{T}]) \cong \F_{2} \oplus \F_{4}$
(this can also be verified using the algorithm of \cite{Friedl1985}).
Hence 
$|J(\F_{2}[\tilde{T}])|=2^{24-3}=2^{21}$.
Moreover, $(\F_{2}[\tilde{T}]/J(\F_{2}[\tilde{T}]))^{\times} \cong (\F_{2} \oplus \F_{4})^{\times} \cong C_{3}$. 
Therefore $|\F_{2}[\tilde{T}]^{\times}| = 3 \cdot 2^{21}$. 

The integral group ring $\Z[\tilde{T}]$ has SFC by \cite[Theorem I]{MR703486}.
Let $V=V_{\tilde{T}}$ be the image of the canonical projection 
$\pi: \Z[\tilde{T}]^{\times} \rightarrow \F_{2}[\tilde{T}]^{\times}$.
Then by Proposition~\ref{prop:ZHtimesC2} and the above calculations we have
\[
|V| \leq
\frac{|\Cl(\Z[\tilde{T}])|^{2}|\F_{2}[\tilde{T}]^{\times}|}{|\Cl(\Z[\tilde{T} \times C_{2}])|} = 3 \cdot 2^{17},
\]
with equality if and only if $\Z[\tilde{T} \times C_{2}]$ has SFC.

To prove equality, we consider the following units of integral group rings as introduced 
by Ritter--Segal \cite{MR987166}.
For $g, h \in \tilde{T}$, let $d = \operatorname{ord}(g)$ and define
\begin{align*}
u_{g, h} &= 1 + (1 - g)h(1 + g + \dotsb + g^{d-1}),\\
u'_{g, h} &= 1 + (1 + g + \dotsb + g^{d-1}) h (1 - g).
\end{align*}
Then $g$, $u_{g, h}$ and $u'_{g,h}$ are elements of $\Z[\tilde{T}]^\times$.
Let 
\[
U = \langle \pi(g), \pi(u_{g, h}), \pi(u'_{g, h}) \mid g, h \in \tilde{T} \rangle
\]
and note that this is a subgroup of $V$.
By Remark~\ref{rmk:embed-units-into-GLFp}, we obtain an explicit embedding
of $\F_{2}[\tilde{T}]^{\times}$ into $\GL_{24}(\F_{2})$.
Since a calculation on a computer algebra system such as \textsc{Magma}~\cite{Magma} or \textsc{Oscar}~\cite{Oscar, Oscar-book, Fieker2017}
shows that the image of $U$ under this embedding is a subgroup of $\GL_{24}(\F_{2})$
of order $3 \cdot 2^{17}$, we must have $U=V$ and $|V|=3\cdot2^{17}$.
Therefore we conclude that $\Z[\tilde{T} \times C_{2}]$ has SFC.
\end{proof}

\begin{remark}
Using Theorem~\ref{thm:tildeT-times-C2-has-SFC} as an input, Nicholson 
\cite[Theorem 1.7]{nicholson2024cancellation}
has shown that the integral group ring $\Z[\tilde{T} \times C_{2}]$ in fact has LFC.
\end{remark}

\begin{lemma}\label{lem:SK1-tildeTxC2}
We have $SK_{1}(\Z[\tilde{T} \times C_{2}]) = 0$. 
\end{lemma}

\begin{proof}
By \cite[Theorem 3]{MR618310}, for any finite group $G$ and any prime $p$, 
the Sylow $p$-subgroup of  $SK_{1}(\Z[G])$ is generated by induction from the $p$-elementary
subgroups of $G$.
If $H$ is any abelian subgroup of $\tilde{T} \times C_{2}$
then $H \cong C_{n}$ or $C_{2} \times C_{n}$ for some $n$,
and so $SK_{1}(\Z[H])=0$ by \cite[Theorem 4.9]{MR0870729}.
If $H$ is any non-abelian elementary subgroup of $\tilde{T} \times C_{2}$
then $H \cong Q_{8}$ or $Q_{8} \times C_{2}$ and so $SK_{1}(\Z[H])=0$ by
\cite[Theorem 4]{MR618310} and \cite[Example 9.9(ii)]{MR933091}, respectively.
Therefore the desired result holds.
\end{proof}

We now have the following consequence of Theorem~\ref{thm:tildeT-times-C2-has-SFC},
which has also been shown independently by Nicholson 
\cite[Proposition 4.14]{nicholson2024cancellation}.

\begin{corollary}\label{cor:lifts-of-tildeTxC2}
Let $G$ be a finite group with quotient isomorphic to $\tilde{T} \times C_{2}$ and with 
$m_{\mathbb{H}}(G)=2$. Then $\Z[G]$ has SFC.
\end{corollary}

\begin{proof}
Let $H=\tilde{T} \times C_{2}$. 
Then $SK_{1}(\Z[H])=0$ by Lemma~\ref{lem:SK1-tildeTxC2} and 
it is straightforward to deduce from the character table of $H$ that $r_{\R}(H)=r_{\Q}(H)=10$.
Thus $K_{1}(\Z[H])$ is represented by (trivial) units by Proposition~\ref{prop:Nicholson-D}.
We also have that $m_{\mathbb{H}}(H)=2$, which can be deduced from the 
character table of $H$ along with its Frobenius--Schur indicators.
Thus the desired result now follows from
Theorems~\ref{thm:Nicholson-A} and \ref{thm:tildeT-times-C2-has-SFC}.
\end{proof}

\subsection{A useful reduction criterion for certain integral group rings}
The following reduction criterion only uses the orders of certain locally free class groups, and thus can be applied relatively easily. 
It will be used in the proof of Theorem~\ref{thm:new-groups-with-SFC} 
(i.e.\ Theorem~\ref{intro-thm:new-groups-with-SFC}).

\begin{prop}\label{prop:ZG-SFC-iff-Gamma-SFC}
Let $G$ be a finite group with a normal subgroup $N$ of order $2$ and let $H = G/N$.
Let $\Gamma=\Z[G]/\mathrm{Tr}_{N} \Z[G]$.
Suppose that $\Z[H \times C_{2}]$ has SFC and that
\begin{equation}\label{eq:quotients-class-groups}
\frac{|\Cl(\Z[H \times C_{2}])|}{|\Cl(\Z[H])|} = \frac{|\Cl(\Z[G])|}{|\Cl(\Gamma)|}. 
\end{equation}
Then $\Z[G]$ has SFC if and only if $\Gamma$ has SFC.
\end{prop}

\begin{proof}
Since $\Z[H \times C_{2}]$ has SFC by hypothesis, so does $\Z[H]$ by Lemma~\ref{lem:quot-group-ring}.
By Proposition~\ref{prop:ZHtimesC2} and \eqref{eq:quotients-class-groups} we have
\begin{equation}\label{eq:VH-eq}
|V_{H}| = 
\frac{|\Cl(\Z[H])|^{2}|\F_{2}[H]^{\times}|}{|\Cl(\Z[H \times C_{2}])|}
= \frac{|\Cl(\Z[H])||\F_{2}[H]^{\times}||\Cl(\Gamma)|}{|\Cl(\Z[G])|}.
\end{equation}
By Proposition~\ref{prop:group-ring-SFC-reduction} and \eqref{eq:VH-eq} we have
\[
|W_{G,N}| \leq 
\frac{|\Cl(\Z[H])||\Cl(\Gamma)||\F_{2}[H]^{\times}|}{|\Cl(\Z[G])|} = |V_{H}|,
\]
and that $\Z[G]$ has SFC if and only if $\Gamma$ has SFC and this inequality is in fact an equality.
But $V_{H} \subseteq W_{G,N}$ by the definitions.
Therefore $|W_{G,N}|=|V_{H}|$ and so $\Z[G]$ has SFC if and only if $\Gamma$ has SFC.
\end{proof}

\section{Criteria for freeness and stable freeness}\label{sec:critsfc}

\subsection{Setup}\label{subsec:criteria-setup}
Let $K$ be a number field with ring of integers $\mathcal{O}=\mathcal{O}_{K}$, 
and let $A$ be a finite-dimensional semisimple $K$-algebra.
Let $\Lambda$ be an $\mathcal{O}$-order in $A$.
By \cite[(10.4)]{Reiner2003} there exists a (not necessarily unique) maximal $\mathcal{O}$-order $\mathcal{M}$ in $A$ containing $\Lambda$.
Let $\mathfrak{f}$ be any full two-sided ideal of $\mathcal{M}$ that is contained in $\Lambda$.

\subsection{A criterion for freeness}
We recall some results from \cite[\S 4]{MR4493243}. 
Define
\begin{equation}\label{eq:def-of-pi}
\pi : (\mathcal{M}/\mathfrak{f})^{\times}
\longrightarrow
{(\mathcal{M} / \mathfrak{f} )^{\times}} / {(\Lambda / \mathfrak{f} )^{\times}}
\end{equation}
to be the map induced by the canonical projection, where the codomain is the 
collection of left cosets of
$ (\Lambda / \mathfrak{f} )^{\times}$ in $(\mathcal{M} / \mathfrak{f} )^{\times}$.
Note that $(\Lambda/\mathfrak{f})^{\times}$ is a subgroup of $(\mathcal{M}/\mathfrak{f})^{\times}$, but is not necessarily a normal subgroup, and so $\pi$ is only a map of sets in general.

For $\eta \in \mathcal{M}$ we write $\overline{\eta}$ for its image in $\mathcal{M} / \mathfrak{f}$.

\begin{prop}[{\cite[Proposition 4.3]{MR4493243}}]\label{prop:central-imsp}
Let $X$ be a left ideal of $\Lambda$. 
Suppose that $X + \mathfrak{f} = \Lambda$ and that there exists $\beta \in \mathcal{M} \cap A^{\times}$ such that 
$\mathcal{M}X = \mathcal{M}\beta$. Let $u \in \mathcal{M}^{\times}$ and let $\alpha = u \beta$.
Then $\overline{\alpha}, \overline{\beta}, \overline{u} \in (\mathcal{M}/\mathfrak{f})^{\times}$ and 
the following assertions are equivalent:
\begin{enumerate}
\item $X = \Lambda\alpha$,
\item $\Lambda\alpha + \mathfrak{f} = \Lambda$,
\item $\overline{\alpha} \in (\Lambda/\mathfrak{f})^{\times}$,
\item $\pi(\overline{\beta}) = \pi(\overline{u^{-1}})$,
\item $\alpha \in X$ and $X$ is locally free over $\Lambda$.
\end{enumerate}
\end{prop}

\begin{remark}
\cite[Proposition 4.3 and Theorem 4.5]{MR4493243} only assumed that $\beta \in \mathcal{M}$
but should have assumed that $\beta \in \mathcal{M} \cap A^{\times}$, or equivalently,
that $\beta \in \mathcal{M}$ and that $X$ is a \emph{full} left ideal of $\Lambda$, that is, $KX=A$. 
Note that any locally free left ideal of $\Lambda$ is necessarily full.
The strengthened assumption is needed to ensure that the condition $X=\Lambda\alpha$ 
implies that $X$ is in fact free over $\Lambda$ (see \cite[Lemma 2.3]{MR4493243}); 
without this, the implication (i) $\Rightarrow$ (v) in Proposition~\ref{prop:central-imsp} may not hold.
This omission does not affect any of the other results of \cite{MR4493243}.
Also note that in loc.\ cit.\ it was assumed that $\mathfrak{f}$ is strictly contained in $\mathcal{M}$, but in fact this assumption is unnecessary.
\end{remark}

\begin{corollary}\label{cor:central-imsp} 
Let $X$ be a left ideal of $\Lambda$. 
Suppose that $X + \mathfrak{f} = \Lambda$ and that there exists $\beta \in \mathcal{M} \cap A^{\times}$ such that $\mathcal{M}X = \mathcal{M}\beta$.  
Then the following assertions are equivalent:
\begin{enumerate}
\item $X$ is free over $\Lambda$.
\item There exists $u \in \mathcal{M}^{\times}$ and $\overline{\alpha} \in (\Lambda/\mathfrak{f})^{\times}$ such that $\overline{u} = \overline{\beta}\overline{\alpha}$.
\item There exists $u \in \mathcal{M}^{\times}$ such that $\pi(\overline{\beta})=\pi(\overline{u})$.
\end{enumerate}
\end{corollary}

\begin{proof}
This follows from Proposition~\ref{prop:central-imsp} after substituting $u$ and $\overline{\alpha}$ for their inverses. 
\end{proof}

\subsection{An exact sequence}
Most of the rest of this section will be concerned with the development of 
a criterion for stable freeness. We shall first need a certain exact sequence.

Let
\[
\mathcal{D} = \mathcal{D}(\Lambda, \mathfrak{f}) 
:= \{ (\lambda_{1}, \lambda_{2}) \in \Lambda \oplus \Lambda
\mid \lambda_{1} \equiv \lambda_{2} \bmod{\mathfrak{f}} \}.
\]
Let $g_{1}=g_{2} : \Lambda \rightarrow \Lambda / \mathfrak{f}$ be the canonical projection map.
For $i=1,2$ define $\pi_{i} : \mathcal{D} \rightarrow \Lambda$ to be the restriction of the projection map from $\Lambda \oplus \Lambda$ onto its $i$th component.
Then we have a fiber product 
\[
\begin{tikzcd}
\mathcal{D} \arrow{d}[swap]{\pi_{2}} \arrow{r}{\pi_{1}} & \Lambda \arrow{d}{g_{1}}  \\
\Lambda \arrow{r}[swap]{g_{2}} & \Lambda / \mathfrak{f} .
\end{tikzcd}
\]
For $i=1,2$ define $q_{i}: \Cl(\mathcal{D}) \rightarrow \Cl(\Lambda)$ by 
$[X]_{\st} \mapsto [\Lambda \otimes_{\mathcal{D},\pi_{i}} X]_{\st}$.
Then by Proposition~\ref{prop:SFC-exact sequence} we obtain an exact sequence
\[
1 \longrightarrow \ker(\partial) \longrightarrow
(\Lambda / \mathfrak{f})^{\times} \stackrel{\partial}{\longrightarrow} \Cl(\mathcal{D})
\xrightarrow{(q_{1},q_{2})} \Cl(\Lambda) \oplus \Cl(\Lambda) \longrightarrow 0,
\]
where $\partial(u)=[M(u)]_{\st}$ and 
$M(u) := \{ (\lambda_{1},\lambda_{2}) \in \Lambda \oplus \Lambda \mid g_{1}(\lambda_{1})u=g_{2}(\lambda_{2}) \}$.
Set 
\[
\Cl(\Lambda, \mathfrak{f}) := \ker (q_{2} : \Cl(\mathcal{D}) \rightarrow \Cl(\Lambda)).
\]
Then we obtain an exact sequence
\begin{equation}\label{eq:basic-exact-sequence}
(\Lambda / \mathfrak{f})^{\times} \stackrel{\partial}{\longrightarrow} \Cl(\Lambda, \mathfrak{f})
\stackrel{q_{1}}{\longrightarrow} \Cl(\Lambda) \longrightarrow 0. 
\end{equation}

\subsection{Decomposition via central idempotents}\label{subsec:decomposition}
Much of the following notation is adopted from \cite{MR4493243}.
Let $C=Z(A)$ and let $\mathcal{O}_{C}$ be the integral closure of $\mathcal{O}$ in $C$.
Set $\mathfrak{g} := \mathfrak{f} \cap \mathcal{O}_{C}$ and note that this is a full ideal of $\mathcal{O}_{C}$.
Let $e_{1}, \ldots, e_{r}$ be the primitive idempotents
of $C$ and set $A_{i} = Ae_{i}$. Then
\begin{equation}\label{eq:weddecomp}
A = A_{1} \oplus \cdots \oplus A_{r}
\end{equation}
is a decomposition of $A$ into indecomposable two-sided ideals (see \cite[(3.22)]{curtisandreiner_vol1}).
Each $A_{i}$ is a simple $K$-algebra with identity element $e_{i}$.
The centers $K_{i} := Z(A_{i})$ are finite field extensions of $K$ via $K \rightarrow K_{i}$, $\alpha \mapsto \alpha e_{i}$,
and we have $K$-algebra isomorphisms $A_{i} \cong \Mat_{n_{i}}(D_{i})$ where $D_{i}$ is a division algebra with $Z(D_{i}) \cong K_{i}$ 
(see \cite[(3.28)]{curtisandreiner_vol1}).
The Wedderburn decomposition \eqref{eq:weddecomp} induces decompositions
\[
C = K_{1} \oplus \cdots \oplus K_{r}
\quad \textrm{and} \quad  \mathcal{O}_{C} = \mathcal{O}_{K_{1}} \oplus \cdots \oplus \mathcal{O}_{K_{r}},
\]
where $\calO_{K_i}$ denotes the ring of algebraic integers of $K_{i}$.
By \cite[(10.5)]{Reiner2003} we have $e_{1}, \ldots, e_{r} \in \mathcal{M}$ and each $\mathcal{M}_{i}:=\mathcal{M}e_{i}$ is a maximal $\mathcal{O}$-order 
(and thus a maximal $\mathcal{O}_{K_{i}}$-order) in $A_{i}$.
Moreover, each $\frf_{i}:=\frf e_{i}$ is a full two-sided ideal of $\calM_{i}$, 
each $\frg_{i} := \frg e_{i}$ is a nonzero integral ideal of $\mathcal{O}_{K_{i}}$, 
and we have decompositions
\[
\mathcal{M} =  \mathcal{M}_{1} \oplus \cdots \oplus \mathcal{M}_{r},
\quad 
\frf = \frf_{1} \oplus \cdots \oplus \frf_{r}
\quad \textrm{and} \quad
\frg = \frg_{1} \oplus \cdots \oplus \frg_{r}.
\]

\subsection{The reduced norm map}\label{subsec:nr}
The reduced norm map $\nr : A \rightarrow C$ is defined componentwise (see \cite[\S 9]{Reiner2003})
and restricts to a group homomorphism $\nr : A^{\times} \rightarrow C^{\times}$.
By the Hasse--Schilling--Maass Theorem \cite[Theorem~7.48]{curtisandreiner_vol1} we have
$\nr(A^{\times})=C^{\times +}$, where
\[
C^{\times+} := \{c \in C^{\times} \mid c \text{ is positive at quaternionic components} \} .
\]
The last condition means that if $c=(c_{i})$ with $c_{i} \in K_{i}^{\times}$
then $\tau(c_{i})>0$ whenever $i \in \{1,\ldots,r\}$ and $\tau\colon
K_{i} \to \R$ is a real embedding such that the corresponding scalar
extension $A_{i} \otimes_{K_{i}, \tau} \R$ is isomorphic to 
a full matrix algebra over $\mathbb{H}$.
By \cite[(45.7)]{curtisandreiner_vol2} we have that 
$\nr(\mathcal{M}^{\times})=\mathcal{O}_{C}^{\times} \cap C^{\times +}$.

\subsection{Id\`eles}\label{subsec:ideles}
Let $\mathfrak{p}$ range over all maximal ideals of $\mathcal{O}$.
Recall from \S \ref{subsec:localizations-and-completions} that 
for an $\mathcal{O}$-lattice $M$, we write $M_{\mathfrak{p}}$ and $\widehat{M}_{\mathfrak{p}}$
for the localization and completion of $M$ at $\mathfrak{p}$, respectively.  
Extending this notation, let $\widehat{K}_{\mathfrak{p}}$ denote the $\mathfrak{p}$-adic completion of
$K$ and let $\widehat{A}_{\mathfrak{p}} = \widehat{K}_{\mathfrak{p}} \otimes_{K} A$.

The group of \textit{id\`eles} of $A$ is defined to be 
\[
  J(A) = \{(a_\frp)_{\frp} \in \textstyle{\prod_{\frp}} \widehat{A}^{\times}_{\frp} \mid
  a_\frp \in\widehat{\Lambda}^{\times}_\frp \text{ for all but finitely many } \frp \}.
\]
Note that $J(A)$ does not depend on the choice of $\mathcal{O}$-order $\Lambda$ in $A$.
Let $U(\Lambda)$ be the subgroup of \textit{unit id\`eles}, defined by 
$U(\Lambda)=\prod_{\frp} \widehat{\Lambda}^{\times}_\frp$.
Let $U_{\frf}(\Lambda)$ be the subgroup of $U(\Lambda)$ defined by
\[
U_{\frf}(\Lambda) := \{(a_{\frp})_{\frp} \in U(\Lambda) \mid a_{\frp} \equiv 1 \bmod \frf_{\frp} \}.
\]

We have canonical isomorphisms
\mpar{eq:identification}
\begin{align}\notag
  \widehat{A}_\frp = \widehat{K}_{\mathfrak{p}} \otimes_{K} A 
  & \cong \bigoplus_{i=1}^{r}
  \widehat{K}_{\mathfrak{p}} \otimes_{K} A_{i} 
  \cong \bigoplus_{i=1}^{r} \widehat{K}_{\mathfrak{p}} \otimes_{K}
  K_{i} \otimes_{K_i} A_{} \\ \label{eq:identification}
  & \cong \bigoplus_{i=1}^{r}\bigoplus_{\frP}\
  \widehat{(K_i)}_{\mathfrak{P}}\otimes_{K_i} A_{i} \cong 
  \bigoplus_{i,\frP}
  \widehat{A}_{i,\frP}
\end{align}
involving various completions, where, for given $i\in\{1,\ldots,r\}$,
$\mathfrak{P}$ runs through all maximal ideals of $\calO_{K_i}$ dividing
$\mathfrak{p}$ and $\widehat{A}_{i,\mathfrak{P}}$ is defined as $\widehat{(A_{i})}_{\mathfrak{P}}$.
Using \eqref{eq:identification}, we will often write elements
of $J(A)$, resp.~$\widehat{A}_{\mathfrak{p}}$, as tuples $(a_{i,\frP})_{i,\frP}$,
where $\frP$ ranges over all maximal ideals of $\calO_{K_i}$,
resp.~over those that contain $\frp$. 
Similarly, we denote by $J(C)$ the group
of id\`eles of $C$. 
Again, we have a canonical isomorphism
\[
  \widehat{C}_\frp\cong \bigoplus_{i,\frP}\ \widehat{K}_{i,\frP}
\]
and we will write elements in $J(C)$, resp.~$\widehat{C}_\frp$, often in the
form $(\alpha_{i,\frP})_{i,\frP}$.

By $\nr\colon J(A)\to J(C)$ we also denote the reduced norm map 
(which translates into the componentwise reduced norm maps $\nr\colon
\widehat{A}_{i,\frP}^{\times} \to \widehat{K}_{i,\frP}^{\times}$ 
after the above identifications). 
Then by \cite[(45.8)]{curtisandreiner_vol2} we have that $\nr(J(A))=J(C)$.

\begin{theorem}[{\cite[(49.17)]{curtisandreiner_vol2}}]\label{thm:class-group-ideles}
There is a canonical isomorphism
\[
\Cl(\Lambda) \cong \frac{J(C)}{C^{\times +}\nr(U(\Lambda))}.
\]
\end{theorem}

\begin{remark}\label{rmk:recover-lattice-from-ideles}
For $\gamma = (\gamma_{\mathfrak{p}})_{\mathfrak{p}} \in J(A)$ define 
$\Lambda \gamma = A \cap (\cap_{\mathfrak{p}} \widehat{\Lambda}_{\mathfrak{p}} \gamma_{\mathfrak{p}})$.
Then the map 
$\varepsilon: J(A) \rightarrow \Cl(\Lambda)$ defined by $\alpha \mapsto [\Lambda \alpha]_{\st}$ is a surjective homomorphism.
Now let $X$ be a locally free left ideal of $\Lambda$. 
Then for each $\mathfrak{p}$ there exists
$\alpha_{\mathfrak{p}} \in \widehat{A}_{\mathfrak{p}}^{\times} \cap \widehat{\Lambda}_{\mathfrak{p}}$
such $\widehat{X}_{\mathfrak{p}} = \widehat{\Lambda}_{\mathfrak{p}} \alpha_{\mathfrak{p}}$.
Set $\alpha = (\alpha_{\mathfrak{p}})_{\mathfrak{p}}$. Then $\Lambda \alpha = X$
and so $\varepsilon(\alpha) = [X]_{\st}$. 
In fact,
$\varepsilon$ factors via $\nr\colon J(A)\to J(C)$ and so under the isomorphism of 
Theorem~\ref{thm:class-group-ideles} $(\nr(\alpha_{\mathfrak{p}})_{\mathfrak{p}}) \in J(C)$
is a representative of $[X]_{\st}$.
Note that we could also have chosen $\alpha_{\mathfrak{p}} \in A^{\times} \cap \Lambda_{\mathfrak{p}}$
such that $X_{\mathfrak{p}} = \Lambda_{\mathfrak{p}} \alpha_{\mathfrak{p}}$.
For further details on the claims in this remark, see \cite[\S 49A]{curtisandreiner_vol2}.
\end{remark}

\subsection{Review of results of Bley--Boltje}\label{subsec:review-BB}
We now recall some key results from \cite{bley-boltje}.

Let $I_{\frg}=I_{\frg}(C)$ denote the group of fractional $\mathcal{O}_{C}$-ideals of $C$
that are coprime to $\frg$, and for each $i$, define $I_{\frg_{i}}(K_{i})$ analogously. 
Then we have a direct product decomposition
\[
I_{\frg}(C)=I_{\frg_{1}}(K_{1})\times\cdots\times I_{\frg_{r}}(K_{r}) .
\]
For each $i\in\{1,\ldots,r\}$ we write $\infty_i$ for the formal
product over real archimedean places $\tau\colon K_i\to\R$ such that
$\R \otimes_{K_{i},\tau}A_{i}$ is a full matrix algebra over $\mathbb{H}$,
and we define the `ray modulo $\frg\infty$' by
\[
P_{\frg}^{+} = P_{\frg}^{+}(C) = \{(\alpha_{i}\mathcal{O}_{K_i})_{i} \in I_{\frg} \mid
  \alpha_{i}\equiv 1 \bmod\!\!^{\times} \frg_{i}\infty_{i} \text{ for all } i=1,\ldots,r \}.
\]
Then $P_{\frg}^{+}$ is a subgroup of $I_{\frg}$ and we have a corresponding
direct product decomposition
\[
P_{\frg}(C)^{+}=P_{\frg_{1}}^{+}(K_{1})\times\cdots\times P_{\frg_{r}}^{+}(K_{r}) .
\]
For each $i$, the quotient $I_{\frg_{i}}(K_{i})/P_{\frg_{i}}^{+}(K_{i})$ is a ray class group of $K_{i}$,
and $I_{\mathfrak{g}}/P_{\mathfrak{g}}^{+}$ is the direct product of these groups.

\begin{theorem}[{\cite[Theorem~1.5]{bley-boltje}}]\label{thm:bley-boltje-thm-isoms}
There are canonical isomorphisms
\[
\Cl(\Lambda, \mathfrak{f})
\cong 
\frac{J(C)}{C^{\times +}\nr(U_{\mathfrak{f}}(\Lambda))}
\cong 
I_{\mathfrak{g}}/P_{\mathfrak{g}}^{+}.
\] 
\end{theorem}

Using these canonical isomorphisms, the exact sequence \eqref{eq:basic-exact-sequence} 
can be rewritten as
\begin{equation}\label{eq:BB-exact-seq}
\left( \Lambda/\frf \right)^{\times} \stackrel{\nu}{\longrightarrow} I_{\frg}/P_{\frg}^{+} 
\stackrel{\mu}{\longrightarrow} \Cl(\Lambda) \lra 0.
\end{equation}

\begin{prop}[{\cite[Proposition~1.8]{bley-boltje}}]\label{prop:nu-explicit}
Let $x \in  \left( \Lambda/\frf \right)^\times $ and let $a \in \Lambda \cap A^{\times}$ such that 
$x \equiv a \bmod{\frf}$. 
Then $\nu(x)  = (\nr(a)\mathcal{O}_{C} \bmod P_{\mathfrak{g}}^{+})$
in $I_\frg/P_\frg^+$.
\end{prop}

Let $x \in  \left( \mathcal{M}/\frf \right)^{\times}$ and let $a \in \mathcal{M} \cap A^{\times}$ such that 
$x \equiv a \bmod{\frf}$. 
Then define
\begin{equation}\label{eq:def-tilde-nu}
\tilde{\nu}(x) := (\nr(a)\mathcal{O}_{C} \bmod P_{\mathfrak{g}}^{+})  \in I_\frg/P_\frg^+. 
\end{equation}
Thus we obtain a group homomorphism
$\tilde{\nu} : (\mathcal{M}/\mathfrak{f})^{\times} 
\longrightarrow I_{\mathfrak{g}} / P_{\mathfrak{g}}^{+}$. 
By Proposition~\ref{prop:nu-explicit}, $\tilde{\nu}$ extends $\nu$, that is,
$\tilde{\nu}(x)=\nu(x)$ for $x \in (\Lambda/\mathfrak{f})^{\times}$.
Note that
\begin{equation}\label{eq:ker-mu}
\ker(\mu) = \nu((\Lambda/\mathfrak{f})^{\times})=\tilde{\nu}((\Lambda/\mathfrak{f})^{\times}).
\end{equation}

\subsection{A technical condition on $\mathfrak{g}$}
For a maximal ideal $\mathfrak{P}$ of
$\calO_C = \calO_{K_1} \oplus \dotsb \oplus \calO_{K_r}$
and a finitely generated $\OC$-module $M$, we write $M_\frP$ 
for the localization of $M$ at $\frP$. 
Note that for a given $\frP$ there exists a unique $1 \leq i \leq r$ and a maximal 
$\calO_{K_i}$-ideal $\mathfrak{P}_i$ such that
\[
\frP =  \calO_{K_1} \oplus \dotsb \oplus \calO_{K_{i-1}} \oplus \frP_i \oplus \calO_{K_{i+1}} \oplus \dotsb \oplus  \calO_{K_r}.
\]
We can and will henceforth identify $M_\frP$ with the localization $(e_{i}M)_{\frP_i}$ of the 
$\mathcal{O}_{K_i}$-module $e_{i}M$ at $\frP_{i}$.
Note that $M$ is uniquely determined by $M_{\frP}$ for all maximal ideals $\frP$ of $\calO_C$ and that $M_{\frp}$, where $\frp$ is a maximal ideal of $\calO$, is uniquely determined by $M_{\frP}$ for all maximal ideals $\frP$ of $\calO_C$ lying above $\frp$.

We make the following non-standard definition. 
It may be viewed as a technical assumption on $\mathfrak{g}$ 
that is needed in the proof of Theorem~\ref{thm:class-group-rep} below. 
In algorithmic applications, $\mathfrak{g}$
will be adjusted to ensure that it satisfies this assumption.

\begin{definition}
The ideal $\mathfrak{g}$ is said to be \emph{locally radical} if
for each maximal ideal $\mathfrak{p}$ of $\calO$ with $\frg_\frp \ne \calO_{C, \frp}$, 
we have $\frg_\frP \subseteq \frP\calO_{C, \frP}$ for all maximal ideals $\frP$ of 
$\mathcal{O}_C$ with $\frP \mid \frp$. 
\end{definition}

\begin{example}
Let $L/K$ be an extension of number fields, 
let $\mathfrak{p}$ be a maximal ideal of $\mathcal{O}_{K}$ 
and let $\mathfrak{P}$ be a prime ideal of $\mathcal{O}_{L}$ lying over $\mathfrak{p}$. 
We consider the $\calO_K$-order $\Lambda = \calO_K + \frP$ and the (unique) maximal 
$\calO_K$-order $\calM = \OL$ in the $K$-algebra $L$.
Then $A=C=L$ and for $\mathfrak{g} = \mathfrak{f} = \mathfrak{P}$ we have 
$\mathfrak{f}_{\mathfrak{p}} \neq \mathcal{M}_{\mathfrak{p}} = \calO_{C, \frp}$ 
and $\mathfrak{g}$ is locally radical if and only if there
is a unique prime ideal of $\mathcal{O}_{L}$ lying over $\mathfrak{p}$.
\end{example}

\begin{lemma}\label{lem:conG}
The following statements hold:
\begin{enumerate}
\item Let $\mathfrak{p}$ be a maximal ideal of $\calO$.
Then $\frf_\frp = \calM_\frp$ if and only if $\frg_\frp = \calO_{C, \frp}$.
\item Let $\mathfrak{P}$ be a maximal ideal of $\mathcal{O}_{C}$. 
Then $\mathfrak{f}_{\mathfrak{P}} \subseteq \rad(\mathcal{M}_{\mathfrak{P}})$
if and only if $\frg_\frP \subseteq \frP\calO_{C, \frP}$.
\end{enumerate}
\end{lemma}

\begin{proof}
We have $\frg_\frp = (\frf \cap \calO_C)_\frp = \frf_\frp \cap \calO_{C, \frp}$ 
since localizations commute with finite intersections. 
Thus $1 \in \mathfrak{g}_{\mathfrak{p}}$ if and only if $1 \in \mathfrak{f}_{\mathfrak{p}}$,
which immediately implies (i). 
Substituting $\mathfrak{P}$ for $\mathfrak{p}$ in this argument shows that 
$\mathfrak{f}_{\mathfrak{P}}=\mathcal{M}_{\mathfrak{P}}$ if and only if 
$\mathfrak{g}_{\mathfrak{P}} = \calO_{C,\mathfrak{P}}$.
By \cite[(18.3)]{Reiner2003}, we have that $\rad(\mathcal{M}_{\mathfrak{P}})$
is the unique maximal two-sided ideal of $\mathcal{M}_{\mathfrak{P}}$.
Moreover, $\frP\calO_{C, \frP}$ is the unique maximal ideal of $\calO_{C, \frP}$.
Together, these last three facts imply (ii).
\end{proof}

\subsection{A criterion for stable freeness}\label{subsec:criterion-for-sf}

Recall that the maps $\mu$ and $\tilde\nu$ were defined in
\eqref{eq:BB-exact-seq} and \eqref{eq:def-tilde-nu}, respectively,
and that for $\eta \in \mathcal{M}$ we write $\overline{\eta}$ for its image in $\mathcal{M} / \mathfrak{f}$.

\begin{theorem}\label{thm:class-group-rep}
Assume that $\mathfrak{g}$ is locally radical. 
Let $X$ be a locally free left ideal of $\Lambda$.
Suppose that $X + \mathfrak{f} = \Lambda$ and that there exists 
$\beta \in \mathcal{M} \cap A^{\times}$ such that $\mathcal{M} X = \mathcal{M}\beta$.
Then 
\[
\mu(\tilde{\nu}(\overline{\beta}))=\mu(\nr(\beta)\OC \bmod P_\frg^+) =[X]_{\st} \text{ in } \Cl(\Lambda).
\]
\end{theorem}

\begin{proof}
Let $\mathfrak{p}$ range over all maximal ideals of $\mathcal{O}$.
For each $\mathfrak{p}$, we define an element $\alpha_{\mathfrak{p}} \in X_{\mathfrak{p}}$ 
with $X_{\mathfrak{p}} = \Lambda_{\mathfrak{p}} \alpha_{\mathfrak{p}}$
as follows.
If $\mathfrak{f}_{\mathfrak{p}} = \Lambda_{\mathfrak{p}} = \mathcal{M}_{\mathfrak{p}}$ 
then we set $\alpha_{\mathfrak{p}} := \beta$. 
If $\frf_\frp \ne \calM_\frp$, we let $\alpha_\frp' \in X_\frp$ be an arbitrarily chosen element such that $X_\frp = \Lambda_\frp \alpha_\frp'$. 
Then $\Lambda_\frp \alpha_\frp' + \frf_\frp = \Lambda_\frp$, 
so $\overline{\alpha}_\frp' \in \left( \Lambda_\frp / \frf_\frp \right)^{\times}$.
Since $\Lambda_\frp$ is semilocal, the canonical map 
$\Lambda_{\mathfrak{p}}^{\times} \longrightarrow (\Lambda_{\mathfrak{p}}/\mathfrak{f}_{\mathfrak{p}})^{\times}$ is surjective by \cite[Chapter III, (2.9) Corollary]{bass}. 
In particular, there exists $\alpha_\frp'' \in \Lambda_\frp^\times$ such that
$\alpha_\frp'' \equiv \alpha_\frp' \bmod{\frf_\frp}$. 
Then we set $\alpha_\frp := (\alpha_\frp'')^{-1}\alpha_\frp'$
and note that $\alpha_\frp \equiv 1 \bmod{\frf_\frp}$.

We have the following commutative diagram
\[
\begin{tikzcd}
 & J(A) \arrow{d}{\nr} \arrow[bend left=90]{dddr}{\varepsilon} \\
 & J(C) \arrow{dr}  \arrow{dl} \\
\frac{J(C)}{C^{\times +}\nr(U_{\mathfrak{f}}(\Lambda))} \arrow{rr} \arrow{d}{\cong} 
\arrow[bend right=60, swap]{dd}{\psi}
&&  \frac{J(C)}{C^{\times +}\nr(U(\Lambda))} \arrow{d}{\cong} \\
\Cl(\Lambda, \mathfrak{f}) \arrow{rr}{q_{1}}  \arrow{d}{\cong} && \Cl(\Lambda) \\
I_{\mathfrak{g}}/P_{\mathfrak{g}}^{+} \arrow{rru}[swap]{\mu} &&
\end{tikzcd} 
\]
where the arrows in the upper triangle are the canonical projections,
the isomorphism on the right is that of Theorem~\ref{thm:class-group-ideles},
the two isomorphisms on the left are those of Theorem~\ref{thm:bley-boltje-thm-isoms},
and $\psi$ is defined to be their composite.
Recall that $\nr$, $q_{1}$, $\mu$, and $\varepsilon$ are defined in \S \ref{subsec:ideles},
\eqref{eq:basic-exact-sequence}, \eqref{eq:BB-exact-seq}, and 
Remark~\ref{rmk:recover-lattice-from-ideles}, respectively.
The only part of the diagram for which commutativity is not immediately implied by
the definitions and Remark~\ref{rmk:recover-lattice-from-ideles} is the central rectangle,
but this follows from the definitions of 
the isomorphisms as discussed in the proof of \cite[Theorem~1.5]{bley-boltje}.

By Remark~\ref{rmk:recover-lattice-from-ideles} and the commutativity of the diagram
we see that 
\[
\mu \circ \psi 
\left(( \nr(\alpha_\frp))_\frp \bmod C^{\times +} \nr(U_\frf(\Lambda))\right)
=[X]_{\st}.
\]  
We recall the recipe for the computation of $\psi$ from the proof of \cite[Theorem~1.5]{bley-boltje}.
By \cite[Remark~1.6]{bley-boltje} we have $U_\frf(\Lambda) =U_\frf(\calM)$. 
Hence both the domain and codomain of $\psi$ break up into direct sums of components, one for each $i = 1, \ldots, r$. So we fix $i \in \{1, \ldots, r\}$, let $L := K_i$ and
replace $\frg_i$ by $\frg$, $\frf_i$ by $\frf$, $\calM_i$ by $\calM$ and $A_i$ by $A$.

Let $\mathfrak{P}$ range over all maximal ideals of $\mathcal{O}_{L}$.
Let $\gamma = \left( \gamma_\frP \right)_\frP \in J(L)$.
In order to compute $\psi(\gamma \bmod L^{\times +} \nr(U_\frf(\calM)))$
we note that by the weak approximation theorem
\cite[Proposition 1.2.3 (3)]{MR1728313} there exists
$\xi \in L^{\times +}$ (named $\beta$ in \cite{bley-boltje}) such that
\[
v_\frP(\gamma_\frP \xi - 1) \ge v_\frP(\frg) \text{ for all } \frP \text{ with } \frP \mid \frg.
\]
If
$
\fra = \prod_\frP \frP^{v_\frP(\gamma_\frP\xi)},
$
where the product extends over all maximal ideals $\frP$ of $\OL$,
then the class of $\mathfrak{a}$ is a representative of  $\psi(\gamma \bmod L^{\times +} \nr(U_\frf(\calM)))$ in $I_\frg/P_\frg^+$.

Using the canonical isomorphism of \eqref{eq:identification} we write 
$\alpha_\frp = \left( \alpha_\frP \right)_{\frP\mid\frp}$ with 
$\alpha_\frP \in \widehat{A}_{\mathfrak{P}}$ for the local basis elements $\alpha_\frp$ chosen as above.
We will apply the above recipe for $\gamma_\frP = \nr_{A_\frP/L_\frP}(\alpha_\frP)$.
The result of \cite[Corollary~2.4]{bley-boltje}
implies that for $\frP \mid \frg$, we have $\nr_{A_\frP/L_\frP}(\alpha_\frP) \equiv 1 \bmod{\frg_\frP}$,
so we can and do choose $\xi = 1$.

It suffices to show that
\begin{equation}\label{val equality}
    {v_\frP(\nr_{A_\frP/L_\frP}(\alpha_\frP))} = v_\frP(\nr_{A_\frP/L_\frP}(\beta))
\end{equation}
for all maximal ideals $\frP$ of $\OL$. To this end, fix a maximal ideal $\frp$ of $\calO$.

If $\frf_\frp = \calM_\frp$, then $\alpha_\frP = \beta$ for all $\frP \mid \frp$ and \eqref{val equality} is clear.
Now suppose $\frf_\frp \neq \calM_\frp$. First note that we have
\[
\calM_\frp = \calM_\frp X_\frp + \frf_\frp = \calM_\frp\beta + \frf_\frp.
\]
Hence there exists an element $\gamma \in \calM_\frp$ such that 
$\gamma\beta \equiv 1 \bmod{\frf_\frp}$.
It follows that $\gamma\beta \equiv 1 \bmod{\frf_\frP}$ for all $\frP \mid \frp$. 
From \cite[Corollary 2.4]{bley-boltje} we obtain $\nr_{A_\frP/L_\frP}(\gamma\beta) \equiv 1 \bmod{\frg_\frP}$
for all $\frP \mid \frp$. 
Since $\frg_\frp \ne \calO_{C, \frp}$ by Lemma~\ref{lem:conG}, 
the assumption that $\mathfrak{g}$ is locally radical implies that 
$\frg_\frP \subseteq \frP\calO_{C, \frP}$. 
Hence $v_\frP(\nr_{A_\frP/L_\frP}(\beta)) = 0$.
Since $\alpha_\frp \equiv 1 \bmod{\frf_\frp}$ the same argument shows that 
$v_\frP(\nr_{A_\frP/L_\frP}(\alpha_\frP)) = 0$ for all maximal ideals $\frP \mid \frp$.

Hence \eqref{val equality} holds for all maximal ideals $\frP$ of $\OL$.
\end{proof}

\subsection{Choosing $\mathfrak{g}$ first}\label{subsec:choose-g-first}
Recall that in \S \ref{subsec:criteria-setup} and \S \ref{subsec:decomposition}, respectively,
we first chose $\mathfrak{f}$ to be a full two-sided ideal of $\mathcal{M}$ that is contained
in $\Lambda$ and then set $\mathfrak{g} := \mathfrak{f} \cap \mathcal{O}_{C}$. 
In some applications, it is convenient to work the other way round: 
first choose a full ideal $\mathfrak{g}$ of $\mathcal{O}_{C}$ such that $\mathfrak{g} \mathcal{M} \subseteq \Lambda$ and then set $
\mathfrak{f} := \mathfrak{g}\mathcal{M}$. 
The following lemma shows that the latter approach is consistent with the former. 

\begin{lemma}\label{lem:g-is-OC-cap-f}
Let $\mathfrak{g}$ be a full ideal of $\mathcal{O}_{C}$ such that 
$\mathfrak{f} := \mathfrak{g} \mathcal{M} \subseteq \Lambda$.
Then $\mathfrak{f}$ is a full two-sided ideal of $\mathcal{M}$ and
$\mathfrak{f} \cap \mathcal{O}_{C} = \mathfrak{g}$.
\end{lemma}

\begin{proof} 
The first claim is clear. For the second claim, it suffices to show
that $\mathfrak{f}_{i} \cap \mathcal{O}_{K_{i}} = \mathfrak{g}_{i}$
for each $i = 1, \ldots, r$. So we fix $i \in \{1, \ldots, r\}$, let $L := K_i$ and
replace $\frg_i$ by $\frg$, $\frf_i$ by $\frf$, $\calM_i$ by $\calM$ and 
$\mathcal{O}_{C}$ by $\mathcal{O}_{L}$.
By \cite[Theorem 4.21]{Reiner2003}, it suffices to show that 
$(\mathfrak{f} \cap \mathcal{O}_{L})_{\mathfrak{P}} = \mathfrak{g}_{\mathfrak{P}}$ for each maximal ideal $\mathfrak{P}$ of $\mathcal{O}_{L}$. Fix such an ideal $\mathfrak{P}$.
Then $\mathcal{O}_{L,\mathfrak{P}}$ is a discrete valuation ring, and so 
we have $\mathfrak{g}_{\mathfrak{P}} = g_{\mathfrak{P}} \mathcal{O}_{L,\mathfrak{P}}$
for some $g_{\mathfrak{P}} \in \mathcal{O}_{L,\mathfrak{P}}$. Then we have
\begin{align*}
(\mathfrak{f} \cap \mathcal{O}_{L})_{\mathfrak{P}} 
&= \mathfrak{f}_{\mathfrak{P}} \cap \mathcal{O}_{L, \mathfrak{P}}
= \mathfrak{g}_{\mathfrak{P}} \mathcal{M}_{\mathfrak{P}} \cap \mathcal{O}_{L, \mathfrak{P}}
= g_{\mathfrak{P}} \mathcal{M}_{\mathfrak{P}} \cap \mathcal{O}_{L, \mathfrak{P}} \\
&= g_{\mathfrak{P}} \mathcal{M}_{\mathfrak{P}} \cap Z(\mathcal{M}_{\mathfrak{P}})
= g_{\mathfrak{P}} Z(\mathcal{M}_{\mathfrak{P}}) 
= g_{\mathfrak{P}} \mathcal{O}_{L, \mathfrak{P}} 
= \mathfrak{g}_{\mathfrak{P}}.
\end{align*}
The authors are grateful to an anonymous referee for suggesting this proof.
\end{proof}

\section{An algorithm for determining SFC}\label{sec:alg-SFC}

\subsection{Test lattices}
We assume the setup and notation of \S \ref{sec:critsfc}.

\begin{definition}\label{def:gen-Swan}
For $\beta \in \mathcal{M} \cap A^{\times}$
we define $ X_\beta = X_{{\beta}, \Lambda, \mathcal{M}}  = \mathcal{M}\beta \cap \Lambda$,
and call this 
the \textit{test lattice} attached to $\beta$ (and $\Lambda, \mathcal{M}$).
Note that $X_{\beta}$ is a left ideal of $\Lambda$.
\end{definition}

The following lemma determines the localizations of certain test lattices.

\begin{lemma}\label{lem:localization-gen-swan-module}
Suppose that $\mathfrak{g}$ is locally radical.
Let $\beta \in \mathcal{M} \cap A^{\times}$ such that $\mathcal{M}\beta + \mathfrak{f} = \mathcal{M}$.
Let $\mathfrak{p}$ be a maximal ideal of $\mathcal{O}$.
Then 
\[
(X_{\beta})_{\mathfrak{p}} = 
\begin{cases}
\Lambda_{\mathfrak{p}} \beta & \text{ if }  \Lambda_{\mathfrak{p}} = \mathcal{M}_{\mathfrak{p}}, \\
\Lambda_{\mathfrak{p}} & \text{ if }  \Lambda_{\mathfrak{p}} \neq \mathcal{M}_{\mathfrak{p}}.
\end{cases}
\]
Moreover, if $\Lambda_{\mathfrak{p}} \neq \mathcal{M}_{\mathfrak{p}}$ then 
$\beta \in \calM_\frp^\times$. 
\end{lemma}

\begin{proof}
If $\Lambda_{\mathfrak{p}} = \mathcal{M}_{\mathfrak{p}}$
then
$(X_{\beta})_{\mathfrak{p}} = \mathcal{M}_{\mathfrak{p}} \beta \cap \Lambda_{\mathfrak{p}} 
=  \mathcal{M}_{\mathfrak{p}} \beta \cap \mathcal{M}_{\mathfrak{p}} = \mathcal{M}_{\mathfrak{p}} \beta = \Lambda_{\mathfrak{p}} \beta$.

Now suppose that $\Lambda_{\mathfrak{p}} \neq \mathcal{M}_{\mathfrak{p}}$.
Then $\mathfrak{f}_{\mathfrak{p}} \neq \mathcal{M}_{\mathfrak{p}}$ and so by 
the assumption that $\mathfrak{g}$ is locally radical 
and Lemma~\ref{lem:conG} we have that 
$\frf_\frP \subseteq \rad(\calM_\frP)$ for all maximal ideals $\frP$ of $\mathcal{O}_C$ with $\frP \mid \frp$.
Thus
\[
\calM_\frP\beta + \frf_\frP = 
(\mathcal{M} + \mathfrak{f})_{\mathfrak{P}}
= \calM_\frP
\]
and so Nakayama's lemma implies that
$\calM_\frP\beta = \calM_\frP$ for all $\frP \mid \frp$, or equivalently, $\beta \in \calM_\frp^\times$.
Hence $\mathcal{M}_{\mathfrak{p}} \beta = \mathcal{M}_{\mathfrak{p}}$ and so
\[
(X_{\beta})_{\mathfrak{p}} 
= \mathcal{M}_{\mathfrak{p}} \beta \cap \Lambda_{\mathfrak{p}}
= \mathcal{M}_{\mathfrak{p}} \cap \Lambda_{\mathfrak{p}}
= \Lambda_{\mathfrak{p}}. \qedhere
\]
\end{proof}

The following result establishes further properties of certain test lattices.
We recall that the maps $\pi, \mu$ and $\tilde\nu$ were defined in \eqref{eq:def-of-pi}, 
\eqref{eq:BB-exact-seq} and \eqref{eq:def-tilde-nu}, respectively.

\begin{prop}\label{prop:Xbeta}
Suppose that $\mathfrak{g}$ is locally radical.
Let $\overline{\beta} \in  \left( \calM / \frf \right)^{\times}$.
Then for every lift $\beta \in \mathcal{M} \cap A^{\times}$ we have that
\begin{enumerate}
\item $\mathcal{M}\beta + \mathfrak{f} = \mathcal{M}$,
\item $X_{\beta}$ is locally free over $\Lambda$,
\item $\mathcal{M} X_\beta = \mathcal{M} \beta$,
\item $X_{\beta} + \frf = \Lambda$, 
\item $\tilde{\nu}(\overline{\beta}) = (\nr(\beta)\OC \bmod P_\frg^+)$ in $I_\frg/P_\frg^+$,
\item $\mu(\tilde{\nu}(\overline{\beta}))
=[X_{\beta}]_{\st}$ in $\Cl(\Lambda)$,
\item $X_{\beta}$ is stably free over $\Lambda$ if and only if 
$\tilde{\nu}(\overline{\beta}) \in \tilde{\nu}((\Lambda/\mathfrak{f})^{\times}) =  \ker(\mu)$, and
\item $X_{\beta}$ is free over $\Lambda$ if and only if 
there exists $u \in \mathcal{M}^{\times}$ such that $\pi(\overline{u})=\pi(\overline{\beta})$.
\end{enumerate} 
\end{prop}

\begin{proof}
Assertion (i) is equivalent to the hypothesis that $\overline{\beta} \in \left( \calM / \frf \right)^{\times}$.
Assertion (ii) follows trivially from Lemma~\ref{lem:localization-gen-swan-module}.

Let $\mathfrak{p}$ be a maximal ideal of $\mathcal{O}$.
By \cite[Theorem 4.21]{Reiner2003}, it will suffice to prove assertions (iii) and (iv) after localising at 
$\mathfrak{p}$.
If $\Lambda_{\mathfrak{p}} = \mathcal{M}_{\mathfrak{p}}$
then by Lemma~\ref{lem:localization-gen-swan-module} we have that 
$(X_{\beta})_{\mathfrak{p}} =  \Lambda_{\mathfrak{p}} \beta$ and thus
\begin{align*}
(\mathcal{M} X_{\beta})_{\mathfrak{p}} 
&= \Lambda_{\mathfrak{p}} \beta
= \mathcal{M}_{\mathfrak{p}} \beta \quad \text{and}\\
(X_{\beta}+\mathfrak{f})_{\mathfrak{p}} 
&= (X_{\beta})_{\mathfrak{p}} + \mathfrak{f}_{\mathfrak{p}}
= \mathcal{M}_{\mathfrak{p}} \beta + \mathfrak{f}_{\mathfrak{p}}
= (\mathcal{M}\beta + \mathfrak{f})_{\mathfrak{p}} 
= \mathcal{M}_{\mathfrak{p}},
\end{align*}
where the final equality holds by part (i).
If $\Lambda_{\mathfrak{p}} \neq \mathcal{M}_{\mathfrak{p}}$
then by Lemma~\ref{lem:localization-gen-swan-module} we have that 
$(X_{\beta})_{\mathfrak{p}} =  \Lambda_{\mathfrak{p}}$ and $\beta \in \calM_\frp^\times$,
and thus
\begin{align*}
(\mathcal{M} X_{\beta})_{\mathfrak{p}} 
&= \mathcal{M}_{\mathfrak{p}} (X_{\beta})_{\mathfrak{p}}
= \mathcal{M}_{\mathfrak{p}} \Lambda_{\mathfrak{p}} = \mathcal{M}_{\mathfrak{p}} = 
\mathcal{M}_{\mathfrak{p}} \beta \quad \text{and}\\
(X_{\beta} + \frf)_{\mathfrak{p}} 
&= (X_{\beta})_{\mathfrak{p}} + \mathfrak{f}_{\mathfrak{p}}
= \Lambda_{\mathfrak{p}} + \mathfrak{f}_{\mathfrak{p}} = \Lambda_{\mathfrak{p}}. 
\end{align*}

Assertion (v) is just the definition of $\tilde{\nu}$ \eqref{eq:def-tilde-nu}.
Assertion (vi) follows from (ii), (iii) and (iv) together with Theorem~\ref{thm:class-group-rep}. 
We have that $X_{\beta}$ is stably free over $\Lambda$ if and only if $[X_{\beta}]_{\st}=0$
and by \eqref{eq:ker-mu} we have $\ker(\mu) =\tilde{\nu}((\Lambda/\mathfrak{f})^{\times})$. 
Therefore assertion (vii) now follows from (vi).
Finally, assertion (viii) follows from (iii) and (iv) together with Corollary~\ref{cor:central-imsp}.
\end{proof}

\begin{corollary}\label{cor:rep-by-Xbeta}
Suppose that $\mathfrak{g}$ is locally radical.
Let $X$ be a locally free left ideal of $\Lambda$.
Suppose that $X + \mathfrak{f} = \Lambda$ and that there exists 
$\beta \in \mathcal{M} \cap A^{\times}$ such that $\mathcal{M} X = \mathcal{M}\beta$.
Then $\overline{\beta} \in (\mathcal{M}/\mathfrak{f})^{\times}$ and $[X_{\beta}]_{\st}=[X]_{\st}$ 
in $\Cl(\Lambda)$. 
\end{corollary}

\begin{proof}
We have $\mathcal{M} 
= \mathcal{M}(X+\mathfrak{f})
= \mathcal{M}X + \mathfrak{f}
= \mathcal{M}\beta + \mathfrak{f}$,
and so $\overline{\beta} \in (\mathcal{M}/\mathfrak{f})^{\times}$. 
Then by Theorem~\ref{thm:class-group-rep} and Proposition~\ref{prop:Xbeta} (vi) we have
$[X]_{\st} =  \mu(\tilde{\nu}(\overline{\beta})) = [X_{\beta}]_{\st}$ in $\Cl(\Lambda)$.
\end{proof}

\begin{corollary}\label{cor:double-cosets-image}
Suppose that $\mathfrak{g}$ is locally radical.
Let $\beta_{1}, \beta_{2} \in \mathcal{M} \cap A^{\times}$ such that 
$\overline{\beta_{1}}, \overline{\beta_{2}} \in (\mathcal{M}/\mathfrak{f})^{\times}$. 
Suppose that $\overline{\beta_{1}},\overline{\beta_{2}}$ have the same image in 
$\mathcal{M}^{\times} \backslash (\mathcal{M}/\mathfrak{f})^{\times}{}{/}{}(\Lambda/\mathfrak{f})^{\times}$. 
Then the following statements hold:
\begin{enumerate}
\item $X_{\beta_{1}}$ is free over $\Lambda$ if and only if $X_{\beta_{2}}$ is free over $\Lambda$.
\item $X_{\beta_{1}}$ is stably free over $\Lambda$ if and only if $X_{\beta_{2}}$ is stably 
free over $\Lambda$.
\end{enumerate}
\end{corollary}

\begin{proof}
By hypothesis there exist $u \in \mathcal{M}^{\times}$ and $\overline{\alpha}
\in (\Lambda/\mathfrak{f})^{\times}$ such that 
$\overline{\beta_{1}} = \overline{u} \overline{\beta_{2}} \overline{\alpha}$ in 
$(\mathcal{M}/\mathfrak{f})^{\times}$.
Then Proposition~\ref{prop:Xbeta} (viii) implies assertion (i).
By the definition of $\tilde{\nu}$  \eqref{eq:def-tilde-nu}, we have that  
$\tilde{\nu}(\overline{u}) = (\nr(u)\mathcal{O}_{C} \bmod{P_\frg^+})$,  
which is trivial since 
$\nr(\mathcal{M}^{\times}) \subseteq \mathcal{O}_{C}^{\times}$ (see \S \ref{subsec:nr}).
Thus
\[
\tilde{\nu}(\overline{\beta_{1}})
=\tilde{\nu}(\overline{u} \overline{\beta_{2}} \overline{\alpha})
=\tilde{\nu}(\overline{u})\tilde{\nu}(\overline{\beta_{2}})\tilde{\nu}(\overline{\alpha})\\
=\tilde{\nu}(\overline{\beta_{2}})\tilde{\nu}(\overline{\alpha}),
\]
where $\tilde{\nu}(\overline{\alpha}) \in \tilde{\nu}((\Lambda/\mathfrak{f})^{\times})$, 
and so assertion (ii) now follows from Proposition~\ref{prop:Xbeta} (vii).
\end{proof}

\begin{lemma}\label{lem:make-lf-lattice-coprime-to-f}
Let $Y \in g(\Lambda)$. 
Then there exists a locally free left ideal $X$ of $\Lambda$ such that 
$X + \mathfrak{f} = \Lambda$ and $X \cong Y$ as $\Lambda$-lattices.
\end{lemma}

\begin{proof}
Since $Y \in g(\Lambda)$, by Roiter's lemma \cite[(27.1)]{Reiner2003} there exists
a monomorphism $\varphi ~\colon~Y \hookrightarrow \Lambda$ such that 
$(\mathfrak{f} \cap \mathcal{O}) + \mathrm{Ann}_{\mathcal{O}}(\coker \varphi) = \mathcal{O}$.
Then we can take $X=\varphi(Y)$.
\end{proof}

\begin{theorem}\label{thm:SFC-criterion}
Suppose that $\mathcal{M}$ has SFC and that $\mathfrak{g}$ is locally radical.
Let $S \subseteq \mathcal{M} \cap A^{\times}$ be a set of representatives of 
$(\mathcal{M}/\mathfrak{f})^{\times}$.
Let $T = \{ \beta \in S \colon X_{\beta} \text{ is stably free over } \Lambda \}$.
Then $\Lambda$ has SFC if and only if $X_{\beta}$ is free over $\Lambda$ for every $\beta \in T$.
\end{theorem}

\begin{proof}
Suppose that $\Lambda$ has SFC. 
If $\beta \in T$ then
$X_{\beta}$ is stably free over $\Lambda$
and so $X_{\beta}$ is in fact free over $\Lambda$.

Recall from \S \ref{subsec:SFCandLFC} that in order to show that $\Lambda$ has SFC, 
it suffices to show that every stably free $\Lambda$-lattice of rank $1$
is in fact free.
So suppose that  $X$ is a stably free $\Lambda$-lattice of rank $1$, that is, $X \in g(\Lambda)$
and $[X]_{\st} = 0$ in $\Cl(\Lambda)$. 
By Lemma~\ref{lem:make-lf-lattice-coprime-to-f} we can and do suppose without loss of
generality that $X$ is a locally free left ideal of $\Lambda$ such that $X + \mathfrak{f} = \Lambda$.
Moreover, $\mathcal{M}X \cong \mathcal{M} \otimes_{\Lambda} X$ is stably free of rank $1$ over 
$\mathcal{M}$.
Since $\mathcal{M}$ has SFC by hypothesis, there exists $\beta_{1} \in \mathcal{M} \cap A^{\times}$
such that $\mathcal{M}X = \mathcal{M}\beta_{1}$.
Then by Corollary~\ref{cor:rep-by-Xbeta} we have 
$\overline{\beta_{1}} \in (\mathcal{M}/\mathfrak{f})^{\times}$
and $[X_{\beta_{1}}]_{\st}=[X]_{\st}=0$. 
Hence by definition of $S$, there exists $\beta_{2} \in S$ such that
$\overline{\beta_{2}}=\overline{\beta_{1}}$. 
Then $[X_{\beta_{2}}]_{\st}=[X_{\beta_{1}}]_{\st}=0$ by Proposition~\ref{prop:Xbeta} (vi)
and so in fact $\beta_{2} \in T$.
Thus $X_{\beta_{2}}$ is free over $\Lambda$ by the assumption on $T$.
By Proposition~\ref{prop:Xbeta}~(viii) this implies that there exists $u \in \mathcal{M}^{\times}$
such that $\pi(\overline{\beta_{2}})=\pi(\overline{u})$. Since $\pi(\overline{\beta_{1}}) = \pi(\overline{\beta_{2}})$, we have $\pi(\overline{\beta_{1}})=\pi(\overline{u})$. 
Thus $X$ is free over $\Lambda$ by Corollary~\ref{cor:central-imsp}.
\end{proof}

\begin{corollary}\label{cor:SFC-criterion}
Suppose that $\mathcal{M}$ has SFC and that $\mathfrak{g}$ is locally radical.
Let $S \subseteq \mathcal{M} \cap A^{\times}$ be a set of representatives of 
$\mathcal{M}^{\times} \backslash (\mathcal{M}/\mathfrak{f})^{\times}{}{/}{}(\Lambda/\mathfrak{f})^{\times}$.
Let 
\[
T = \{ \beta \in S \colon (\nr(\beta)\mathcal{O}_{C} \bmod P_{\mathfrak{g}}^{+}) 
\in \tilde{\nu}((\Lambda/\mathfrak{f})^{\times})\}.
\]
Then $\Lambda$ has SFC if and only if $X_{\beta}$ is free over $\Lambda$ for every $\beta \in T$.
\end{corollary}

\begin{proof}
By Proposition~\ref{prop:Xbeta}~(v),(vii) we have 
$T = \{ \beta \in S \colon X_{\beta} \text{ is stably free over } \Lambda \}$. 
Thus the desired result follows easily from 
Theorem~\ref{thm:SFC-criterion} and Corollary~\ref{cor:double-cosets-image}.
\end{proof}

\subsection{An algorithm for determining SFC}
The following algorithm determines whether an order has SFC.
However, it is only practical for orders in algebras of `small' dimension and
in many situations, it is often much faster to use  Algorithm~\ref{alg:SFC-fail} or 
\ref{alg:sfc-fiberproduct}.

\begin{algorithm}\label{alg:SFC-naive}
Let $K$ be a number field with ring of integers $\mathcal{O}=\mathcal{O}_{K}$, 
and let $A$ be a finite-dimensional semisimple $K$-algebra.
Let $\Lambda$ be an $\mathcal{O}$-order in $A$.
If the following algorithm returns \ensuremath{\mathsf{true}}, then $\Lambda$ has SFC.
If it returns \ensuremath{\mathsf{false}}, then $\Lambda$ fails SFC.
\renewcommand{\labelenumi}{(\arabic{enumi})}
\begin{enumerate}
\item Compute the center $C$ of $A$ and the decomposition of $A$ into simple $K$-algebras.
\item Compute a maximal $\mathcal{O}$-order $\mathcal{M}$ in $A$ containing $\Lambda$.
\item \label{alg-SFC-tf-M-check-step} Check whether $\mathcal{M}$ has SFC.
If not, return \ensuremath{\mathsf{false}}. 
If so and $\Lambda=\mathcal{M}$ return \ensuremath{\mathsf{true}}.
Otherwise, continue with the steps below.
\item Compute 
$\mathfrak{f} := \{ x \in A \mid x\mathcal{M} \subseteq \Lambda \} 
\cap \{ x \in A \mid \mathcal{M}x \subseteq \Lambda \}$.
\item Compute $\mathcal{O}_{C} := \mathcal{M} \cap C$ and 
$\mathfrak{g} := \mathfrak{f} \cap \mathcal{O}_{C} = \mathfrak{f} \cap C$.
\item \label{alg-SFC-tf-M-adjust-step}
Adjust $\mathfrak{f}$ (and thus $\mathfrak{g}$) to ensure
that $\mathfrak{g}$ is locally radical.
\item Compute a set of generators of $(\Lambda/\mathfrak{f})^{\times}$.
\item Compute $I_{\mathfrak{g}}/P_{\mathfrak{g}}^{+}$ and
$\tilde{\nu}((\Lambda/\mathfrak{f})^{\times}) \le I_{\mathfrak{g}}/P_{\mathfrak{g}}^{+}$ as finite abelian groups.
\item \label{alg-SFC-tf-M-set-of-reps-step} Compute the set $(\mathcal{M}/\mathfrak{f})^{\times}$ and a set of representatives 
$S \subseteq \mathcal{M} \cap A^{\times}$ of $(\mathcal{M}/\mathfrak{f})^{\times}$.
\item Compute the set
$T = \{ \beta \in S \colon (\nr(\beta)\mathcal{O}_{C} \bmod P_{\mathfrak{g}}^{+}) 
\in \tilde{\nu}((\Lambda/\mathfrak{f})^{\times})\}$.
\item \label{alg-SFC-tf-last-step} For each $\beta \in T$ check whether $X_\beta$ is free over $\Lambda$.
If this is true for all $\beta \in T$, return \ensuremath{\mathsf{true}}. 
Otherwise return \ensuremath{\mathsf{false}}.
\end{enumerate}
\end{algorithm}

\begin{proof}[Proof of correctness]
In step \eqref{alg-SFC-tf-M-check-step}, the correctness of an output \ensuremath{\mathsf{false}} follows from Corollary~\ref{cor:cancellation-over-order} and the correctness of an output 
\ensuremath{\mathsf{true}} is clear. 
By Proposition~\ref{prop:Xbeta} (v),(vii) we have 
$T = \{ \beta \in S \colon X_{\beta} \text{ is stably free over } \Lambda \}$.
Thus in step \eqref{alg-SFC-tf-last-step}, the correctness of the output follows from 
Theorem~\ref{thm:SFC-criterion}.
\end{proof}

\begin{proof}[Further details on each step]
Steps (1)--(5) can be performed using \cite[1.5 B]{Friedl1985}, \cite[Kapitel 3 and 4]{friedrichs},
Corollary~\ref{cor:sfc-alg-maxord}, \cite[Algorithmus (2.16)]{friedrichs} and  \cite[Algorithm 1.5.1]{MR1728313}, respectively.

Step (6): First note that for a given ideal $\frg$
there are only finitely many maximal ideals $\frp$ of $\calO$ such that
$\frg_\frp \ne \calO_{C, \frp}$. 
If $\frg = \frf \cap \OC$ is not locally radical then
there exists such a $\mathfrak{p}$ and a maximal ideal $\mathfrak{P}$ of $\mathcal{O}_{C}$ with $\frP \mid \frp$ such that 
$\frg_\frP \not \subseteq \frP\calO_{C, \frP}$.
For each such $\mathfrak{P}$, 
we replace $\frf$ by $\frP\frf$ and observe that $(\frP\frf \cap \OC)_\frP = \frP\frf_\frP \cap \calO_{C, \frP} \subseteq \frP\left(\frf_\frP \cap \calO_{C, \frP}\right) \subseteq \frP\calO_{C, \frP}$.

Step (7) can be performed using the algorithm described in \cite[\S 6.6]{MR4493243}
(this is a minor modification of \cite[\S 3.4--3.7]{bley-boltje}). 

Step (8): In the notation of \S \ref{subsec:review-BB}, compute $\infty_{i}$ for $i=1, \ldots, r$
using \cite[Theorem~3.5]{Nebe2009}.
The structure of the finitely generated abelian group $I_{\mathfrak{g}}/P_{\mathfrak{g}}^{+}$ can 
then be computed componentwise using \cite[Algorithm 4.3.1]{MR1728313}.
Using the output of Step~(7) and the definition of $\tilde{\nu}$, 
then determine generators of $\tilde{\nu}((\Lambda/\mathfrak f)^\times)$. 
  
Steps (9) and (10):
Since both rings $\calM/\frf$ and $\Lambda/\frf$ are finite, their units can be computed as sets 
in finite time. (For a more practical approach, see \S \ref{sec:primary-decomposition}).
Given $\beta \in S$, algorithms for finitely generated abelian groups, as presented for example in \cite[\S 4.1]{MR1728313}, can be used
to determine whether $\tilde{\nu}(\overline \beta) =
(\nr(\beta)\mathcal{O}_{C} \bmod P_{\mathfrak{g}}^{+})$ belongs to $\tilde{\nu}((\Lambda/\frf)^\times)$.  
  
Step~(11): Given $\beta \in T$, \cite[Algorithm 3.1]{BJ11} 
(see also \cite[Algorithm 4.1]{MR4136552})
can be used to determine
whether $X_{\beta}$ is free over $\Lambda$ (though stated for orders in group algebras,
this algorithm can also be applied to orders in arbitrary finite-dimensional semisimple algebras
as explained in \cite[Remark 4.9]{BJ11}). 
Note that this algorithm requires certain hypotheses, which ensure that one can find the preimage of the image of the canonical map 
$\mathcal{M}^{\times} \longrightarrow (\mathcal{M}/\mathfrak{f})^{\times}$, 
as well as a generator of $\mathcal{M}X_{\beta}$ as an $\mathcal{M}$-module.
Since generators of $\mathcal{M}^{\times}$ can in principle be found using the algorithms of~\cite{Braun2015},
and since $\mathcal{M}X_{\beta} = \mathcal{M}\beta$ by Proposition~\ref{prop:Xbeta}~(iii),
these hypotheses are in fact unnecessary in this situation.
Note that we also give an alternative freeness testing algorithm in \S \ref{sec:hybrid-method} that 
may be more efficient under certain circumstances.
\end{proof}

\begin{remark}
Steps (4)--(6) of Algorithm~\ref{alg:SFC-naive} can be replaced by the following
steps.
\begin{itemize}
\item[(4)$'$] Compute $\mathcal{O}_{C} := \mathcal{M} \cap C$ and 
$\mathfrak{g} := \{ x \in C \mid x\mathcal{M} \subseteq \Lambda \}$.
\item[(5)$'$] If necessary, adjust $\mathfrak{g}$ to ensure that it is locally radical.
\item[(6)$'$] Compute $\mathfrak{f} := \mathfrak{g}\mathcal{M}$.
\end{itemize}
Then by Lemma~\ref{lem:g-is-OC-cap-f} the ideals $\mathfrak{f}$ and $\mathfrak{g}$
would satisfy $\mathfrak{g} = \mathfrak{f} \cap \mathcal{O}_{C} = \mathfrak{f} \cap C$. 
However, the disadvantage of this approach is that in general it will produce a smaller
choice of $\mathfrak{f}$ and thus larger choices of the sets $S$ and $T$. 
\end{remark}

\begin{remark}
Algorithm~\ref{alg:SFC-naive} can be modified to use Corollary~\ref{cor:SFC-criterion}
instead of Theorem~\ref{thm:SFC-criterion}. 
This involves replacing step \eqref{alg-SFC-tf-M-set-of-reps-step} by
\begin{itemize}
\item[(9)$'$] Compute a set of representatives $S \subseteq \mathcal{M} \cap A^{\times}$ of
$\mathcal{M}^{\times} \backslash (\mathcal{M}/\mathfrak{f})^{\times}{}{/}{}(\Lambda/\mathfrak{f})^{\times}$.
\end{itemize}
In principle, this version of the algorithm should be much more efficient since the sets $S$ and $T$ become much smaller. However, Step (9)$'$ itself comes with its own challenges.
\end{remark}

\section{Orders without SFC}\label{sec:withoutSFC}

\subsection{An algorithm to show that an order fails SFC}\label{subsec:alg-fail-SFC}
The following algorithm can be viewed as a modified version of Algorithm~\ref{alg:SFC-naive}
that cannot prove that an order has SFC, but can be used to show that it fails SFC.

\begin{algorithm}\label{alg:SFC-fail}
Let $K$ be a number field with ring of integers $\mathcal{O}=\mathcal{O}_{K}$, 
and let $A$ be a finite-dimensional semisimple $K$-algebra.
Let $\Lambda$ be an $\mathcal{O}$-order in $A$.
If the following algorithm returns \ensuremath{\mathsf{false}}, then $\Lambda$ fails SFC.
If it returns \ensuremath{\mathsf{true}}, then $\Lambda$ is maximal and has SFC.
If it returns \ensuremath{\mathsf{inconclusive}}, then nothing can be said about whether $\Lambda$ has SFC.
\renewcommand{\labelenumi}{(\arabic{enumi})}
\begin{enumerate}
\item Perform steps (1) and (2) of Algorithm~\ref{alg:SFC-naive}.
\item \label{alg-SFC-fail-M-check-step}
Check whether $\mathcal{M}$ has SFC.
If not, return \ensuremath{\mathsf{false}}. 
If so and $\Lambda=\mathcal{M}$ return \ensuremath{\mathsf{true}}.
Otherwise, continue with the steps below.
\item Perform steps (4)--(8) of Algorithm~\ref{alg:SFC-naive}.
\item \label{alg-SFC-fail-compute-random-elms-step} 
Choose $r \geq 1$ and compute `random' elements 
$\beta_{1},\dotsc,\beta_{r} \in \mathcal{M} \cap A^{\times}$ 
such that $\overline{\beta_{i}} \in (\mathcal{M}/\mathfrak{f})^{\times}$
and $(\nr(\beta_{i})\mathcal{O}_{C} \bmod P_{\mathfrak{g}}^{+}) \in \tilde{\nu}((\Lambda/\mathfrak{f})^{\times})$ for all $i \in \{ 1, \ldots, r \}$.
\item \label{alg-SFC-fail-last-step} 
For each $i \in \{ 1, \ldots, r \}$ check whether $X_{\beta_{i}}$ is free over $\Lambda$.
If this is true for all $i \in \{ 1, \ldots, r \}$, return \ensuremath{\mathsf{inconclusive}}.
Otherwise, return \ensuremath{\mathsf{false}}.
\end{enumerate}
\end{algorithm}

\begin{proof}[Proof of correctness]
In step \eqref{alg-SFC-fail-M-check-step}, the correctness of an output \ensuremath{\mathsf{false}} follows from Corollary~\ref{cor:cancellation-over-order} and the correctness of an output 
\ensuremath{\mathsf{true}} is clear. 
Step \eqref{alg-SFC-fail-compute-random-elms-step} 
ensures that $X_{\beta_{i}}$
is stably free over $\Lambda$ for each $i \in \{ 1, \ldots, r \}$
by Proposition~\ref{prop:Xbeta} (v),(vii).
Thus the output of step \eqref{alg-SFC-fail-last-step}  is correct because if there exists 
$i \in \{ 1, \ldots, r \}$ such that $X_{\beta_{i}}$ is not free over $\Lambda$,
then $\Lambda$ clearly fails SFC.
Note that step \eqref{alg-SFC-fail-M-check-step} is in fact optional,
in that the algorithm is still correct even if it is omitted. 
\end{proof}

\begin{proof}[Further details on each step]
Steps (1)--(3) and step \eqref{alg-SFC-fail-last-step} of Algorithm~\ref{alg:SFC-fail} 
can be performed in the same way 
as the corresponding steps  of Algorithm~\ref{alg:SFC-naive}.
To find the `random' elements in step~\eqref{alg-SFC-fail-compute-random-elms-step}, 
one can determine elements of  $\mathcal{M} \cap A^{\times}$
with random coordinates in a fixed basis of bounded height and then check whether they are invertible modulo $\mathfrak{f}$.
Alternatively, one could also determine elements of $\mathcal{M} \cap A^{\times}$ 
generating $(\mathcal{M}/\mathfrak {f})^{\times}$ using the algorithm described
in \cite[\S 6.6]{MR4493243} (which is a minor modification of 
\cite[\S 3.4--3.7]{bley-boltje}) 
and then take random words of bounded length in the generators.
\end{proof}

\begin{remark}
Thanks to the fact that Algorithm~\ref{alg:SFC-fail} does not need to perform steps (9) or (10)
of Algorithm~\ref{alg:SFC-naive}, in practice the former is often much faster than the latter
if one wants to prove that a given order $\Lambda$ fails SFC.
\end{remark}

\begin{remark}\label{rmk:projection-alg-fail}
Due to the invocation of the (expensive) freeness test of~\cite[Algorithm~3.1]{BJ11},
Algorithm~\ref{alg:SFC-fail} is only practical when $\Lambda$ is an order in an algebra 
$A$ of `moderate' dimension. 
However, 
it is also possible to obtain results when $A$ is of larger dimension as follows.
A projection of $K$-algebras $A \to B$ induces a projection of $\Lambda$ onto an 
$\mathcal{O}$-order $\Gamma$ in $B$.
If we can apply Algorithm~\ref{alg:SFC-fail} to show that $\Gamma$ fails SFC,
then Theorem~\ref{thm:cancellation-under-map} shows that $\Lambda$ also fails SFC.
Since we are free to choose $B$, we can take $B$ to be a direct sum of simple components of $A$ 
of small dimension, chosen such that all simple components of $A$ 
that do not satisfy the Eichler condition
are also components of $B$ and Algorithm~\ref{alg:SFC-fail} finishes quickly.
\end{remark}

\subsection{Integral group rings without SFC}\label{subsec:group-rings-fail-SFC}
We now use Algorithm~\ref{alg:SFC-fail} and Remark~\ref{rmk:projection-alg-fail} to
show that certain integral group rings fail SFC. 
Parts of the following result, which is Theorem~\ref{intro-thm:new-group-rings-fail-SFC} from the introduction,
have already been proven by Swan \cite{MR703486} and Chen 
\cite{ChenPhD} (see Remark~\ref{rmk:Swan-Chen-fail-SFC}).

\begin{theorem}\label{thm:new-group-rings-fail-SFC}
Let $G$ be one of the following finite groups:
\begin{enumerate}
\item $Q_{4n} \times C_{2}$ for $2 \leq n \leq 5$,
\item $\tilde{T} \times C_{2}^{2}$, 
$\tilde{O} \times C_{2}$, $\tilde{I} \times C_{2}^{2}$, or
\item $G_{(32,14)}$, $G_{(36,7)}$, $G_{(64,14)}$, $G_{(100,7)}$. 
\end{enumerate}
Then $\Z[G]$ fails SFC but $\Z[G/N]$ has SFC for every proper quotient $G/N$.
\end{theorem}

\begin{remark}\label{rmk:Swan-Chen-fail-SFC}
Theorem~\ref{thm:new-group-rings-fail-SFC} agrees with previously known results.
Recall from \S \ref{sec:review-canc-int-grp-rings}
that Swan \cite[Theorem 15.1]{MR703486} already proved that $\Z[Q_{8} \times C_{2}]$ fails SFC and Chen \cite{ChenPhD} already proved that $\Z[G]$ fails SFC for
$G=Q_{12} \times C_{2}$, $Q_{16} \times C_{2}$, $Q_{20} \times C_{2}$, $G_{(36,7)}$, $G_{(100,7)}$.
\end{remark}

\begin{remark}
In order, the groups listed in Theorem~\ref{thm:new-group-rings-fail-SFC} (i)
are isomorphic to $G_{(16,12)}$, $G_{(24,7)}$, $G_{(32,41)}$, $G_{(40,7)}$,
and those in (ii) to $G_{(96,198)}$, $G_{(96,188)}$, $G_{(480,960)}$.
\end{remark}

\begin{proof}[Proof of Theorem~\ref{thm:new-group-rings-fail-SFC}]
We first show that $\Z[G]$ fails SFC.
Further details on the computations used in this part of the proof can be found in 
Appendix~\ref{appendix}.
Let $e$ be the central idempotent of $\Q[G]$ such that $e\Q[G]$ is the direct sum
of all simple components of $\Q[G]$ that are
either commutative or totally definite quaternion algebras. 
For each group $G$ listed apart from $G=\tilde{O} \times C_{2}$, 
verification that $e\Z[G]$ fails SFC was performed using Algorithm~\ref{alg:SFC-fail}.
Theorem~\ref{thm:cancellation-under-map} then implies that $\Z[G]$ also fails SFC.
For $G=\tilde{O} \times C_{2}$ there exists a normal subgroup $N$ of order $2$
such that $G/N \cong S_{4} \times C_{2}$. For this choice of $G$ and $N$,
the verification that the order $\Gamma := \Z[G]/\mathrm{Tr}_{N}\Z[G]$ fails SFC was 
performed using Algorithm~\ref{alg:SFC-fail}; again,
Theorem~\ref{thm:cancellation-under-map} then implies that $\Z[G]$ also fails SFC. 

We now show that $\Z[G/N]$ has SFC for every proper quotient $G/N$.
By Lemma~\ref{lem:quot-group-ring}, it suffices to consider maximal proper quotients $H=G/N$.
If $\Q[H]$ satisfies the Eichler condition, then $\Z[H]$ has LFC 
by Theorem~\ref{thm:Jacobinski-cancellation}, and 
if $H$ is isomorphic to any of
\begin{equation}\label{eq:list-bp-have-lfc} 
Q_{8}, \, Q_{12}, \,  Q_{16}, \, Q_{20}, \, \tilde{T}, \, \tilde{O}, \, \tilde{I},
\end{equation}
then $\Z[H]$ has LFC by Theorem~\ref{thm:bp-canc-group-rings}.
Therefore it suffices to consider maximal quotients $H$ such that $\Q[H]$ does not satisfy
the Eichler condition and is not isomorphic to any of the groups in \eqref{eq:list-bp-have-lfc}.
The only groups $G$ listed that have such maximal quotients are $\tilde{T} \times C_{2}^{2}$,
$\tilde{I} \times C_{2}^{2}$,
$G_{(32,14)}$ and $G_{(64,14)}$.
For $G=\tilde{T} \times C_{2}^{2}$, each such $H$ is isomorphic to $\tilde{T} \times C_{2}$,
and so $\Z[H]$ has SFC by Theorem~\ref{thm:tildeT-times-C2-has-SFC}.
For $G=\tilde{I} \times C_{2}^{2}$, each such $H$ is isomorphic to $\tilde{I} \times C_{2}$,
and so $\Z[H]$ has SFC by Theorem~\ref{thm:new-groups-with-SFC}.
For $G=G_{(32,14)}$, each such $H$ is isomorphic to $G_{(16,4)} \cong C_{4} \rtimes C_{4}$
and so $\Z[H]$ has SFC by Theorem~\ref{thm:Nicholson-C}.
For $G=G_{(64,14)}$, each such $H$ is isomorphic to $G_{(32,10)} \cong Q_{8} \rtimes C_{4}$.
Since $G_{(32,10)}$ has a quotient isomorphic to $Q_{16}$ and 
$m_{\mathbb{H}}(G_{(32,10)})=m_{\mathbb{H}}(Q_{16})=2$, the group ring $\Z[G_{(32,10)}]$ 
has SFC by Theorem~\ref{thm:Nicholson-B}.
\end{proof}

\begin{remark}
Let $G$ be any group listed in Theorem~\ref{thm:new-group-rings-fail-SFC}
and let $\mathcal{M}$ be any maximal $\Z$-order in $\Q[G]$. 
Then every binary polyhedral quotient of $G$ is isomorphic to 
$Q_{8}$, $Q_{12}$, $Q_{16}$, $Q_{20}$, $\tilde{T}$, $\tilde{O}$, or $\tilde{I}$,
and hence $\mathcal{M}$ has LFC (and thus SFC) by Corollary~\ref{cor:LFC-max-order-group-rings}.
\end{remark}

The following is the most general result to date on integral group rings that fail SFC.

\begin{corollary}\label{cor:group-ring-fail-SFC}
Let $G$ be a finite group with a quotient isomorphic to a group listed in 
Theorem~\ref{thm:new-group-rings-fail-SFC}
or to $Q_{4n}$ for some $n \geq 6$. 
Then $\Z[G]$ fails SFC. 
\end{corollary}

\begin{proof}
This follows from Theorem~\ref{thm:new-group-rings-fail-SFC} and Corollary~\ref{cor:ZG-fails-SFC-when-large-Q4n-quotient}.
\end{proof}

\section{Determining SFC using fiber products}\label{sec:det-SFC-fiber-prods}

The main result of this section is Algorithm~\ref{alg:sfc-fiberproduct}, which
may be viewed as a constructive version of Corollary~\ref{cor:Eichler-W-consequences}.

\subsection{Setup}
Let $K$ be a number field with ring of integers $\mathcal{O}=\mathcal{O}_{K}$, 
and let $A$ be a finite-dimensional semisimple $K$-algebra.
Let $\Lambda$ be an $\mathcal{O}$-order in $A$. 
Let $I$ and $J$ be two-sided ideals of $\Lambda$ such that $I \cap J = 0$
and $K(I+J)=A$. Let $\Lambda_{1} = \Lambda / I$, let $\Lambda_{2} = \Lambda / J$
and let $\overline{\Lambda} = \Lambda / (I+J)$.
Then $\Lambda_{1}$ and $\Lambda_{2}$ are $\mathcal{O}$-orders
and $\overline{\Lambda}$ is an $\mathcal{O}$-torsion $\mathcal{O}$-algebra.
Then as in Example~\ref{example:fiber-ideals}, we have a fiber product 
\begin{equation}\label{eq:fiber-product-IJ}
\begin{tikzcd}
\Lambda \arrow{d}[swap]{\pi_{2}} \arrow{r}{\pi_{1}} & \Lambda_{1} \arrow{d}{g_{1}}  \\
\Lambda_{2} \arrow{r}[swap]{g_{2}} & \overline{\Lambda},
\end{tikzcd}
\end{equation}
where each of $\pi_{1},\pi_{2},g_{1}, g_{2}$ is the relevant canonical surjection.

Let $A_{1} := K\Lambda_{1} \cong KJ$ and let $A_{2} := K\Lambda_{2} \cong KI$. 
Then by making the appropriate identifications, we can and do assume that 
\begin{equation}\label{eq:identification2}
\Lambda = \{ (x_{1},x_{2}) \in \Lambda_{1} \oplus \Lambda_{2} \mid g_{1}(x_{1})=g_{2}(x_{2}) \} \subseteq \Lambda_1 \oplus \Lambda_2 \subseteq A_1 \oplus A_2 = A. 
\end{equation}
In particular, $\pi_{i}$ is the map induced by the canonical projection $A \to A_{i}$ for $i=1,2$.

Recall from Proposition~\ref{prop:properties-of-Mu} that for $u \in \overline{\Lambda}^{\times}$,
we define 
\[
M(u) := \{ (x_{1},x_{2}) \in \Lambda_{1} \oplus \Lambda_{2} \mid g_{1}(x_{1})u=g_{2}(x_{2}) \}.
\]
In particular, this is a locally free $\Lambda$-lattice of rank $1$.
We refer to such a lattice $M(u)$ as a \emph{pullback lattice}. 
Recall from Corollary~\ref{cor:properties-of-Mu} that we have a group homomorphism
\[
\partial : \overline{\Lambda}^{\times} \longrightarrow \Cl(\Lambda), 
\quad u \mapsto [M(u)]_{\st}.
\]

We shall relate pullback lattices to test lattices and
use this to give an alternative description of $\ker(\partial)$.
Under the further assumption that $A_{1}$ satisfies the Eichler condition, 
this will lead to an algorithm to test whether $\Lambda$ has SFC.

\subsection{Test lattices and pullback lattices}\label{subsec:rel-to-gen-Swan}
For $i=1,2$, let $\mathcal{M}_{i}$ be a choice of maximal 
$\mathcal{O}$-order in $A_{i}$ containing
$\Lambda_{i}$ (such choices exist by \cite[(10.4)]{Reiner2003}) and 
let $\mathfrak{f}_{i}$ be any full two-sided ideal of $\mathcal{M}_{i}$ that 
is contained in $\Lambda \cap A_{i}$.
Let $\mathcal{M} = \mathcal{M}_{1} \oplus \mathcal{M}_{2}$ and let 
$\mathfrak{f} = \mathfrak{f}_{1} \oplus \mathfrak{f}_{2}$.
For $i=1,2$, let $C_{i}$ be the center of $A_{i}$ and let $\mathfrak{g}_{i} := \mathfrak{f}_{i} \cap \mathcal{O}_{C_{i}} =\mathfrak{f}_{i} \cap C_{i}$. 
Let $C = C_{1} \oplus C_{2}$ and let $\mathfrak{g}=\mathfrak{g}_{1} \oplus \mathfrak{g}_{2}$. 
We set
\[
\frh_1 := \ker(g_1), \quad \frh_2 := \ker(g_2),  \quad \frh := \pi_1^{-1}(\frh_1) = \pi_2^{-1}(\frh_2) = I+J,
\] 
and observe that
\[
\overline{\Lambda} = \Lambda/\frh \cong \Lambda_1/\frh_1 \cong \Lambda_2/\frh_2.
\]
Using the identification in \eqref{eq:identification2} we have $\frh_{i} = \Lambda \cap A_{i}$ for $i=1,2$.

Since for any $x_{1} \in \Lambda_{1}$ with $(x_{1}, 0) \in \Lambda$ we have $x_{1} \in \ker(g_{1})$,
it follows that $\frf_{1} \subseteq \frh_{1}$.
Similarly, $\frf_{2} \subseteq \frh_{2}$ and therefore $\frf \subseteq \frh$.
Recall from Definition~\ref{def:gen-Swan} that for $\beta \in \mathcal{M} \cap A^{\times}$, the test lattice $X_{\beta}=X_{\beta, \Lambda, \mathcal{M}}$ is defined to be
$\mathcal{M} \beta \cap \Lambda$.

The following result shows that under mild assumptions, pullback lattices are in fact
(isomorphic to) test lattices. 

\begin{prop}\label{prop:mu-gen-swan}
Suppose that $\mathfrak{g}$ is locally radical.
Let $\beta = (\beta_{1}, \beta_{2}) \in \Lambda_{1} \oplus \Lambda_{2}$ with 
$\beta \in A^{\times}$ and $\overline{\beta_{i}} \in \left( \Lambda_i / \frf_i \right)^\times$ for $i=1,2$. 
Let $u = g_{1}(\beta_{1}) g_{2}(\beta_{2})^{-1} \in \overline{\Lambda}^{\times}$.
Then $X_{\beta} = M(u) \beta$, so in particular,
$X_{\beta} \cong M(u)$ as $\Lambda$-lattices.
\end{prop}

\begin{proof}
Let $\mathfrak{p}$ be a maximal ideal of $\mathcal{O}$.
By \cite[Theorem 4.21]{Reiner2003}, it suffices to prove that
$(X_{\beta})_{\mathfrak{p}} = (M(u) \beta)_{\mathfrak{p}}$.
Note that $(M(u) \beta)_{\mathfrak{p}} = M(u)_{\mathfrak{p}} \beta$.

Suppose that $\Lambda_{\mathfrak{p}}=\mathcal{M}_{\mathfrak{p}}$.
Since 
$\Lambda_{\mathfrak{p}} \subseteq \Lambda_{1,\mathfrak{p}} \oplus \Lambda_{2, \mathfrak{p}}$ 
and $\Lambda_{\mathfrak{p}}$ is maximal, we have that 
$\Lambda_{\mathfrak{p}} =  \Lambda_{1, \mathfrak{p}} \oplus \Lambda_{2, \mathfrak{p}}$.
Thus localising the fiber product \eqref{eq:fiber-product} at $\mathfrak{p}$ 
implies that $\frh_{i,\frp} = \calM_{i,\frp} = \Lambda_{i, \frp}$ for $i=1,2$ and 
$\overline{\Lambda}_{\mathfrak{p}} = 0$.
It is then easy to see that
$M(u)_\frp = \Lambda_{1, \frp} \oplus \Lambda_{2, \frp} =  \calM_\frp$ and hence 
$M(u)_{\mathfrak{p}} \beta = \mathcal{M}_{\mathfrak{p}} \beta = (X_{\beta})_{\mathfrak{p}}$,
where the last equality follows from Lemma~\ref{lem:localization-gen-swan-module}.

Now suppose that $\Lambda_{\mathfrak{p}} \neq \mathcal{M}_{\mathfrak{p}}$.
Then $\beta \in \mathcal{M}_{\mathfrak{p}}^{\times}$ by
Lemma~\ref{lem:localization-gen-swan-module}.
Thus $\beta_{i} \in \mathcal{M}_{i,\mathfrak{p}}^{\times}$
and so in fact 
$\beta_{i}
\in \Lambda_{i,\mathfrak{p}} \cap \mathcal{M}_{i,\mathfrak{p}}^{\times}
= \Lambda_{i,\mathfrak{p}}^{\times}$ (see \cite[Chapter I, Theorem 7.9(iii)]{MR1183469}).
Then we have
\begin{align*}
M(u)_{\mathfrak{p}} \beta 
&=  
\{ (a_{1} \beta_{1}, a_{2}\beta_{2}) \mid (a_{1},a_{2}) 
\in \Lambda_{1,\mathfrak{p}} \oplus \Lambda_{2,\mathfrak{p}}
\text{ and } g_{1, \mathfrak{p}}(a_1) u = g_{2, \mathfrak{p}}(a_2) \} \\
&= \{ (a_{1} \beta_{1}, a_{2}\beta_{2}) \mid (a_{1},a_{2}) 
\in \Lambda_{1,\mathfrak{p}} \oplus \Lambda_{2,\mathfrak{p}} \text{ and } 
g_{1, \mathfrak{p}}(a_{1}\beta_{1}) = g_{2, \mathfrak{p}}(a_{2}\beta_{2}) \} \\
&= \{ (b_{1}, b_{2}) \mid (b_{1},b_{2}) 
\in \Lambda_{1,\mathfrak{p}} \oplus \Lambda_{2,\mathfrak{p}} \text{ and } 
g_{1, \mathfrak{p}}(b_{1}) = g_{2, \mathfrak{p}}(b_{2}) \} \\
&= \Lambda_{\mathfrak{p}} = (X_{\beta})_{\mathfrak{p}},
\end{align*}
where the last equality follows from Lemma~\ref{lem:localization-gen-swan-module}.
\end{proof}

\subsection{Alternative description of $\ker(\partial)$}
Recall from \eqref{eq:BB-exact-seq} and \eqref{eq:def-tilde-nu} respectively
that we have an exact sequence
\[
\left( \Lambda/\frf \right)^{\times} \stackrel{\nu}{\longrightarrow} I_\frg/P_\frg^+ 
\stackrel{\mu}{\longrightarrow} \Cl(\Lambda) \lra 0,
\]
and that the reduced norm map induces a group homomorphism
\[
\tilde{\nu} : (\mathcal{M}/\mathfrak{f})^{\times} \longrightarrow I_{\mathfrak{g}} / P_{\mathfrak{g}}^{+}\] 
that extends $\nu$. 
Therefore
$(I_{\mathfrak{g}}/P_{\mathfrak{g}}^{+})/\tilde{\nu}((\Lambda/\mathfrak{f})^{\times}) \cong \Cl(\Lambda)$.
Let
\[
\tilde{\mu} : I_{\mathfrak{g}}/P_{\mathfrak{g}}^{+} \longrightarrow (I_{\mathfrak{g}}/P_{\mathfrak{g}}^{+})/\tilde{\nu}((\Lambda/\mathfrak{f})^{\times})
\]
denote the canonical surjection.
Define 
\begin{equation}\label{eq:nu_2-def}
\nu_{2}' \colon (\Lambda_{2}/\mathfrak{h}_{2})^{\times} 
\longrightarrow 
(I_{\mathfrak{g}}/P_{\mathfrak{g}}^{+})/\tilde{\nu}((\Lambda/\mathfrak{f})^{\times}),
\quad 
x_{2} \mapsto
\tilde{\mu}(\tilde{\nu}((\overline{1}, \overline{\alpha_{2}}))), 
\end{equation}
where $\alpha_{2} \in \Lambda_{2} \cap A_{2}^{\times}$ is any element such that 
$\overline{\alpha_{2}} \in (\Lambda_{2}/\mathfrak{f}_{2})^{\times}$
and $x_{2} = (\alpha_{2} \bmod \mathfrak{h}_{2})$.

\begin{lemma}\label{lem:kerpartial}
The map $\nu_{2}'$ is a well-defined group homomorphism.
If $\mathfrak{g}$ is locally radical
then $\ker(\partial) = {\tilde{g}_{2}}(\ker(\nu'_{2}))$,
where $\tilde{g}_{2} : \Lambda_{2}/\frh_{2} \cong \overline{\Lambda}$
is the ring isomorphism induced by $g_{2}$.
\end{lemma}

\begin{proof}
Let $\iota_{2} \colon (\Lambda_{2}/\mathfrak{f}_{2})^{\times} \longrightarrow (\mathcal{M}/\mathfrak{f})^{\times}$
and $\theta_{2} \colon (\Lambda_2/\mathfrak f_2)^\times \longrightarrow (\Lambda_2/\mathfrak h_2)^{\times}$ be the canonical group homomorphisms 
(recall from \S \ref{subsec:rel-to-gen-Swan} that $\frf_{2} \subseteq \frh_{2}$).
Then $\iota_{2}$ is clearly injective and $\theta_{2}$ is surjective by \cite[Chapter III, (2.9) Corollary]{bass}. 
Moreover, it follows from the definition of $\nu_{2}'$ that
$\tilde{\mu} \circ \tilde{\nu} \circ \iota_{2} = \nu_{2}' \circ \theta_{2}$.
Every element in $\ker(\theta_{2})$ is represented by an element of $\Lambda_{2}$ of the form 
$\lambda_{2}=1 + h_{2}$ with $h_{2} \in \frh_{2}$.
Since 
\[
(1,\lambda_{2}) = (1, 1) + (0, h_{2}) = 1 + (0, h_{2}) \in \Lambda,
\]
we have
$\tilde{\nu}(\overline{1} , \overline{\lambda_{2}}) \in \tilde \nu((\Lambda/\mathfrak{f})^{\times})$.
Therefore $\nu_{2}'$ is a well-defined group homomorphism.

Let $x_{2} \in \left(\Lambda_2/\frh_2 \right)^{\times}$ and let
$\alpha_{2}$ be as in the definition of the map $\nu_{2}'$.
Then we have
\begin{align*}
x_{2} \in \ker(\nu_{2}') 
& \Longleftrightarrow \ \tilde{\nu}((\overline{1}, \overline{\alpha_{2}})) \in
\tilde{\nu}((\Lambda/\mathfrak{f})^{\times}) & & \\
& \Longleftrightarrow \ [X_{(1, \alpha_2)}]_{\mathrm{st}} = 0 & &
\text{(Proposition~\ref{prop:Xbeta}~(vii))} \\
&\Longleftrightarrow \ [M(g_2(\alpha_2)^{-1})]_{\mathrm{st}} = 0 & & 
\text{(Proposition~\ref{prop:mu-gen-swan})}\\
&\Longleftrightarrow \  \tilde{g}_{2}(\alpha_{2} \bmod \mathfrak{h}_{2} ) \in \ker(\partial). 
\end{align*}
Since $x_{2} = (\alpha_{2} \bmod \mathfrak{h}_{2})$
in $\Lambda_{2}/\frh_{2} \cong \overline{\Lambda}$,
we conclude that $\ker(\partial) = \tilde{g}_{2}(\ker({\nu}_{2}'))$.
\end{proof}

\subsection{An algorithm for determining SFC using fiber products}
The following algorithm is based on Corollary~\ref{cor:Eichler-W-consequences}.

\begin{algorithm}\label{alg:sfc-fiberproduct}
Let $K$ be a number field with ring of integers $\mathcal{O}=\mathcal{O}_{K}$, 
and let $A$ be a finite-dimensional semisimple $K$-algebra.
Let $\Lambda$ be an $\mathcal{O}$-order in $A$ with 
a fiber product as in \eqref{eq:fiber-product-IJ}.
Assume that $A_{1}$ is nontrivial and satisfies the Eichler condition.
If the following algorithm returns \ensuremath{\mathsf{true}}, then $\Lambda$ has SFC.
If it returns \ensuremath{\mathsf{false}}, then $\Lambda$ fails SFC.
\renewcommand{\labelenumi}{(\arabic{enumi})}
\begin{enumerate}
\item Compute the center $C= C_{1} \oplus C_{2}$ of $A=A_{1} \oplus A_{2}$ 
and the decomposition of $A= A_{1} \oplus A_{2}$ into simple $K$-algebras.
\item For $i=1,2$, compute a maximal $\mathcal{O}$-order $\mathcal{M}_{i}$ in $A_{i}$ containing $\Lambda_{i}$.
\item \label{alg-SFC-tf-M-check-step2} Check whether $\mathcal{M}_{2}$ has SFC.
If not, return \ensuremath{\mathsf{false}}.
If so, continue with the steps below.
\item For $i=1,2$, compute 
$\mathfrak{f}_{i} := \{ x \in A_{i} \mid x\mathcal{M}_{i} \subseteq \Lambda_{i} \} 
\cap \{ x \in A_{i} \mid \mathcal{M}_{i}x \subseteq \Lambda_{i} \}$.

\item For $i=1,2$, compute $\mathcal{O}_{C_{i}} := \mathcal{M}_{i} \cap C_{i}$ and 
$\mathfrak{g}_{i} := \mathfrak{f}_{i} \cap \mathcal{O}_{C_{i}} = \mathfrak{f}_{i} \cap C_{i}$.
\item \label{alg-SFC-tf-M-adjust-step2}
For $i=1,2$, adjust $\mathfrak{f}_{i}$ (and thus $\mathfrak{g}_{i}$) to ensure
that $\mathfrak{g}_{i}$ is locally radical.
\item Check whether $\Lambda_{2}$ has SFC.
If not, return \ensuremath{\mathsf{false}}. 
If so, continue with the steps below.
\item Compute $\mathfrak{h}_{2} = \ker(g_{2})$ and $\overline{\Lambda}_{2} = \Lambda_{2} / \mathfrak{h}_{2}$.
\item Compute $\overline{\Lambda}_{2}^{\times}$ as a finite group.
\item Compute $\overline{\Lambda}_{2}^{\times \mathrm{ab}} = \overline{\Lambda}_{2}^\times{}{/}{}[\overline \Lambda_{2}^{\times}, \overline{\Lambda}_{2}^{\times}]$ as a finite abelian group.
\item Compute a set of generators of $(\Lambda/\mathfrak{f})^{\times}$, where 
$\mathfrak{f} = \mathfrak{f}_{1} \oplus \mathfrak{f}_{2}$.
\item Compute $I_{\mathfrak{g}}/P_{\mathfrak{g}}^{+}$ and
$\tilde{\nu}((\Lambda/\mathfrak{f})^{\times}) \le I_{\mathfrak{g}}/P_{\mathfrak{g}}^{+}$ 
as finite abelian groups, where $\mathfrak{g} = \mathfrak{g}_{1} \oplus \mathfrak{g}_{2}$.
\item Compute generators of the subgroup
$\ker(\nu_{2}') \leq \overline{\Lambda}_{2}^{\times}$.
\item Compute elements $\overline{\gamma}_{1}, \ldots, \overline{\gamma}_{r} \in 
\overline{\Lambda}_{2}^{\times}$ such that the set of their images in 
$\overline{\Lambda}_{2}^{\times \mathrm{ab}}$ is a minimal generating set of the 
subgroup $\ker(\nu_{2}')/[\overline{\Lambda}_{2}^{\times}, \overline{\Lambda}_{2}^{\times}] \leq  
\overline{\Lambda}_{2}^{\times \mathrm{ab}}$.
\item For each $i \in \{ 1, \ldots, r \}$, choose an element
$\beta_{i} \in \Lambda_{2} \cap A_{2}^{\times}$ with
$(\beta_{i} \bmod \mathfrak{f}_{2}) \in \left( \Lambda_{2}/\frf_{2} \right)^{\times}$ and 
such that $\overline{\gamma}_{i} = ( \beta_{i} \bmod{\frh_{2}})$.
\item For each $i \in \{ 1, \ldots, r \}$, check whether $X_{(1, \beta_{i})}$ is free over $\Lambda$. 
If this is true for all $i \in \{ 1, \ldots, r \}$, return \ensuremath{\mathsf{true}}.
Otherwise, return \ensuremath{\mathsf{false}}.
\end{enumerate}
\end{algorithm}

\begin{proof}[Proof of correctness]
By Theorem~\ref{thm:Jacobinski-cancellation} and the assumption that $A_{1}$
satisfies the Eichler condition, $\Lambda_{1}$ must have SFC.
Thus by Corollary~\ref{cor:cancellation-direct-product}, $\Lambda_{1} \oplus \Lambda_{2}$
has SFC if and only if $\Lambda_{2}$ has SFC.
Therefore the correctness of an output \ensuremath{\mathsf{false}} in either step (3) or step (7) follows from Corollary~\ref{cor:cancellation-over-order} since 
$\Lambda \subseteq \Lambda_{1} \oplus \Lambda_{2} \subseteq 
\Lambda_{1} \oplus \mathcal{M}_{2}$.

By Lemma~\ref{lem:kerpartial}, we have $\ker(\partial) = {\tilde{g}_{2}}(\ker(\nu'_{2}))$,
where $\tilde{g}_{2} : \Lambda_{2}/\frh_{2} \cong \overline{\Lambda}$
is the ring isomorphism induced by $g_{2}$.
Thus the set of images of $\tilde{g}_{2}(\overline{\gamma}_{1}),\ldots,\tilde{g}_{2}(\overline{\gamma}_{r})$
in $\overline{\Lambda}^{\times \mathrm{ab}}= \overline{\Lambda}^\times{}{/}{}[\overline \Lambda^{\times}, \overline{\Lambda}^{\times}]$ generates the subgroup  $\ker(\partial) /  [\overline{\Lambda}^{\times}, \overline{\Lambda}^{\times}] \leq \overline{\Lambda}^{\times \mathrm{ab}}$.
Moreover, by Proposition~\ref{prop:mu-gen-swan}, for each $i \in \{ 1, \ldots, r \}$, 
we have $X_{(1, \beta_{i})} \cong M(g_{2}(\beta_{i})^{-1})$ as $\Lambda$-lattices.
Therefore the correctness of the output in step (16)
follows from Corollary~\ref{cor:Eichler-W-consequences} with
$N = [\overline{\Lambda}^{\times}, \overline{\Lambda}^{\times}]$.
\end{proof}

\begin{proof}[Further details on each step]
Steps (1)--(6), (11), and (12) can be performed in the same way 
as the corresponding steps  of Algorithm~\ref{alg:SFC-naive}.

In principle, Step (7) can be performed using any appropriate algorithm,
such as Algorithm~\ref{alg:SFC-naive}, for example.
Note that Algorithm~\ref{alg:sfc-fiberproduct} itself can be applied recursively, 
which is useful because the choice of fiber product can have a large
impact on the running time.

Step (8) is straightforward and 
Step (9) can be performed using the methods outlined in \S \ref{subsec:unit-groups-finite-rings}.
Note that, in practice, Step (9) is often the bottleneck in the computation. 
Steps (10), (13), and (14) can be performed using standard methods in computational group theory, as described in \cite{MR2129747}, for example, and the explicit definition of $\nu_{2}'$
given in \eqref{eq:nu_2-def}. 
Step (15) is straightforward and in Step (16), the freeness testing can be performed either as described in Step (11) of Algorithm~\ref{alg:SFC-naive} or the method described in 
\S \ref{sec:hybrid-method}.
\end{proof}

\section{Primary decomposition and unit groups of quotient rings}\label{sec:primary-decomposition}

\subsection{Setup}\label{subsec:primary-setup}
Let $K$ be a number field with ring of integers $\mathcal{O}=\mathcal{O}_{K}$, 
and let $A$ be a finite-dimensional semisimple $K$-algebra.
Let $\Lambda$ be an $\mathcal{O}$-order in $A$ and let
$\mathfrak{h}$ be any full two-sided ideal of $\Lambda$.
In \S \ref{subsec:primary-dec}, we describe a practical algorithm to compute
a set of generators of the unit group $(\Lambda/\mathfrak{h})^{\times}$.
In \S \ref{subsec:unit-groups-finite-rings}, we then describe a practical algorithm
to compute $(\Lambda/\mathfrak{h})^{\times}$ as an abstract group.
These are important subtasks in Algorithm~\ref{alg:sfc-fiberproduct}.

\subsection{Primary decomposition and generators of unit groups of quotient rings}\label{subsec:primary-dec}

Let $C$ be an \'etale $K$-algebra and let $\Gamma$ be an $\calO$-order in $C$.
Let $\frg$ be a proper ideal of $\Gamma$ of full rank.
We now describe the computation of a primary decomposition of 
$\frg$, that is, the computation of coprime primary ideals 
$\mathfrak{q}_{1},\dotsc,\mathfrak{q}_{r}$ of $\Gamma$, such that
\[ 
\frg = \prod_{i=1}^{r} \mathfrak{q}_{i} = \bigcap_{i=1}^{r} \mathfrak{q}_{i}. 
\]

\begin{prop}\label{prop:primdec}
Let $C$ be an \'etale $K$-algebra and let  $\Gamma$ be an $\calO$-order in $C$.
Let $\frg$ be a proper ideal of $\Gamma$ of full rank.
Let $\mathfrak{p}_{1},\dotsc, \mathfrak{p}_{r}$ be the maximal ideals of $\Gamma$ 
containing $\frg$. Then the following hold.
\begin{enumerate}
\item 
There exist exponents $e_{1}, \ldots, e_{r} \in \Z_{>0}$ such that
\[
\frp_1^{e_1} \cdots \frp_r^{e_r} \subseteq \frg \subseteq \frp_1 \cdots \frp_r.
\]
\item
For each $i \in \{1,\dotsc,r \}$, the ideal
$\mathfrak{p}_i^{e_i} + \mathfrak{g}$ is $\mathfrak{p}_i$-primary.
\item
A primary decomposition of $\mathfrak{g}$ is given by
\[ 
\mathfrak{g} 
= \bigcap_{i=1}^{r} (\mathfrak{p}_{i}^{e_i} + \mathfrak{g}) 
= \prod_{i=1}^{r} (\mathfrak{p}_{i}^{e_i} + \mathfrak{g}).
\] 
\end{enumerate}
\end{prop}

\begin{proof}
(i) 
As in the proof of \cite[(3.4)~Lemma]{Neukirch1999} one can show that there is a product of maximal ideals 
$\frq_{i}$ of $\Gamma$ such that $\frq_1 \cdots \frq_m \subseteq \frg$.
We take $m$ to be minimal.
Suppose that one of the $\frq_i$, say $\frq_1$, is not contained in the set $\{\frp_1, \ldots, \frp_r\}$.
Then we claim that $\frq_2\cdots\frq_m \sseq \frg$ contradicting the minimality of $m$.
Indeed, we then have $\frq_1 + \frg = \Gamma$, so that we can write $1 = x+y$ with 
$x \in \frq_1$ and $y \in \frg$. 
For any $z \in \frq_2\cdots\frq_m$ we obtain $z = zx + zy \in \frg$. 
We conclude that there exist $e_1, \ldots, e_r \in \Z_{>0}$ such that 
$\frp_1^{e_1} \cdots \frp_r^{e_r} \subseteq \frg$.

(ii) Since $\mathfrak p_i^{e_i} \subseteq \mathfrak p_i^{e_i} + \mathfrak g \subseteq \frp_i$, the radical of $\mathfrak p_i^{e_i} + \mathfrak g$ is the maximal ideal $\mathfrak p_i$. This implies that $\mathfrak p_i^{e_i} + \mathfrak g$ is $\mathfrak p_i$-primary by \cite[Proposition~(4.2)]{MR242802}.

(iii) By part (ii), it suffices to show that $\frg = \bigcap_{i=1}^r (\frp_i^{e_i} + \frg) = \prod_{i=1}^r (\frp_i^{e_i} + \frg)$. 
The second equality is clear because the ideals $\frp_i^{e_i} + \frg$ are pairwise coprime. 
For the first equality, observe that  
$\prod_{i=1}^r (\frp_i^{e_i} + \frg) \subseteq \frp_1^{e_1} \cdots \frp_r^{e_r} + \frg = \frg$;
the reverse inclusion is obvious.
\end{proof}

\begin{remark}
Let $\mathcal{O}_{C}$ be the (unique) maximal $\mathcal{O}$-order in $C$.
This is equal to the integral closure of $\Gamma$ in $C$.     
By \cite[Theorem 5.10]{MR242802}, each of the maximal ideals $\frp_i$ is of the form 
$\frP_i \cap \Gamma$ for some maximal ideal $\frP_i$ of $\OC$.
Let $\{ \frP_1, \ldots, \frP_n \}$ be the set of maximal ideals of $\OC$ with $\frg\OC \subseteq \frP_i$. Then
\[
\{\frp_1, \ldots, \frp_r \} = \{ \frP_1 \cap \Gamma, \ldots, \frP_n \cap \Gamma \}.
\]
Note that $\frP_1, \ldots, \frP_n$ can be computed using standard methods of 
algorithmic number theory~\cite[Theorem 4.9]{Lenstra1992}.
\end{remark}

\begin{corollary}\label{cor:algprimdec}
Let $C$ be an \'etale $K$-algebra and let  $\Gamma$ be an $\calO$-order in $C$.
Let $\frg$ be a proper ideal of $\Gamma$ of full rank.
Then there exists an algorithm that computes the primary decomposition of $\frg$.
\end{corollary}

\begin{corollary}\label{cor:algunit}
Let $\Lambda$ be an $\mathcal{O}$-order in a finite-dimensional semisimple $K$-algebra $A$
and let $\mathfrak{h}$ be a full two-sided ideal of $\Lambda$. 
Then there exists an algorithm that determines generators of $(\Lambda/\mathfrak{h})^{\times}$.
\end{corollary}

\begin{proof}
The algorithm is trivial in the case $\mathfrak{h}=\Lambda$, and so henceforth
we assume that $\mathfrak{h} \neq \Lambda$.
Let $C$ denote the center of $A$ and let $\mathcal{O}_{C}$ denote the (unique) maximal
$\mathcal{O}$-order contained in $C$. 
Let $\mathfrak{g}=\mathfrak{h} \cap C$. 
Then $\mathfrak{g}$ is a proper ideal of an $\mathcal{O}$-order $\Gamma$, 
which may be strictly contained in $\OC$.
In the case that $\mathfrak{h}$ is equal to $\mathfrak{f}$, a two-sided ideal of 
a maximal $\mathcal{O}$-order $\mathcal{M}$ in $A$ containing $\Lambda$,
the desired algorithm is described in \cite[\S 6.6]{MR4493243}
(this is a minor modification of \cite[\S 3.4--3.7]{bley-boltje}). 
This relies crucially on the computation of the primary decomposition of $\mathfrak{g}$,
which can be accomplished using work of Bley--Endres~\cite{bley-endres} under the assumption that $\Gamma=\OC$ (which must hold when $\mathfrak{h}=\mathfrak{f}$).
In fact, the computation of the primary decomposition of $\mathfrak{g}$ is the only
obstacle to generalising the algorithm to arbitrary choices of $\mathfrak{h}$.
By removing the restriction that $\Gamma=\OC$ on the computation of primary decompositions in 
Corollary~\ref{cor:algprimdec}, we thereby obtain an algorithm for the computation of generators of 
$(\Lambda/\mathfrak{h})^{\times}$
for arbitrary full two-sided ideals $\mathfrak{h}$ of $\Lambda$. 
\end{proof}

\subsection{Group structure of unit groups of finite rings}\label{subsec:unit-groups-finite-rings}
Let $\overline{\Lambda} = \Lambda / \mathfrak{h}$.
Then $\overline{\Lambda}$ is a finite ring and an important challenge is to 
compute its units $\overline{\Lambda}^{\times}$ as an abstract group.
The point is that, although we can compute generators of $\overline{\Lambda}^{\times}$
by Corollary~\ref{cor:algunit}, we require a representation of $\overline{\Lambda}^{\times}$
that will allow us to efficiently determine the group-theoretic properties required for applications, 
such as the computation of the abelianisation $\overline{\Lambda}^{\times \mathrm{ab}}$.
Specifically, representations which are useful for actual computations include
realisations of $\overline{\Lambda}^{\times}$ as a subgroup of either 
$\GL_{n}(\F_{q})$ for some finite field $\F_{q}$,
or of symmetric groups $\mathfrak{S}_{d}$, both of course with reasonable sizes for $n$ and 
$q$, or $d$, respectively (see \cite{MR2129747}).

Let $\Aut_{\Z}(\overline{\Lambda})$ denote the group of $\Z$-linear automorphisms of the
(additive) $\Z$-module $\overline{\Lambda}$. 
Consider the injective homomorphism
\[ 
\overline{\Lambda}^{\times} \longrightarrow \Aut_{\Z}(\overline{\Lambda}), 
\quad x \longmapsto (y \mapsto xy).
\]
Since we assume that we can compute generators of $\overline{\Lambda}^{\times}$, 
the problem is reduced to finding a  suitable representation for $\Aut_{\Z}(\overline{\Lambda})$. 
We now describe two ways in which this can be achieved in practice.
\begin{enumerate}
\item
If $\overline{\Lambda}$ is a free $\Z/n\Z$-algebra of rank $k$ for some $n,k \in \Z_{>0}$,
then after choosing a $\Z/n\Z$-basis we have an isomorphism 
$\Aut_{\Z}(\overline{\Lambda}) \cong \GL_{k}(\Z/n\Z)$.
If we further assume that $n$ is equal to a prime $p$, then this in turn yields a realisation
of $\overline{\Lambda}^{\times}$ as a subgroup of $\GL_{k}(\F_{p})$.
\item
Generators of $\Aut_{\Z}(\overline{\Lambda})$ 
can be computed using results of Shoda~\cite{MR1512510}.
These can be used to determine a faithful permutation representation (see~\cite{MR1962794}) 
$\Aut_{\Z}(\overline{\Lambda}) \longrightarrow \mathfrak{S}_{d}$ for some $d \in \Z_{>0}$, 
which in turn yields a representation of $\overline{\Lambda}^{\times}$ as a permutation group.
\end{enumerate}
In general, the method of (i) allows computations to be performed more efficiently than the
method of (ii), but of course the method of (i) is not always applicable. 
We also note that if $\overline{\Lambda}$ can be written as a direct product of rings
$R_{1} \times \cdots \times R_{s}$, then
\[
\overline{\Lambda}^{\times} = (R_{1} \times \cdots \times R_{s})^{\times} = R_{1}^{\times} \times \cdots \times R_{s}^{\times}
\]
and so either of the above methods can be applied to each component $R_{i}$ individually,
which is more efficient than applying either method to $\overline{\Lambda}$ directly.

\section{Freeness testing using algebra decompositions}\label{sec:hybrid-method}

\subsection{Setup}
Let $K$ be a number field with ring of integers $\mathcal{O}=\mathcal{O}_{K}$, 
and let $A$ be a finite-dimensional semisimple $K$-algebra.
Let $\Lambda$ be an $\mathcal{O}$-order in $A$ and let $\mathcal{M}$ be a maximal
$\mathcal{O}$-order in $A$ containing $\Lambda$.
Recall from Definition~\ref{def:gen-Swan} that for $\beta \in \mathcal{M} \cap A^{\times}$, the test lattice $X_{\beta}=X_{\beta, \Lambda, \mathcal{M}}$ is defined to be
$\mathcal{M} \beta \cap \Lambda$.

Algorithms~\ref{alg:SFC-naive}, \ref{alg:SFC-fail} and \ref{subsec:rel-to-gen-Swan}
for determining whether $\Lambda$ has SFC each rely on the ability to check whether  
a given (stably free) $\Lambda$-lattice of the form $X_{\beta}$ is in fact free over $\Lambda$.
As explained in the further details of Step (11) of Algorithm~\ref{alg:SFC-naive},
in principle, this can be performed using \cite[Algorithm 3.1]{BJ11}. 
However, this is only practical when $A$ is of `moderate' dimension.
In this section, we shall present a modified freeness test 
under assumptions of the following paragraph.

Henceforth assume that $A$ admits a decomposition $A = A_{1} \oplus A_{2}$ of $K$-algebras
such that $A_{1}$ satisfies the Eichler condition.
In applications, we can always assume that $A_2 \ne 0$ is a direct sum of
totally definite quaternion algebras. 
If $A_{1}=0$ then we need to use a slight modification of \cite[Algorithm 3.1]{BJ11}
as discussed in the further details of Step (11) of Algorithm~\ref{alg:SFC-naive}.

Let $e_{1},e_{2} \in A$ be the central idempotents such that $A_{i} = e_{i}A$ for $i=1,2$. 
Let $C$ denote the center of $A$ and let $\mathcal{O}_{C}$ denote the (unique) maximal $\mathcal{O}$-order in $C$.
Let $\mathfrak{f}$ be any full two-sided ideal of $\mathcal{M}$ that is contained in 
$\Lambda$ and let $\mathfrak{g} = \mathfrak{f} \cap \mathcal{O}_{C}$.
Then the central idempotents $e_{1}$ and $e_{2}$ induce decompositions
\[
\mathcal{M} = \mathcal{M}_{1} \oplus \mathcal{M}_{2}, \quad
\mathfrak{f} = \mathfrak{f}_{1} \oplus \mathfrak{f}_{2}, \quad
C = C_{1} \oplus C_{2}, \quad
\mathcal{O}_{C} = \mathcal{O}_{C_{1}} \oplus \mathcal{O}_{C_{2}}, \quad
\mathfrak{g} = \mathfrak{g}_{1} \oplus \mathfrak{g}_{2}.
\]
For $i = 1,2$ let $T_{i}$ denote the image of 
$\mathcal{M}_{i}^{\times} \longrightarrow \left(\mathcal{M}_{} / \mathfrak{f}_{i} \right)^{\times}$ 
and note that $T := T_{1} \times T_{2}$ is the image of 
$\mathcal{M}^{\times} \longrightarrow (\mathcal{M}/\mathfrak{f})^{\times}$.

\subsection{Criteria for certain ideals to be free}
We shall consider left ideals $X$ of $\Lambda$ with the following two properties:
$X + \mathfrak{f} = \Lambda$ and there exists $\beta \in \mathcal{M} \cap A^{\times}$ such that $\mathcal{M}X = \mathcal{M}\beta$. 
Note that by Proposition~\ref{prop:Xbeta}, the test lattices
$X_{\beta}$ considered in Algorithms~\ref{alg:SFC-naive}, \ref{alg:SFC-fail} and \ref{subsec:rel-to-gen-Swan} always satisfy both of these properties.

\begin{lemma}\label{lem:hybrid}
Let $X$ be a left ideal of $\Lambda$. 
Suppose that $X + \mathfrak{f} = \Lambda$ and that there exists $\beta \in \mathcal{M} \cap A^{\times}$ such that $\mathcal{M}X = \mathcal{M}\beta$.
Write $\beta=(\beta_{1},\beta_{2}) \in \mathcal{M}_{1} \oplus \mathcal{M}_{2}$.
Then the following assertions are equivalent:
\begin{enumerate}
\item 
$X$ is free over $\Lambda$.
\item
There exists $\overline{\alpha} \in (\Lambda/\mathfrak{f})^{\times}$ 
such that $\overline{\beta} \overline{\alpha} \in T$.
\item
There exist $\overline{\alpha}_{2} \in \overline{\beta}{}^{-1}_{2}T_{2}$ and 
$\overline{\alpha}_{1} \in (\mathcal{M}_{1}/\mathfrak{f}_{1})^{\times}$ such that
\[
(\overline{\alpha}_{1}, \overline{\alpha}_{2}) \in 
(\Lambda/\mathfrak{f})^{\times} \quad \text{and} \quad
\overline{\beta}_{1} \overline{\alpha}_{1} \in T_{1}.
\] 
\end{enumerate}
\end{lemma}

\begin{proof}
The equivalence of (i) and (ii) follows from Corollary \ref{cor:central-imsp}. 
Assertion (iii) is just an immediate reformulation of (ii).
\end{proof}

\cite[Algorithm 3.1]{BJ11} essentially uses the equivalence of conditions (i) and (ii) in 
Lemma~\ref{lem:hybrid} to reduce the freeness test to an enumeration over the set $T$ of cardinality $\lvert T \rvert = \lvert T_1 \rvert \cdot \lvert T_2 \rvert$.
In the following, we show how condition~(iii) can be used to replace this by an enumeration over $T_{2}$ and computations related to the reduced norm. 
(See also Remark~\ref{rmk:compare-freeness-methods}.)

By \cite[Lemma 4.4]{MR4493243},  
there exists a surjective group homomorphism $\overline{\nr}$
that fits into the commutative diagram
\begin{equation}\label{eq:barnr-comm-diag}
\xymatrix@1@!0@=48pt { 
\mathcal{M}_{1}^{\times} \ar[rr]^{\nr}  \ar[d] & & \mathcal{O}_{C_{1}}^{\times} \ar[d] \\
(\mathcal{M}_{1}/\mathfrak{f}_{1})^{\times} \ar[rr]^{\overline{\nr}}  & &
(\mathcal{O}_{C_{1}}/\mathfrak{g}_{1})^{\times},
} 
\end{equation}
where the vertical maps are induced by the canonical projections.
For $i=1,2$, let 
$\pi_{i} : \Lambda \longrightarrow \mathcal{M}_{i}$ denote the canonical projection
and set $\pi_{i}(\Lambda):=\Lambda_{i}$.
Hence we can and do also consider $\pi_{i}$ as a map
$\pi_{i} \colon : \Lambda \longrightarrow \Lambda_{i}$, depending on context.

Each $\pi_{i}$ induces a group homomorphism
\begin{equation}\label{eq:def-overlinepii}
\overline{\pi}_{i} \colon (\Lambda/\mathfrak{f})^{\times} 
\longrightarrow (\Lambda_{i}/\mathfrak{f}_{i})^{\times},  
\end{equation}
which is surjective 
by~\cite[Chapter III, (2.9) Corollary]{bass}.
Let $\Sigma := \ker(\overline{\pi}_{2})$ and let
\begin{equation}\label{eq:def-S}
S := \bnr(\overline{\pi}_{1}(\Sigma)) = \left\{ \bnr(\overline{\gamma}_{1}) \mid \exists 
\overline{\gamma}_{2} \in (\Lambda_{2}/\mathfrak{f}_{2})^{\times} : 
(\overline{\gamma}_{1}, \overline{\gamma}_{2}) \in \Sigma \right\}.  
\end{equation}
Then $S$ is a subgroup of $\left( \calO_{C_1} / \mathfrak{g}_{1} \right)^{\times}$
by definition and
\begin{align*}
S 
&= 
\left\{ \bnr(\overline{\gamma}_{1}) \mid \overline{\gamma}_{1} 
\in \left( \mathcal{M}_{1} / \frf_1 \right)^{\times} \text{ and }
(\overline{\gamma}_{1}, \overline{1}) \in  \left(\Lambda/\frf\right)^{\times} \right\}, \\
&= \left\{ \bnr(\overline{\gamma}_{1}) \mid \overline{\gamma}_{1} \in 
\left( \mathcal{M}_{1} / \frf_{1} \right)^{\times} \text{ and }
(\gamma_{1}, 1) \in  \Lambda \right\} \nonumber
\end{align*}
where the first equality follows from \eqref{eq:def-S} and the second follows from the fact that 
$(\Lambda/\mathfrak{f})^{\times} = (\mathcal{M}/\mathfrak{f})^{\times} \cap (\Lambda/\mathfrak{f})$
(see \cite[Chapter I, Theorem 7.9(iii)]{MR1183469}).
Let $\sigma$ denote the following composition of canonical maps
\[
\calO_{C_1}^{\times} \longrightarrow \left( \calO_{C_1} / \frg_1 \right)^{\times} \longrightarrow \left( \calO_{C_1} / \frg_1 \right)^{\times}\!/ S.
\]

\begin{lemma}\label{lem:fixeda2}
Let $X$ be a left ideal of $\Lambda$. 
Suppose that $X + \mathfrak{f} = \Lambda$ and that there exists $\beta \in \mathcal{M} \cap A^{\times}$ such that $\mathcal{M}X = \mathcal{M}\beta$.
Then $\overline{\beta} \in \left( \calM/\frf \right)^{\times}$.
Write $\beta=(\beta_{1},\beta_{2}) \in \mathcal{M}_{1} \oplus \mathcal{M}_{2}$.
Let $\overline{\alpha}_{2} \in \overline{\beta}_{2}^{-1} T_{2}$ and assume that 
$\overline{\gamma}_{1} \in  \left( \mathcal{M}_{1} / \mathfrak{f}_{1} \right)^{\times}$
satisfies
$(\overline{\gamma}_{1}, \overline{\alpha}_{2}) \in \left( \Lambda / \mathfrak{f} \right)^{\times}$.
Then the following assertions are equivalent:
\begin{enumerate}
\item
There exists $\overline{\alpha}_{1} \in \left( \mathcal{M}_{1} / \mathfrak{f}_{1} \right)^{\times}$
such that 
$(\overline{\alpha}_{1}, \overline{\alpha}_{2}) \in (\Lambda/\mathfrak{f})^{\times} 
\text{ and } \overline{\beta}_{1} \overline{\alpha}_{1} \in T_{1}$.
\item
$\bnr (\overline{\beta}_{1} \overline{\gamma}_{1}) \bmod{S} \in \im(\sigma )$.
\end{enumerate}
Moreover, if these equivalent conditions hold, then $X$ is free over $\Lambda$.
\end{lemma}

\begin{proof}
We have $\mathcal{M} 
= \mathcal{M}(X+\mathfrak{f})
= \mathcal{M}X + \mathfrak{f}
= \mathcal{M}\beta + \mathfrak{f}$,
and so $\overline{\beta} \in (\mathcal{M}/\mathfrak{f})^{\times}$. 
  
Assume that~(i) holds. 
Then by definition of $T_{1}$ there exists $u_1 \in \mathcal{M}_{1}^{\times}$
such that
$\bnr(\overline{\beta}_{1}\overline{\alpha}_{1}) 
= \bnr(\overline{u}_{1}) = \overline{\epsilon}_{1}$ in 
$\left( \calO_{C_1} / \mathfrak{g}_{1} \right)^{\times}$, where $\epsilon_{1} := \nr({u}_{1})$.
Hence
$\bnr( (\overline{\beta}_{1} \overline{\gamma}_{1})(\overline{\gamma}_{1}^{-1} \overline{\alpha}_{1})) = \overline{\epsilon}_{1}$ 
and it remains to show that $\bnr( \overline{\gamma}_{1}^{-1} \overline{\alpha}_{1} ) \in S$.
This in turn is a consequence of
\[
( \overline{\gamma}_{1}^{-1} \overline{\alpha}_{1}, \overline{1}) 
= ( \overline{\gamma}_{1}^{-1} \overline{\alpha}_{1}, \overline{\alpha}_{2}^{-1} \overline{\alpha}_{2})
= ( \overline{\gamma}_{1}, \overline{\alpha}_{2})^{-1} 
(\overline{\alpha}_{1}, \overline{\alpha}_{2})  \in \left( \Lambda / \frf \right)^{\times}.
\]
Now assume that (ii) holds.
Then there exists a unit $\epsilon_1 \in \calO_{C_1}^{\times}$ and
an element $(\overline{\delta}_{1}, \overline{1}) \in \left(\Lambda / \frf \right)^{\times}$
with $\overline{\delta}_1 \in \left(\calM_1/\frf_1\right)^\times$ such that
$\bnr(\overline{\beta}_{1} \overline{\gamma}_{1} \overline{\delta}_{1}) 
= \overline{\epsilon}_{1}$ in $\left(\calO_{C_1}/ \frg_1\right)^{\times}$.
By Eichler's theorem on units \cite[(51.34)]{curtisandreiner_vol2} there
exists $u_{1} \in \mathcal{M}_{1}^{\times}$ such that 
$\overline{\beta}_{1} \overline{\gamma}_{1} \overline{\delta}_{1} = \overline{u}_{1}$
in $\left(\mathcal{M}_1/\frf_{1}\right)^{\times}$, that is,
$\overline{\beta}_{1} \overline{\gamma}_{1} \overline{\delta}_{1} \in T_{1}$. 
Since $(\overline{\gamma}_{1} \overline{\delta}_{1}, \overline{\alpha}_{2}) 
= ( \overline{\gamma}_{1}, \overline{\alpha}_{2})( \overline{\delta}_{1} , \overline{1}) 
\in \left( \Lambda / \frf \right)^{\times}$ we may take
$\overline{\alpha}_{1} := \overline{\gamma}_{1} \overline{\delta}_{1}$. 

The final assertion follows immediately from Lemma~\ref{lem:hybrid}.
\end{proof}

\begin{prop}\label{prop:freeness-test}
Let $X$ be a left ideal of $\Lambda$. 
Suppose that $X + \mathfrak{f} = \Lambda$ and that there exists $\beta \in \mathcal{M} \cap A^{\times}$ such that $\mathcal{M}X = \mathcal{M}\beta$.
Then the following statements hold:
\begin{enumerate}
\item $X$ is free over $\Lambda$ if and only if there exist
$\overline{\alpha}_{2} \in \overline{\beta}_{2}^{-1} T_{2}$ and $\overline{\gamma}_{1} \in \left( \mathcal{M}_{1}/\mathfrak{f}_{1}\right)^{\times}$ such that 
$(\overline{\gamma}_{1}, \overline{\alpha}_{2}) \in \left( \Lambda / \mathfrak{f} \right)^{\times}$
and $\bnr(\overline{\beta}_{1} \overline{\gamma}_{1}) \bmod{S} \in \im ( \sigma )$.
\item Fix $\overline{\alpha}_{2} \in \overline{\beta}_{2}^{-1} T_{2}$.
Suppose there exists 
$\overline{\gamma}_{1} \in \left( \mathcal{M}_{1}/\mathfrak{f}_{1}\right)^{\times}$ such that 
$(\overline{\gamma}_{1}, \overline{\alpha}_{2}) \in \left( \Lambda / \mathfrak{f} \right)^{\times}$
and $\bnr(\overline{\beta}_{1} \overline{\gamma}_{1}) \bmod{S} \in \im ( \sigma )$. 
Then for every $\overline{\mu}_{1} \in \left( \mathcal{M}_{1}/\mathfrak{f}_{1}\right)^{\times}$ such that 
$(\overline{\mu}_{1}, \overline{\alpha}_{2}) \in \left( \Lambda / \mathfrak{f} \right)^{\times}$,
we have $\bnr(\overline{\beta}_{1} \overline{\mu}_{1}) \in \im ( \sigma )$. 
\end{enumerate}
\end{prop}

\begin{proof}
(i) Suppose that $X$ is free over $\Lambda$. 
Then by Lemma~\ref{lem:hybrid} there exist $\overline{\alpha}_{2} \in \overline \beta_{2}^{-1}T_{2}$ and 
$\overline{\alpha}_{1} \in (\mathcal{M}_{1}/\mathfrak{f}_{1})^{\times}$ such that
$(\overline{\alpha}_{1}, \overline{\alpha}_{2}) \in 
(\Lambda/\mathfrak{f})^{\times} \text{ and }
\overline{\beta}_{1} \overline{\alpha}_{1} \in T_{1}$.
Moreover, $\bnr(\overline{\beta}_{1} \overline{\alpha}_{1}) \bmod S \in \im ( \sigma )$ by 
Lemma~\ref{lem:fixeda2}, and so we may take $\overline{\gamma}_{1}=\overline{\alpha}_{1}$.
The converse holds by Lemmas~\ref{lem:hybrid} and \ref{lem:fixeda2}.

(ii) The elements $\overline{\beta}_{1}$ and $\overline{\gamma}_{1}$ satisfy 
Lemma~\ref{lem:fixeda2} (ii). 
Thus Lemma~\ref{lem:fixeda2}~(i) holds, and this is independent of $\overline{\gamma}_{1}$.
Thus if $\overline{\mu}_{1} \in \left( \mathcal{M}_{1}/\mathfrak{f}_{1}\right)^{\times}$ also satisfies
$(\overline{\mu}_{1}, \overline{\alpha}_{2}) \in \left( \Lambda / \mathfrak{f}\right)^{\times}$,
then Lemma~\ref{lem:fixeda2}~(ii) applied to 
$\overline{\mu}_{1}$ implies that
$\bnr(\overline{\beta}_{1} \overline{\mu}_{1}) \bmod S \in \im ( \sigma )$.
\end{proof}

\begin{lemma}\label{lem:completion}
Let $\alpha_{2} \in \mathcal{M}_{2}$ with
$\overline{\alpha}_{2} \in \left(\mathcal{M}_{2}/\mathfrak{f}\right)^{\times}$.
The following assertions are equivalent:
\begin{enumerate}
\item
There exists an element
$\overline{\gamma}_{1} \in (\calM_1/\frf_1)^\times$ such that 
$(\overline{\gamma}_{1}, \overline{\alpha}_{2}) \in \left(\Lambda/\frf\right)^{\times}$.
\item
$\pi_{2}^{-1}(\{\alpha_{2} \}) \neq \emptyset$, 
where $\pi_{2} \colon \Lambda \longrightarrow \mathcal{M}_{2}$ is the canonical projection.
\end{enumerate}
\end{lemma}

\begin{proof}
It is clear that (i) implies (ii). Suppose that (ii) holds.
Then there exists $\alpha_{1} \in \mathcal{M}_{1}$
such that $(\alpha_{1},\alpha_{2}) \in \Lambda$. 
Thus $\alpha_{2} \in \Lambda_{2}$.
By assumption $\overline{\alpha}_{2} \in (\mathcal{M}_{2}/\mathfrak{f}_{2})^{\times}$
and hence $\overline{\alpha}_{2} \in \left(\Lambda_{2}/\mathfrak{f}_{2}\right)^{\times}$
since $(\Lambda_{2}/\frf_{2})^{\times} = (\mathcal{M}_{2}/\frf_{2})^{\times} \cap (\Lambda_{2}/\frf_{2})$ (see \cite[Chapter I, Theorem 7.9(iii)]{MR1183469}).
Since the canonical map $\overline{\pi}_{2} \colon (\Lambda/\mathfrak{f})^{\times} 
\longrightarrow (\Lambda_{2}/\mathfrak{f}_{2})^{\times}$ of \eqref{eq:def-overlinepii}
is surjective, this completes the proof.
\end{proof}

\begin{algorithm}\label{alg:hybrid}
Suppose that $A_{1}$ satisfies the Eichler condition and that every simple component of $A_{2}$ is a totally definite quaternion algebra.
Let $X$ be a left ideal of $\Lambda$. 
Suppose that $X + \mathfrak{f} = \Lambda$ and that there exists $\beta \in \mathcal{M} \cap A^{\times}$ such that $\mathcal{M}X = \mathcal{M}\beta$.
Write $\beta=(\beta_{1},\beta_{2}) \in \mathcal{M}_{1} \oplus \mathcal{M}_{2}$.
If $X$ is free over $\Lambda$ then the following algorithm returns  \ensuremath{\mathsf{true}};
otherwise it returns \ensuremath{\mathsf{false}}.
\renewcommand{\labelenumi}{(\arabic{enumi})}
\begin{enumerate}
\item Compute $T_2$.
\item Compute $S$. 
\item Compute $\im(\sigma)$.
\item
For each $\overline{\alpha}_{2} \in \overline{\beta}_{2}^{-1} T_{2}$ do the following:
\begin{enumerate}
\item
Compute $\overline{\gamma}_{1} \in \left( \calM_1/\frf_1\right)^{\times}$ such that 
$(\overline{\gamma}_{1}, \overline{\alpha}_{2}) \in \left( \Lambda / \frf \right)^{\times}$ or, 
if no such $\overline{\gamma}_{1}$ exists,
continue with the next $\overline{\alpha}_{2}$.
\item
If $\bnr(\overline{\beta}_{1} \overline{\gamma}_{1}) \bmod S \in \im ( \sigma )$, return \ensuremath{\mathsf{true}}.
\end{enumerate}
\item
Return \ensuremath{\mathsf{false}}.
\end{enumerate}
\end{algorithm}

\begin{proof}[Proof of correctness]
The correctness of the output follows from Proposition~\ref{prop:freeness-test}.
\end{proof}

\begin{proof}[Further details on each step]
Step (1): This can be performed as described in \cite[\S 7]{BJ11}. 

Step (2): Although the sets $\Sigma$ and $S$ are finite and can be found by enumerating $(\Lambda/\mathfrak{f})^{\times}$ and $(\mathcal{M}_{1}/\mathfrak{f}_{1})^{\times}$, respectively, this is impractical. We present a practical algorithm for computing $\Sigma$ and $S$ in \S \ref{subsec:compsigma}.

Step (3): This is straightforward given the output of step (2) and standard algorithms
from computational algebraic number theory (see \cite{MR1728313}).

Step 4(a): For a given $\overline{\alpha}_{2} \in (\mathcal{M}_2/\mathfrak{f}_2)^{\times}$
we can determine whether there exists 
$\overline{\gamma}_{1} \in (\mathcal{M}_{1}/\mathfrak{f}_{1})^{\times}$ such 
that  $(\overline{\gamma}_{1},\overline{\alpha}_{2}) \in (\Lambda/\mathfrak{f})^{\times}$ 
(and find such an element if it does exist) as follows.
The canonical projection $\pi_{2} \colon \Lambda \longrightarrow \mathcal{M}_2$ is a ring homomorphism and hence a homomorphism of the underlying finitely generated abelian groups.
Thus we can check condition~(ii) of Lemma~\ref{lem:completion} and find 
$\delta_{1} \in \mathcal{M}_{1}$ with $(\delta_{1}, \alpha_{2}) \in \Lambda$ 
if it exists (see \cite[\S 4.1]{MR1728313}).
Note that $(\overline{\delta}_{1}, \overline{\alpha}_{2}) \in \Lambda/\mathfrak{f}$
maps to $\overline{\alpha}_{2} \in (\Lambda_{2}/\mathfrak{f}_{2})^{\times}$
under the canonical map
$\tilde{\pi}_{2} \colon \Lambda/\mathfrak{f} \to \Lambda_{2}/\mathfrak{f}_{2}$, 
but this preimage of $\overline{\alpha}_{2}$ is not necessarily a unit.
Since $\tilde{\pi}_{2}$ is a homomorphism of finite abelian groups, we can determine random elements of $\ker(\tilde{\pi}_{2})$ and thus of
$(\overline{\delta}_{1}, \overline{\alpha}_{2}) + \ker(\tilde{\pi}_{2}) 
= \tilde{\pi}_{2}^{-1}(\{\overline{\alpha}_{2}\})$ until we have found an invertible preimage.

Step 4(b): This can be performed using the commutative diagram \eqref{eq:barnr-comm-diag}
and standard algorithms for finitely generated abelian groups, as presented in \cite[\S 4.1]{MR1728313}, for example.

Step (5) is trivial.
\end{proof}

\begin{remark}\label{rmk:compare-freeness-methods}
Compared to \cite[Algorithm 3.1]{BJ11} where the size 
of the enumeration set is essentially the cardinality of the image of
the canonical map $\mathcal{M}^{\times} \longrightarrow \left( \calM / \frf \right)^{\times}$,
that is, $\lvert T_1 \rvert \cdot \lvert T_2 \rvert$,  the enumeration set in Algorithm~\ref{alg:hybrid}
is of size $\lvert T_2 \rvert$, and hence only depends on the component $A_2$.
Note that $T_{2}$ can be computed as explained in \cite[\S 4.5]{BJ11}
(see also \cite[Remark 7.5]{Kirschmer2010}), and this is much faster 
than computing the image of $\mathcal{M} \longrightarrow (\mathcal{M}/\mathfrak{f})^{\times}$
for $\mathcal{M}$ a maximal $\mathcal{O}$-order in an arbitrary finite-dimensional semisimple
$K$-algebra. If $|T_{1}|$ is `small', then it may be that \cite[Algorithm 3.1]{BJ11} is faster 
than Algorithm~\ref{alg:hybrid} thanks to the optimisation described in 
\cite[\S 7]{BJ11} (also see \cite[\S 9]{MR4136552}).
\end{remark}

\subsection{Computation of $S$}\label{subsec:compsigma}
Recall the definition of $S$ given in \eqref{eq:def-S}.
Since we require only generators of $S$, it will be enough to consider the computation of (generators of) $\Sigma = \ker(\overline{\pi}_{2})$.
Let $\mathfrak{h}_{1} = \Lambda \cap \mathcal{M}_{1}$ 
and consider
the two-sided ideal $\tilde{\mathfrak{f}} := \mathfrak{h}_1 + \mathfrak{f} = \mathfrak{h}_1 \oplus \mathfrak{f}_2$.
Let
\[
\overline{\pi} \colon ( \Lambda / \frf )^{\times} \longrightarrow  
( \Lambda / \tilde\frf )^{\times}
\]
denote the canonical map, which is surjective by \cite[Chapter III, (2.9) Corollary]{bass}.

\begin{lemma}\label{lem:sigmachar}
We have $\Sigma = \ker(\overline{\pi})$.
\end{lemma}

\begin{proof}
Since the composition 
$\Lambda \stackrel{\pi_{2}}{\longrightarrow} \Lambda_{2} \longrightarrow \Lambda_{2}/\mathfrak{f}_{2}$
is surjective with kernel $\tilde{\mathfrak{f}}$, we obtain an isomorphism
$\overline{\tau} \colon ( \Lambda/\tilde{\mathfrak{f}} )^{\times} \longrightarrow 
(\Lambda_2/\mathfrak{f}_2)^{\times}$ making the diagram
\[ 
\begin{tikzcd}
& ( \Lambda/\tilde{\mathfrak{f}})^{\times} \arrow{d}{\overline{\tau}}  \\
( \Lambda/\mathfrak{f} )^{\times} \arrow{ru}{\overline{\pi}} \arrow{r}{\overline{\pi}_{2}} &
( \Lambda_{2}/\mathfrak{f}_{2} )^{\times}
\end{tikzcd}
\] 
commute. The claim now follows from $\Sigma = \ker(\overline{\pi}_{2})$.
\end{proof}

\begin{lemma}\label{lem:sigmaprimary}
There exist coprime primary ideals $\{\mathfrak{q}_1,\dotsc,\mathfrak{q}_{r + s}\}$ and $\{\tilde{\mathfrak{q}}_1,\dotsc,\tilde{\mathfrak{q}}_r\}$ of $\Lambda \cap C$ such that the following assertions hold:
\begin{enumerate}
\item
We have $\mathfrak{q}_i \subseteq \tilde{\mathfrak{q}}_i$ for $1 \leq i \leq r$.
\item
We have
\[ 
\Lambda/\mathfrak{f} \cong \prod_{i=1}^{r + s} \Lambda/(\mathfrak q_i \Lambda + \mathfrak{f}) \quad \text{and} \quad \Lambda/\tilde{\mathfrak{f}} \cong \prod_{i=1}^r \Lambda/(\tilde{\mathfrak{q}}_i \Lambda + \tilde{\mathfrak{f}}).
\] 
\item
We have
\[ 
\Sigma \cong \prod_{i=1}^r \ker(\tau_i) \times \prod_{i=1}^s (\Lambda/(\mathfrak{q}_{r + i}\Lambda + \mathfrak{f}))^{\times},
\] 
where 
$\tau_{i} \colon (\Lambda/(\mathfrak{q}_{i}\Lambda + \mathfrak{f}))^{\times} \longrightarrow (\Lambda/(\tilde{\mathfrak{q}}_{i}\Lambda + \tilde{\mathfrak{f}}))^{\times}$
are the canonical surjections.
\end{enumerate}
\end{lemma}

\begin{proof}
Consider $\mathfrak{g} = \mathfrak{f} \cap C$ and $\tilde{\mathfrak{g}} = \mathfrak{\tilde{f}} \cap C$, which are full $(\Lambda \cap C)$-ideals with $\mathfrak{g} \subseteq \tilde{\mathfrak{g}}$.
Let $\mathfrak{p}_1,\dotsc,\mathfrak{p}_r$ and $\mathfrak{p}_1,\dotsc, \mathfrak{p}_{r + s}$ be the maximal ideals of $\Lambda \cap C$ containing $\tilde{\mathfrak{g}}$ and $\mathfrak{g}$, respectively.
Let now $\mathfrak{g} = \prod_{i=1}^{r + s} \mathfrak{q_i}$
be a primary decomposition, where $\mathfrak{q}_{i}$ is $\mathfrak{p}_{i}$-primary for 
$1 \leq i \leq r + s$.
Since the $\mathfrak{p}_{i}$ are maximal, it is clear
from \cite[Proposition (4.2)]{MR242802}
that $\tilde\frq_i := \mathfrak{q}_i + \mathfrak{\tilde{g}}$ is also $\mathfrak{p}_{i}$-primary
for $1 \leq i \leq r$.
Thus
\[
  \tilde{\mathfrak{g}} = \mathfrak{g} + \tilde{\mathfrak{g}} =
  \left( \prod_{i=1}^{r + s} \mathfrak{q}_i \right) + \tilde{\mathfrak{g}} =
   \prod_{i=1}^{r + s} \left(\mathfrak{q}_i  + \tilde{\mathfrak{g}}\right) =
 \prod_{i=1}^r \tilde{\mathfrak{q}}_i
\]
is a primary decomposition of $\tilde{\mathfrak{g}}$.
This immediately shows~(i).
For (ii), note that~\cite[Lemma~3.5]{bley-boltje} implies that
\[
  \mathfrak{f} = \bigcap_{i=1}^{r + s} (\mathfrak{q}_i \Lambda + \mathfrak f) \quad
  \text{ and }\quad\tilde{\mathfrak{f}} = \bigcap_{i=1}^{r} 
  (\tilde{\mathfrak{q}}_i \Lambda +\tilde{\mathfrak f}),
\]
and
thus the claim follows from the Chinese remainder theorem.
Finally, (iii) follows from (ii) together with Lemma~\ref{lem:sigmachar}.
\end{proof}

Since we know how to compute generators for each of the groups
$\left( \Lambda / (\frq_{r + i}\Lambda + \frf) \right)^{\times}$ for $i=1, \ldots, s$ by the results of
\S \ref{subsec:primary-dec},
it remains to consider the computation of $\ker(\tau_i)$ for $i = 1, \ldots,r$.
Each element in $\ker(\tau_i)$ is the image in
$(\Lambda/(\mathfrak{q}_{i}\Lambda + \mathfrak{f}))^{\times}$ of an element
of the form $1+\lambda$ with $\lambda \in  (\tilde\frq_i\Lambda + \tilde\frf)$. 
Note, however, that not every element in $1+(\tilde\frq_i\Lambda + \tilde\frf)$
is actually a unit mod $(\frq_i\Lambda + \frf)$.

\begin{lemma}\label{lem:singlesigma}
Let $\mathfrak{p}$ be a maximal ideal of $\Lambda \cap C$, 
let $\mathfrak{q}$ and $\tilde{\mathfrak{q}}$ be two $\mathfrak{p}$-primary ideals of 
$\Lambda \cap C$ with ${\mathfrak{q}} \subseteq \tilde{\mathfrak{q}}$, and let
\[ 
\tau \colon (\Lambda/(\mathfrak{q}\Lambda + \mathfrak{f}))^\times \longrightarrow (\Lambda/(\tilde{\mathfrak{q}}\Lambda + \tilde{\mathfrak{f}}))^{\times} 
\]
be the canonical projection.
Furthermore set
\[ 
\tilde\frQ := \tilde\frq\Lambda + \tilde\frf, \quad \frQ_1 := \tilde\frq\Lambda + \frf, \quad \frQ := \frq\Lambda + \frf.
\] 
Then the following assertions hold:
\begin{enumerate}
\item For every $\alpha \in \Lambda$ we have
\[
\alpha \text{ is invertible mod } \frQ \iff \alpha \text{ is invertible mod } \frQ_1.
\]
\item
Identifying $(1+\tilde\frQ)$ and $(1+\frQ_1)$ with their images in 
$\Lambda / \frQ$ we have
$\ker(\tau) = (1+\tilde\frQ) \cap \left( \Lambda / \frQ \right)^\times$
and the sequence
\[
0 
\longrightarrow \frac{1+\frQ_1}{1+\frQ} 
\longrightarrow \ker(\tau) 
\longrightarrow \frac{(1+\tilde\frQ) \cap \left( \Lambda / \frQ \right)^\times}{1+\frQ_1} 
\longrightarrow 0
\]
is exact.
\end{enumerate}
\end{lemma}

\begin{proof}
(i) 
Since both $\tilde\frq$ and $\frq$ are $\frp$-primary,
there exists $m>0$ such that $\tilde\frq^m \subseteq \frq$.
Hence $\frQ_1^m \subseteq \frQ$ and every element of $\frQ_1/\frQ$ is nilpotent in $\Lambda/\frQ$. 
Suppose $\alpha$ is invertible mod $\frQ_1$. 
Then there exist $\beta \in \Lambda$ and $\xi \in \frQ_1$
such that $\alpha \beta = 1-\xi$. 
For $x \in \Lambda$, let $\overline{x}$ denote its image in  $\Lambda/\frQ$.
Then we have
\[
\overline{\alpha} \overline{\beta} \cdot \sum_{n=0}^{\infty} \overline{\xi}^{n} 
= (\overline{1} - \overline{\xi}) \cdot \sum_{n=0}^{\infty} \overline{\xi}^{n} = \overline{1}.
\]
The converse is clear. (ii) now follows easily from (i).
\end{proof}

\begin{remark}\label{rem:compsigma}
In conjunction with Lemma~\ref{lem:sigmaprimary}, 
Lemma~\ref{lem:singlesigma} yields the following method
for computing generators of $\Sigma$ (and hence of $S$).
Primary decompositions in $\Lambda \cap C$ can be determined using
Corollary~\ref{cor:algprimdec} and thus a decomposition
\[
\Sigma \cong \prod_{i=1}^r \ker(\tau_i) \times \prod_{i=1}^{s} 
(\Lambda/(\mathfrak{q}_{r + i}\Lambda + \mathfrak{f}))^{\times},
\] 
with suitable primary ideals $\mathfrak{q}_{i}, \tilde{\mathfrak q}_i$ as in Lemma~\ref{lem:sigmaprimary}~(iii) can be obtained.
For the groups $(\Lambda/(\mathfrak{q}_{r + i}\Lambda + \mathfrak{f}))^\times$, generators can be obtained 
by Corollary~\ref{cor:algunit}
and thus it suffices to determine each $\ker(\tau_i)$.
For $1 \leq i \leq r$ we set $\tau := \tau_i$ and fix the notation as in Lemma~\ref{lem:singlesigma}.
Since 
\begin{equation}\label{eq:ker-tau-ses}
1 
\longrightarrow \frac{1+\frQ_1}{1+\frQ} 
\longrightarrow \ker(\tau) 
\longrightarrow \frac{(1+\tilde\frQ) \cap \left( \Lambda / \frQ \right)^\times}{1+\frQ_1}
\longrightarrow 1  
\end{equation}
is exact, it suffices to find generators for the first and third group.
Generators for $ \frac{1+\frQ_1}{1+\frQ}$ can be computed as in \cite[\S 3.7]{bley-boltje}.
For the computation of generators of the nontrivial group on the right
hand side of \eqref{eq:ker-tau-ses}, the elements can be listed as follows:
determine a set of representatives of $\tilde\frQ / \frQ_1$
and check for each such representative $\alpha$ whether $1+\alpha$ 
is a unit mod $\frQ$.
\end{remark}

\subsection{A probabilistic algorithm}

In the previous section, the computation of $\Sigma$ respectively $S$ was reduced to the problem of enumerating all elements of quotients of the form $\tilde\frQ/\frQ_1$ for two-sided $\Lambda$-ideals $\tilde\frQ$ and $\frQ_1$.
Since the order of $\tilde\frQ/\frQ_1$ can be quite large, we present a probabilistic approach, which avoids enumeration.
We assume that $S' \subseteq S$ is a subgroup and consider the following composition, which we denote by $\sigma'$:
  \[ \calO_{C_1}^\times \longrightarrow \left( \calO_{C_1} / \frg_1 \right)^\times \longrightarrow
    {\left( \calO_{C_1} / \frg_1 \right)^\times} / S'.
\]
Replacing $S$ by $S'$ in Lemma~\ref{lem:fixeda2}, we obtain the following weaker statement.

\begin{lemma}\label{lem:fixeda2-sprime}
Let $X$ be a left ideal of $\Lambda$. 
Suppose that $X + \mathfrak{f} = \Lambda$ and that there exists $\beta \in \mathcal{M} \cap A^{\times}$ such that $\mathcal{M}X = \mathcal{M}\beta$.
Write $\beta=(\beta_{1},\beta_{2}) \in \mathcal{M}_{1} \oplus \mathcal{M}_{2}$.
Let $\overline{\alpha}_{2} \in \overline{\beta}_{2}^{-1} T_{2}$ and assume that 
$\overline{\gamma}_{1} \in  \left( \mathcal{M}_{1} / \mathfrak{f}_{1} \right)^{\times}$
satisfies
$(\overline{\gamma}_{1}, \overline{\alpha}_{2}) \in \left( \Lambda / \mathfrak{f} \right)^{\times}$.
If $\bnr (\overline{\beta}_{1} \overline{\gamma}_{1}) \bmod S' \in \im(\sigma')$, then $X$ is free over $\Lambda$.
\end{lemma}

This yields the following probabilistic variant of Algorithm~\ref{alg:hybrid}.

\begin{algorithm}
Suppose that $A_{1}$ satisfies the Eichler condition and that every simple component of $A_{2}$ is a totally definite quaternion algebra.
Let $X$ be a left ideal of $\Lambda$. 
Suppose that $X + \mathfrak{f} = \Lambda$ and that there exists $\beta \in \mathcal{M} \cap A^{\times}$ such that $\mathcal{M}X = \mathcal{M}\beta$.
If the following algorithm returns \ensuremath{\mathsf{true}}, then $X$ is free over $\Lambda$;
if it returns \ensuremath{\mathsf{inconclusive}}, then nothing can be said about whether
$X$ is free over $\Lambda$.
\renewcommand{\labelenumi}{(\arabic{enumi})}
\begin{enumerate}
\item Compute $T_2$.
\item Compute $S'$. 
\item Compute $\im(\sigma')$.
\item
For each $\overline{\alpha}_{2} \in \overline{\beta}_{2}^{-1} T_{2}$ do the following:
\begin{enumerate}
\item
Compute $\overline{\gamma}_{1} \in \left( \calM_1/\frf_1\right)^{\times}$ such that 
$(\overline{\gamma}_{1}, \overline{\alpha}_{2}) \in \left( \Lambda / \frf \right)^{\times}$ or, 
if no such $\overline{\gamma}_{1}$ exists,
continue with the next $\overline{\alpha}_{2}$.
\item
If $\bnr(\overline{\beta}_{1} \overline{\gamma}_{1}) \in \im (\sigma') \bmod S'$, return \ensuremath{\mathsf{true}}.
\end{enumerate}
\item
Return \ensuremath{\mathsf{inconclusive}}.
\end{enumerate}
\end{algorithm}

\begin{proof}[Proof of correctness]
The correctness of the output follows from Lemma~\ref{lem:fixeda2-sprime}.
\end{proof}

\begin{proof}[Further details on each step]
Apart from Step (2), all steps are essentially the same as in Algorithm~\ref{alg:hybrid}.
Step (2) can be performed as follows. 
Recall that $S = \bnr(\overline{\pi}_{1}(\Sigma))$. 
Thus in order to find a subgroup $S'$ of $S$ it is sufficient to find a subgroup of
\[ 
\Sigma = 
\ker((\Lambda/\mathfrak{f})^{\times} \longrightarrow (\Lambda_{2}/\mathfrak{f}_{2})^{\times})). 
\] 
Random elements of $\Sigma$ can be found by sampling preimages of 
$\overline{1} \in \Lambda_{2}/\mathfrak{f}_{2}$ under the canonical map 
$\Lambda/\mathfrak{f} \longrightarrow \Lambda_{2}/\mathfrak{f}_{2}$
(a homomorphism of finitely generated abelian groups) and discarding the non-invertible elements.
In this way we obtain a sequence of elements
$(\overline{\alpha}_{i})_{i \in \Z_{\geq 1}}$ in $\Sigma$
and determine $S' := \langle \bnr(\overline{\pi}_{1}(\overline{\alpha}_{i})) \mid i \leq B \rangle$, where 
$B \in \Z_{\geq 1}$ is either a fixed predetermined bound or chosen dynamically until $S'$ stabilises.
\end{proof}

\section{Further results on SFC for integral group rings}\label{sec:further-SFC-for-int-group-rings}

The following result is Theorem~\ref{intro-thm:new-groups-with-SFC} from the introduction.

\begin{theorem}\label{thm:new-groups-with-SFC}
Let $G$ be one of the following finite groups:
\begin{equation}\label{eq:list-groups-have-SFC}
\tilde{T} \times C_{2}, \, 
\tilde{T} \times Q_{12}, \, 
\tilde{T} \times Q_{20}, \,
Q_{8} \rtimes \tilde{T}, \,
Q_{8} \rtimes Q_{12}, \,
\tilde{I} \times C_{2}, \, 
G_{(192, 183)}, \,
G_{(384, 580)}, \,
G_{(480, 962)}.
\end{equation}
Then $\Z[G]$ has SFC.
\end{theorem}

\begin{remark}
In order, the first six groups listed in \eqref{eq:list-groups-have-SFC} are isomorphic to 
\[
G_{(48,32)}, \, G_{(288,409)}, \, G_{(480,266)}, \, G_{(192,1022)}, \,  
G_{(96, 66)}, \, G_{(240,94)}.
\]
\end{remark}

\begin{proof}[Proof of Theorem~\ref{thm:new-groups-with-SFC}]
Further details on the computations used in this proof can be found in 
Appendix~\ref{appendix}.

The result for $\tilde{T} \times C_{2}$ is Theorem~\ref{thm:tildeT-times-C2-has-SFC}.
Note that $\tilde{T} \times Q_{12}$ and $\tilde{T} \times Q_{20}$ both have quotients
isomorphic to $\tilde{T} \times C_{2}$, so Theorem~\ref{thm:tildeT-times-C2-has-SFC} can be 
recovered from Lemma~\ref{lem:quot-group-ring} and the result for either $\tilde{T} \times Q_{12}$ or $\tilde{T} \times Q_{20}$.

Let $G$ be any group listed in \eqref{eq:list-groups-have-SFC}
and let $\mathcal{M}$ be any maximal $\Z$-order in $\Q[G]$. 
Then every binary polyhedral quotient of $G$ is isomorphic to 
$Q_{12}$, $Q_{20}$, $\tilde{T}$, $\tilde{O}$, or $\tilde{I}$,
and thus $\mathcal{M}$ has LFC by Corollary~\ref{cor:LFC-max-order-group-rings}.
Hence every direct summand of $\mathcal{M}$ has LFC by 
Corollary~\ref{cor:cancellation-direct-product}. 
Thus in all applications of Algorithm~\ref{alg:sfc-fiberproduct} below, it is not necessary to perform Step (3)
as the maximal orders in question are already known to have LFC (and thus SFC).

For $G=Q_{8} \rtimes Q_{12}$, $\tilde{T} \times Q_{12}$, $\tilde{T} \times Q_{20}$, or $\tilde{I} \times C_{2}$, 
there exists a normal subgroup $N$ of order $2$, such that $H:=G/N$
is isomorphic to $\GL_{2}(\F_{3})$, $G_{(144,127)}$, $G_{(240,108)}$, or $C_{2} \times A_{5}$, respectively.
In each case, $\Q[H]$ satisfies the Eichler condition and this implies that 
$\Q[H \times C_{2}]$ does as well.
Thus $\Z[H \times C_{2}]$ has LFC by Theorem~\ref{thm:Jacobinski-cancellation}.
Let $\Gamma = \Z[G] / \mathrm{Tr}_{N} \Z[G]$.
The orders of the locally free class groups of $\Z[H]$, $\Z[G]$, $\Z[H \times C_{2}]$
and $\Gamma$ are shown in the table below, and from this it is easily verified that equality \eqref{eq:quotients-class-groups} holds in each case.
Therefore Proposition~\ref{prop:ZG-SFC-iff-Gamma-SFC}
shows that $\Z[G]$ has SFC if only if $\Gamma$ has SFC.

\begin{table}[H]
\begin{tabular}{ccccccc}
$G$ & $H$ & $\lvert \Cl(\Z[G])\rvert$ & $\lvert \Cl(\Z[H]) \rvert$ &  $\lvert \Cl(\Z[H \times C_2]) \rvert$ &  $\lvert \Cl(\Gamma) \rvert$ \\ \toprule
$Q_8 \rtimes Q_{12}$ & $\GL_{2}(\F_{3}) $ & $2^9$ & $2$ & $2^7$ &  $2^3$  \\
$\tilde{T} \times Q_{12}$ & $G_{(144,127)}$ & $2^{28} \cdot 3^7$ & $2^5 \cdot 3^3$ & $2^{23} \cdot 3^7$ & $2^{10} \cdot 3^3$ \\
$\tilde{T} \times Q_{20}$ & $G_{(240,108)}$ &$2^{41} \cdot 3^3$ & $2^9 \cdot 3$ & $2^{34} \cdot 3^3$ & $2^{16} \cdot 3$ \\
$\tilde{I} \times C_{2}$ & $C_{2} \times A_{5}$ & $2^{10}$ & $2$ & $2^6$ & $2^5$ \\
\bottomrule
\end{tabular}
\end{table}

For $G=Q_{8} \rtimes \tilde{T}$, there is a normal subgroup $N$ of order $2$
such that $H:=G/N$ is isomorphic to $G_{(96,204)}$. 
It can easily be checked that $\Q[H]$ satisfies the Eichler condition.
Let $\Gamma = \Z[G] / \mathrm{Tr}_{N} \Z[G]$.
Unfortunately, Proposition~\ref{prop:ZG-SFC-iff-Gamma-SFC} cannot be applied in this situation
because the required equality \eqref{eq:quotients-class-groups} does not hold. 
However, $\Z[G]$ can be shown to have SFC if and only if $\Gamma$ has SFC by 
an application of 
Algorithm~\ref{alg:sfc-fiberproduct} with $\Lambda_{1}=\Z[H]$ and $\Lambda_{2}=\Gamma$.
Note that $\overline{\Lambda}=\F_{2}[H]$, and so an explicit embedding of $\overline{\Lambda}^{\times}$ into $\GL_{96}(\F_{2})$ can be determined as explained in 
Remark~\ref{rmk:embed-units-into-GLFp} (see also \S \ref{subsec:unit-groups-finite-rings}).
Crucially, this allows computations in $\overline{\Lambda}^{\times}$ to be performed
reasonably efficiently.

Therefore, for $G=Q_{8} \rtimes Q_{12}$, $Q_{8} \rtimes \tilde{T}$, 
$\tilde{T} \times Q_{12}$, $\tilde{T} \times Q_{20}$, or $\tilde{I} \times C_{2}$, we have shown that 
$\Z[G]$ has SFC if and only if $\Gamma$ has SFC.
This is highly advantageous from a computational point of view, since 
$\dim_{\Q} \Q \otimes_{\Z} \Gamma = \frac{1}{2}\dim_{\Q} \Q[G] = \frac{1}{2}|G|$.
To verify that $\Gamma$ has SFC, we again apply Algorithm~\ref{alg:sfc-fiberproduct}, 
this time with the fiber product induced by the decomposition
$\mathbb{Q} \otimes_{\mathbb{Z}} \Gamma = B_{1} \oplus B_{2}$, with $B_{1}$ 
the largest component of $\mathbb{Q} \otimes_{\mathbb{Z}} \Gamma$ 
satisfying the Eichler condition (equivalently, with $B_{2}$ the largest component that is the direct sum of totally definite quaternion algebras). 
Computations in $\overline{\Lambda}^{\times}$ are performed by first computing an explicit embedding into the symmetric
group $\mathfrak{S}_{d}$ for some $d \in \Z_{>0}$ (see \S \ref{subsec:unit-groups-finite-rings}).
This is less efficient than working with an embedding into $\GL_{n}(\F_{p})$
for appropriate choices of $n$ and $p$, but such an embedding does not exist in general.

For $G=G_{(192,183)}$ or $G_{(480,962)}$, there exists a normal subgroup $N$ of $G$ such that 
$H:=G/N$ is isomorphic to $Q_{8} \rtimes Q_{12}$ and so $\Z[H]$ has SFC by the result for $Q_{8} \rtimes Q_{12}$.
Let $\Gamma = \Z[G] / \mathrm{Tr}_{N} \Z[G]$.
It can easily be checked that $\Q \otimes_{\Z} \Gamma$ satisfies the
Eichler condition and thus $\Gamma$ has LFC by Theorem~\ref{thm:Jacobinski-cancellation}.
Then $\Z[G]$ can be shown to have SFC by an application of Algorithm~\ref{alg:sfc-fiberproduct}
with $\Lambda_{1}=\Gamma$ and $\Lambda_{2}=\Z[H]$. 
Note that $\overline{\Lambda}=\F_{p}[H]$ for $p=2$ or $5$, which allows 
computations in $\overline{\Lambda}^{\times}$ to be performed
reasonably efficiently, as explained above. 

For $G=G_{(384,580)}$ there exists a normal subgroup 
$N$ of $G$ such that 
$H:=G/N$ is isomorphic to $G_{(192,183)}$ and so $\Z[H]$ has SFC by the result for $G_{(192,183)}$.
Then the same method as in the previous paragraph with $\overline{\Lambda}=\F_{2}[H]$
shows that $\Z[G]$ has SFC.
\end{proof}

\begin{remark}
The computations involved in the proof of Theorem~\ref{thm:new-groups-with-SFC} 
rely on the algorithm of \cite[\S 3]{bley-boltje} to compute locally free class groups and 
on the new freeness testing methods of \S \ref{sec:hybrid-method}. 
These in turn rely on the computation of ray class groups and of unit groups of rings of integers of number fields.
Due to reasons of efficiency, state-of-the-art algorithms to perform these computations yield solutions that are provably correct under the assumption that the Generalized Riemann Hypothesis (GRH) holds.
In the proof of Theorem~\ref{thm:new-groups-with-SFC}, the number fields involved are the character fields of the irreducible complex characters of the finite groups $G$ listed in \eqref{eq:list-groups-have-SFC}:
\begin{align*}
\Q, \quad \Q(i), \quad \Q(\sqrt{-3}), \quad \Q(\sqrt 5), \quad \Q(\sqrt 2), \quad \Q(\sqrt{-2}), \\ \Q(\zeta_8), \quad \Q(\sqrt{-15}), \quad  \Q(\zeta_{12}), 
\quad \Q(\sqrt{-3}, \sqrt{5}), \quad \Q(\zeta_{16}).     
\end{align*}
As these fields are of small degree and discriminant, the necessary computations can be performed unconditionally (or can be found in the literature) and hence Theorem~\ref{thm:new-groups-with-SFC} does not depend on GRH.
\end{remark}

\begin{theorem}\label{thm:part-class-1023}
Let $G$ be a finite group with $|G| \leq 1023$.
Let 
\begin{multline*}
\Sigma =  
\{
3595, 3599, 3704, 3716, 3719, 3721, 3722, 4179, 4487, 4490, 
5069, 5324, 5331, 5332, \\ 5338, 5339, 5340, 5346, 5347, 5348, 5349, 
8602, 8604, 8605, 8610, 8611, 8612, 8614 \}. 
\end{multline*}
Suppose that $G$ is not isomorphic to any of the following groups:
\renewcommand{\labelenumi}{(\alph{enumi})}
\begin{enumerate}
\item $G_{(384,m)}$ for $m \in \{ 572, 577, 579, 5867, 18129, 18138, 18228 \}$,
\item $G_{(480,222)}$,
\item $G_{(576,m)}$ for $m \in \{ 1447, 2006 \}$, 
\item $G_{(720,417)}$,
\item $G_{(768,m+1080000)}$ for 
$m \in \Sigma$,
\item $G_{(864,m)}$ for $m \in \{ 690, 692, 717, 731, 732, 4053, 4123 \}$, or
\item $G_{(960,m)}$ for $m \in \{ 642, 805, 5752 \}$.
\end{enumerate}
\renewcommand{\labelenumi}{(\roman{enumi})}
Then $\Z[G]$ has SFC if and only if $G$ has no quotient isomorphic to
\begin{enumerate}
\item $Q_{4n} \times C_{2}$ for $2 \leq n \leq 5$,
\item $\tilde{T} \times C_{2}^{2}$, $\tilde{O} \times C_{2}$, $\tilde{I} \times C_{2}^{2}$,
\item $G_{(32,14)}$, $G_{(36,7)}$, $G_{(64,14)}$, $G_{(100,7)}$, or
\item $Q_{4n}$ for $6 \leq n \leq 255$.
\end{enumerate} 
\end{theorem}

\begin{proof}
We use the GAP \cite{GAP4} small groups database \cite{MR1935567}, which in particular, contains all groups $G$ with $|G| \leq 1023$.
For each such $G$, we have performed the following checks
using a computer program whose implementation is discussed in Appendix~\ref{appendix}.
\begin{itemize}
\item If $m_{\mathbb{H}}(G) \leq 1$ then $\Z[G]$ has SFC by Theorem~\ref{thm:Nicholson-C}.
\item If $G$ has a quotient isomorphic to any of the groups listed in (i), (ii), (iii), or (iv), 
then $\Z[G]$ fails SFC by Corollary~\ref{cor:group-ring-fail-SFC}.
\item If $G$ has a binary polyhedral quotient $H$ listed in Theorem~\ref{thm:bp-canc-group-rings} such that 
  $m_{\mathbb{H}}(G)=m_{\mathbb{H}}(H)$ then $\Z[G]$ has SFC by Theorem~\ref{thm:Nicholson-B}.
\item If $G$ has a quotient isomorphic to  $\tilde{T} \times C_{2}$ and 
$m_{\mathbb{H}}(G)=2$ then $\Z[G]$ has SFC by Corollary~\ref{cor:lifts-of-tildeTxC2}.
\item If $G$ is any group listed in \eqref{eq:list-groups-have-SFC} 
then $\Z[G]$ has SFC by Theorem~\ref{thm:new-groups-with-SFC}.
\item If $G \cong G_{(576,5128)} \cong \tilde{T} \times \tilde{T}$ then 
$\Z[G]$ has SFC by Theorem~\ref{thm:TnxIm}.
\item If $G \cong H \times K$ where $H$ is any group listed in \eqref{eq:list-groups-have-SFC} and $K$ is any of $C_{3}, C_{5}, C_{7}, C_{9}$ or $C_{3}^{2}$, then $\Z[G]$ has SFC by 
Theorems~\ref{thm:new-groups-with-SFC} and \ref{thm:direct-products}.
\end{itemize}
The program outputs the groups $G$ with $|G| \leq 1023$ for which
the question of whether $\Z[G]$ has SFC is not determined by any
of the above conditions. These groups are listed in (a)--(g) above. 
\end{proof}

The following result is Theorem~\ref{intro-thm:classification-383} from the introduction.

\begin{corollary}\label{cor:classification-383}
Let $G$ be a finite group with $|G| \leq 383$.  
Then $\Z[G]$ has SFC if and only if $G$ has no quotient of the form
\begin{enumerate}
\item $Q_{4n} \times C_{2}$ for $2 \leq n \leq 5$,
\item $\tilde{T} \times C_{2}^{2}$, $\tilde{O} \times C_{2}$, $\tilde{I} \times C_{2}^{2}$,
\item $G_{(32,14)}$, $G_{(36,7)}$, $G_{(64,14)}$, $G_{(100,7)}$, or
\item $Q_{4n}$ for $6 \leq n \leq 95$.
\end{enumerate} 
\end{corollary}

\appendix

\section{Implementation of code including data and timings}\label{appendix}

All code used in this article is available at \texttt{github.com/thofma/TestingSFC.jl/}. 
All timings below are in seconds and are for computations performed using 
a single core of  a server with an Intel Xeon 6226R 2.90GHz CPU and 384GB RAM.
The software used was \textsc{OSCAR} 1.4.1 \cite{Oscar,Oscar-book}
(which includes \textsc{Hecke} \cite{Fieker2017} and \textsc{GAP} \cite{GAP4})
and \textsc{Magma} V2.28-7 \cite{Magma}.

\subsection{Computations of locally free class groups}
Locally free class groups of orders can be computed using the algorithm of \cite[\S 3]{bley-boltje} implemented in~\textsc{OSCAR}. 
The timings of the computations used in the proofs of Theorems~\ref{thm:tildeT-times-C2-has-SFC} and \ref{thm:new-groups-with-SFC} are given in Table~\ref{tab:classgroups}.
Recall that in each case apart from that in the first row, 
we have $H=G/N$ and $\Gamma = \Z[G] / \mathrm{Tr}_{N} \Z[G]$,
where $N$ is a certain normal subgroup of $G$ of order $2$.

\begin{table}[h]
\caption{Locally free class group computations}
\vspace{1em}
\begin{tabular}{cccccc}
$G$ & $H$ &  $\Cl(\Z[G])$ & $\Cl(\Z[H])$ & $\Cl(\Z[H \times C_2])$ & 
$\Cl(\Gamma)$ \\ \toprule
-- & $\tilde{T}$ & -- & 3s & 3s & -- \\
$Q_8 \rtimes Q_{12}$ & $\mathrm{GL}_{2}(\F_{3})$ & 17s & 1s & 19s & 1s \\
$\tilde{T} \times Q_{12}$ & $G_{(144,127)}$ & 2012s & 112s & 2205s & 115s \\
$\tilde{T} \times Q_{20}$ & $G_{(240,108)}$ & 6446s & 604s & 9337s & 1479s \\
$\tilde{I} \times C_{2}$ & $C_{2} \times A_{5}$ & 769s & 68s & 1027s & 85s \\
\bottomrule
\label{tab:classgroups}
\end{tabular}
\end{table}

\subsection{The proof of Theorem~\ref{thm:new-group-rings-fail-SFC}}
In Table~\ref{tab:notsfc} we present details on the computations, which were performed using \textsc{OSCAR}.
Recall that for each group $G$ under consideration, we find a quotient $\Gamma$ of $\Z[G]$, for which failure of SFC is shown using Algorithm~\ref{alg:SFC-fail}.
We report the dimension of $\Q \Gamma \cong \Q \otimes_{\Z} \Gamma$ and the overall run time.

\begin{table}[h]
\caption{Proof of failure of SFC for $\Z[G]$ using Algorithm~\ref{alg:SFC-fail}}
\label{tab:notsfc}
\begin{tabular}{ccc}
    \toprule
    $G$ & $\dim(\Q \Gamma)$ & total time \\
    \midrule
    $Q_{8} \times C_2$       & 16 & 121s \\
    $Q_{12} \times C_2$      & 16 & 131s \\
    $Q_{16} \times C_2$      & 24 & 136s \\
    $Q_{20} \times C_2$      & 24 & 234s \\
    $\tilde{T} \times C_2^2$ & 28 & 210s \\
    $\tilde{O} \times C_2$   & 48 & 5740s \\
    $\tilde{I} \times C_2^2$ & 36 & 1443s \\
    $G_{(32, 14)}$             & 20 & 131s \\
    $G_{(36, 7)}$              & 20 & 133s \\
    $G_{(64, 14)}$             & 24 & 152s \\
    $G_{(100, 7)}$             & 24 & 124s \\
    \bottomrule
\end{tabular}
\end{table}

\subsection{The proof of Theorem~\ref{thm:new-groups-with-SFC}}

We present details on the computations, which were performed using both~\textsc{OSCAR} and~\textsc{Magma}.
Table~\ref{tab:reduction} contains information on the application of Algorithm~\ref{alg:sfc-fiberproduct} to the respective integral group ring 
$\Lambda=\Z[G]$.
Recall that in each case, 
we have $H=G/N$ and $\Gamma = \Z[G] / \mathrm{Tr}_{N} \Z[G]$,
where $N$ is a certain normal subgroup of $G$ of order $2$.
The column labeled `reduced to' refers to the order, for which the SFC property is shown to be equivalent to the SFC property of $\Z[G]$.
The columns grouped together under `(9)--(10)' refer to computations of $\overline{\Lambda}^{\times \mathrm{ab}}$ in Steps~(9)--(10) of Algorithm~\ref{alg:sfc-fiberproduct}.
(Recall that in the notation of Algorithm~\ref{alg:sfc-fiberproduct},
we have that $\overline{\Lambda}_{2} \cong \overline{\Lambda}$.)
We give a description of the finite ring $\overline{\Lambda}$, 
a list $[d_1^{n_1},\dotsc,d_r^{n_r}]$ such that $\overline{\Lambda}^{\times  \mathrm{ab}}$ is isomorphic to
\[
\Z/d_1^{n_1}\Z \times \dotsb \times \Z/d_r^{n_r}\Z,
\]
and the run time of these steps, respectively.
The columns grouped together under `(13)--(16)' refer to freeness checks of test lattices in Steps~(13)--(16) of Algorithm~\ref{alg:sfc-fiberproduct}.
We report the number $r$ of test lattices checked for freeness and the run time for these steps respectively.
Finally, the column labeled `total time' contains the total run time of Algorithm~\ref{alg:sfc-fiberproduct}.

Table~\ref{tab:gammasfc}  contains information on the proof of the SFC property for $\Gamma$ using Algorithm~\ref{alg:sfc-fiberproduct},
where $\Gamma$ is the order for which the SFC property is equivalent to the SFC property of $\Z[G]$.
The fiber product used is obtained from the `maximal Eichler splitting' of 
$\Q \Gamma \cong \Q \otimes_{\Z} \Gamma$ which is induced by the decomposition
$\mathbb{Q} \Gamma = B_{1} \oplus B_{2}$, with $B_{1}$ 
the largest component of $\mathbb{Q}\Gamma$ 
satisfying the Eichler condition.
The columns are similar to the columns of Table~\ref{tab:reduction}, except
that we now also include information on Step~(7), which checks whether the projection $\Lambda_2$ of $\Lambda$ onto $B_{2}$ has SFC.
We report both $\dim(\Q \Lambda_2) = \dim(B_2)$ and the run time of this step.

\begin{table}[h]
  \centering
  \caption{Reduction of the SFC property for $\Lambda = \Z[G]$ via Algorithm~\ref{alg:sfc-fiberproduct}}
  \label{tab:reduction}
  \begin{tabular}{ccccccccc}
    \toprule
     & & & \multicolumn{3}{c}{(9)--(10)} & \multicolumn{2}{c}{(13)--(16)} & \\
    \cmidrule(lr){4-6}
    \cmidrule(lr){7-8}
                         $G$ & $H$ & reduced to & $\overline\Lambda$ & $\overline{\Lambda}^{\times \mathrm{ab}}$ & time & $r$ & time & total time\\
    \midrule
    $Q_8 \rtimes \tilde{T}$ & $G_{(96,204)}$ & $\Gamma$ & $\F_2[H]$ & $[2^6, 12]$ & 625s & 2 & 362s & 1080s  \\
    $G_{(192,183)}$ & $Q_8 \rtimes Q_{12}$ & $\Z[H]$ & $\F_2[H]$ & $[2^{10}, 4, 8]$ & 686s & 4 & 810s & 1612s  \\
    $G_{(480,962)}$ & $Q_8 \rtimes Q_{12}$ & $\Z[H]$ & $\F_2[H]$ &  $[4^{12}, 24^2]$ & 636s &  11 & 67730s & 104619s \\
    $G_{(384,580)}$ & $G_{(192,183)}$ & $\Z[H]$ & $\F_2[H]$ & $[2^{17}, 4^4, 8^2]$ & 38750s & 10 & 107968s & 152376s \\
    \bottomrule
  \end{tabular}
\end{table}

\begin{table}[H]
  \centering
  \caption{Proof of SFC property for $\Lambda = \Gamma$ using Algorithm~\ref{alg:sfc-fiberproduct} and the `maximal Eichler splitting'}
  \label{tab:gammasfc}
  \begin{tabular}{cccccccccc}
    \toprule
    & & \multicolumn{2}{c}{(9)--(10)} & \multicolumn{2}{c}{(13)--(16)} & \multicolumn{2}{c}{(7)} & \\
    \cmidrule(lr){3-4}
    \cmidrule(lr){5-6}
    \cmidrule(lr){7-8}

                       $G$ & $\dim(\Q\Gamma)$ & $\overline{\Lambda}^{\times \mathrm{ab}}$ & time & $r$ & time & $\dim(\Q\Lambda_2)$ & time & total time\\
    \midrule
    $Q_8 \rtimes Q_{12}$ & 48 & $[2^2, 4, 8^2]$ & 51s & 5 & 9s &12 & 27s &  201s \\
    $\tilde{T} \times Q_{12}$ & 144 & $[2^6, 6^3, 24]$ & 16s & 4 & 380s & 12 & 27s & 532s\\
    $\tilde{T} \times Q_{20}$ & 240 & $[2^8, 4^4, 12, 14]$ & 10183s & 8 & 6800s & 12 & 255s & 19098s\\
    $\tilde{I} \times C_2$ & 120 & $[2^2, 4^2, 8, 24]$ & 28s & 5 & 210s & 16 & 88s& 485s \\
    $Q_8 \rtimes \tilde{T}$ & 96 & $[2^5, 12]$ & 5100s & 2 & 480s & 8 & 9s & 7055s \\
    \bottomrule
  \end{tabular}
\end{table}

\subsection{The proof of Theorem~\ref{thm:part-class-1023}}\label{subsec:comps-part-class-1023}
There are two implementations of the checks described in this proof. 
One is written in \textsc{Magma} and the other 
written in \textsc{OSCAR}.
For technical reasons, there are in fact two separate programs in each case:
one for $|G|=512$, and another for $|G| \leq 1023$ and $|G| \neq 512$.

\bibliography{SFC_WriteUp}{}
\bibliographystyle{amsalpha}

\end{document}